\documentclass[envcountsame,a4paper,10pt]{llncs}
\usepackage{latexsym,amsmath} 
\usepackage{times}
\usepackage{psfrag}
 
\usepackage[dvips]{epsfig} 
\graphicspath{{./Fig_eps/}{./Eps/}} 
\DeclareGraphicsExtensions{.ps,.eps} 
\usepackage{color} 
\usepackage{a4wide}

\newcommand\rk{r_{k\mathrm{-SAT}}}
\newcommand\sign{\mathrm{sign}}
\newcommand\id{\mathrm{id}}
\newcommand\BP{\mathrm{BP}} 
\newcommand\BPGD{Belief Propagation Guided Decimation} 
\newcommand\SPGD{Survey Propagation Guided Decimation}

\newcommand\PHI{\vec\Phi} 
\newcommand\PHIbin{\vec\Phi_{\mathrm{bin}}} 
\newcommand\PHIseq{\vec\Phi_{\mathrm{seq}}}

\newcommand\cutnorm[1]{\left\|{#1}\right\|_{\qed}} 
 
\newcommand\cB{\mathcal{B}} 
\newcommand\cC{\mathcal{C}} 
 
\newcommand\cF{\mathcal{F}} 
 
\newcommand\cE{\mathcal{E}} 
\newcommand\cU{\mathcal{U}} 
\newcommand\cN{\mathcal{N}} 
\newcommand\cQ{\mathcal{Q}} 
 
\newcommand\cS{\mathcal{S}} 
\newcommand\cT{\mathcal{T}}

\newcommand\cL{\mathcal{L}} 
\newcommand\cM{\mathcal{M}} 
\newcommand\cP{\mathcal{P}}

\newcommand\cV{\mathcal{V}} 
 
\newcommand\cZ{\mathcal{Z}}

\def\cC{{\mathcal C}}
\def\cE{{\cal E}}

\newcommand\eul{\mathrm{e}} 
\newcommand\eps{\varepsilon}

\newcommand\Erw{\mathrm{E}} 
\newcommand\pr{\mathrm{P}}

\newcommand{\vecone}{\vec{1}}

\newcommand{\Bin}{{\rm Bin}}

\newcommand{\bink}[2] {{{#1}\choose {#2}}}

\newcommand\ra{\rightarrow} 

\newcommand\bc[1]{\left({#1}\right)} 
\newcommand\cbc[1]{\left\{{#1}\right\}} 
\newcommand\bcfr[2]{\bc{\frac{#1}{#2}}} 
\newcommand{\bck}[1]{\left\langle{#1}\right\rangle} 
\newcommand\brk[1]{\left\lbrack{#1}\right\rbrack} 
\newcommand\scal[2]{\bck{{#1},{#2}}} 
\newcommand\norm[1]{\left\|{#1}\right\|} 
\newcommand\abs[1]{\left|{#1}\right|}

\newcommand\RR{\mathbf{R}}

\newcommand{\whp}{w.h.p.} 
\newcommand{\stacksign}[2]{{\stackrel{\mbox{\scriptsize #1}}{#2}}}

\newcommand\Lem{Lemma}
\newcommand\Prop{Proposition}
\newcommand\Thm{Theorem}
\newcommand\Cor{Corollary}
\newcommand\Sec{Section}

\newcommand\algstyle{\small\sffamily}

\begin{document} 
\spnewtheorem{algorithm}[theorem]{Algorithm}{\bfseries}{}
\spnewtheorem{experiment}[theorem]{Experiment}{\bfseries}{}
\spnewtheorem{fact}[theorem]{Fact}{\bfseries}{\itshape}
\spnewtheorem{hypothesis}[theorem]{Hypothesis}{\bfseries}{\itshape}

\title{On Belief Propagation Guided Decimation 
	 for Random $k$-SAT
	}
 
\author{
Amin Coja-Oghlan\thanks{Supported by EPSRC grant EP/G039070/2.
	An extended abstract version of this work appeared in the proceedings of the ACM-SIAM Symposium on Discrete Algorithms (`SODA') 2011.}
} 
\date{\today} 

\institute{Goethe University, Mathematics Institute, 10 Robert Mayer St, Frankfurt 60486, Germany\\
	\email{acoghlan@math.uni-frankfurt.de}} 

\maketitle 

\begin{abstract}
Let $\PHI$ be a uniformly distributed random $k$-SAT formula with $n$ variables and $m$ clauses.
Non-construc\-tive arguments show that $\PHI$ is satisfiable for clause/variable ratios $m/n\leq\rk\sim2^k\ln2$  with high probability.
Yet no efficient algorithm is know to find a satisfying assignment
beyond $m/n\sim 2^k\ln(k)/k$  with a non-vanishing probability.
On the basis of deep but non-rigorous statistical mechanics ideas,
	a message passing algorithm called {\em\BPGD}
has been put forward
	(M\'ezard, Parisi, Zecchina: Science 2002; Braunstein, M\'ezard, Zecchina: Random Struc.\ Alg.~2005).
Experiments suggested that the algorithm might succeed for densities very close to $\rk$
	for $k=3,4,5$ (Kroc, Sabharwal, Selman: SAC 2009).
Furnishing the first rigorous analysis of this algorithm on a non-trivial input distribution,
in the present paper we show that \BPGD\
fails to solve random $k$-SAT formulas already for $m/n=O(2^k/k)$,
almost a factor of $k$ below the satisfiability threshold $\rk$.
Indeed, the proof refutes  a key hypothesis on which \BPGD\ hinges for such $m/n$.
\end{abstract}

\section{Introduction and results}\label{Sec_Intro}

Let $k\geq3$ and $n>1$ be integers, let $r>0$ be a fixed real number (independent of $n$),
and set $m=\lceil rn\rceil$.
Let $\PHI=\PHI_k(n,m)$ be a propositional formula obtained
by choosing a set of $m$ clauses of length $k$ over the variables $x_1,\ldots,x_n$ uniformly at random
such that no variable occurs in the same clause more than once
	(either positively or negatively).
For $k,r$ fixed we say that $\PHI$ has some property $\cP$ \emph{with high probability} (`\whp') if
$\lim_{n\ra\infty}\pr\brk{\PHI\in\cP}=1$.

\subsection{Background and motivation}

Since the 1990s the random formula $\PHI$ has gained a reputation as an extremely challenging benchmark for SAT solving.
More precisely, early computer experiments led to two key hypotheses~\cite{Cheeseman,KirkpatrickSelman,MitchellSelmanLevesque}.
First, that there is a {\em sharp threshold} for satisfiability.
That is, for any clause length $k$ there is a threshold value $\rk>0$ such that 
the random formula $\PHI$ is satisfiable \whp\
if $r<\rk$, while $\PHI$ is unsatisfiable \whp\ if $r>\rk$.
Second, that standard SAT-solvers such as DPLL-based algorithms require an exponential time to find a satisfying assignment
for densities $r$ `close' to $\rk$.
Thus, while these algorithms are highly efficient on ``real-world'' SAT instances, the simplest conceivable model
of random formulas eludes them.
These two hypotheses have inspired a considerable amount of research over the years, both experimental and theoretical~\cite{AchHandbook}.
Moreover, similar phenomena have been hypothesised in many other random problems~\cite{ANP}.

While the precise values (and even the existence) of the $k$-SAT threshold remain unknown for any $k\geq3$, 
asymptotically tight upper and lower bounds have been established.
Indeed, non-constructive arguments show that $\PHI$ is satisfiable \whp\ if $r<2^k\ln2-\frac32\ln2-\eps_k$,
while $\PHI$ is unsatisfiable \whp\ if $r>2^k\ln2-\frac12(1+\ln2)+\eps_k$,
with $\eps_k\ra0$ for large $k$~\cite{yuval,Kosta,KKKS}.
Thus, the transition from satisfiable to unsatisfiable takes place at about $\rk\sim 2^k\ln2$.

With respect to the computational problem, in spite of two decades of extensive research in the CS community
no algorithm seemed capable of finding a satisfying assignment for densities $r$ anywhere close to $\rk$ in polynomial time 
with a non-vanishing probability. 
More precisely, the best rigorously analysed polynomial time algorithm, designed specifically to
``beat'' random formulas, is known to succeed for $r<(1-\eps_k)2^k\ln(k)/k$ \whp,
	and seems to fail beyond~\cite{BetterAlg}.
Furthermore, a plethora of algorithms are known to fail for asymptotically even smaller densities $r=\rho\cdot 2^k/k$
with $\rho>0$ an absolute constant (independent of $k$). 
Examples include simple linear-time algorithms such as Unit Clause
($\rho=\eul/2$)~\cite{ChaoFranco2} or Shortest Clause ($\rho=1.817$)~\cite{FrSu},
as well as a wide range of DPLL-type algorithms ($\rho=11/4$)~\cite{AchBeameMolloy}.
In summary, there remained a factor of about $k/\ln k$ between the satisfiability threshold and the density where
algorithms are known to find satisfying assignments efficiently.

Against this gloomy background, it came as a considerable surprise when experiments indicated that  certain highly efficient
\emph{message passing algorithms}
come 
within a whisker of the conjectured satisfiability threshold~\cite{BMZ,GomesSelman,Kroc,MPZ}. 
These algorithms, called \emph{\BPGD}
and \emph{\SPGD}, were put forward 
on the basis of the ``cavity method'', a very insightful but non-rigorous technique from statistical mechanics~\cite{BMZ,pnas}.%
	\footnote{The message passing procedure upon which \BPGD\ is based
		has been rediscovered several times in the context of different applications, see \Sec~\ref{Sec_related} for details.
		In the physics literature it was originally known under the name ``Bethe-Peierls approximation''.
		By contrast, the message passing technique that underpins Survey Propagation seems to be new.}
Conceptually, Belief/Survey Propagation Guided Decimation are 
more sophisticated than the previously studied algorithms 
by an order of magnitude;
we will give a detailed account in \Sec~\ref{Sec_BPdec}.
As a consequence, the techniques that were developed to analyze previous algorithms fail dramatically for Belief/Survey Propagation.

The performance of the new message passing algorithms can be exemplified nicely in the case $k=4$.
The conjectured threshold for the existence of satisfying assignments is $r_{4-\mathrm{SAT}}\approx 9.93$~\cite{Mertens}.
According to experiments from~\cite{Kroc}, Survey Propagation guided decimation finds satisfying assignments efficiently
for densities up to $r=9.73$.
Experiments from~\cite{RTS} suggest that the ``vanilla'' version of \BPGD\
 succeeds up to $r=9.05$.
With a certain tweak
(the ``most biased variable'' decimation rule)
\BPGD\ succeeds up to $r=9.24$~\cite{Kroc}.
By comparison, the
 best ``classical''
algorithm  {\tt SCB} from~\cite{FrSu}
finds satisfying assignments in polynomial time merely up to $r=5.54$,
while {\tt zChaff}, an industrial SAT solver, is effective up to $r=5.35$~\cite{Kroc}.

\subsection{Unsatisfied with physics}

Ever since these stunning experimental results were reported,
coming up with a rigorous analysis of the new message passing algorithms has been
one of the key challenges in the area of random constraint satisfaction problems (cf.~\cite{ANP}).
The present paper contributes the first such analysis.
More specifically, we study the ``vanilla'' version of \BPGD\ (`{\tt BPdec}'),
the simplest but arguably most natural version.
We establish a {\em negative} result: 
{\tt BPdec}
fails to find a satisfying assignment \whp\ for densities $r>\rho\cdot 2^k/k$ for a certain absolute constant $\rho>0$.
In other words, we prove that, perhaps surprisingly, {\tt BPdec} does {\em not} outperform simpler combinatorial algorithms such as the one from~\cite{BetterAlg}
	asymptotically.

Stating the result precisely requires a little care, because it involves two levels of randomness:
the choice of the random formula
$\PHI$, and the `coin tosses' of the randomized algorithm {\tt BPdec}.
For a (fixed, non-random) $k$-CNF $\Phi$
let $\mathrm{success}(\Phi)$ denote the probability that {\tt BPdec}$(\Phi)$ outputs a satisfying assignment.
Here, of course, `probability' refers to the coin tosses of the algorithm only.
Then, if we apply {\tt BPdec} to the {\em random} $k$-CNF $\PHI$, 
the success probability $\mathrm{success}(\PHI)$  becomes a random variable.
Recall that $\PHI$ is unsatisfiable for $r>2^k\ln2$ \whp

\begin{theorem}\label{Thm_main}
There is a constant $\rho_0>0$ such that for any $k,r$ satisfying
	\begin{equation}\label{eqmain}
	\rho_0\cdot2^k/k\leq r\leq2^k\ln2
	\end{equation}
we have $\mathrm{success}(\PHI)\leq\exp(-\Omega(n))$ \whp
\end{theorem}

\Thm~\ref{Thm_main} contrasts with the
very promising experimental results. 
The explanation for this is that the experiments were conducted for `small' $k=3,4,5$~\cite{Kroc,RTS}.
Indeed, already for $k=10$ large-scale experiments are difficult to carry out,
because the relevant density $r$ scales exponentially with $k$.
Thus, the good experimental performance can be attributed to the value of the constant $\rho_0$ in \Thm~\ref{Thm_main}.
Because the analysis is intricate as is, no attempt has been made to compute (or optimize) $\rho_0$.

Since Belief/Survey Propagation guided decimation were suggested~\cite{MPZ}, there have been various stabs at
explaining the performance of Belief/\SPGD\
by means of non-rigorous physics arguments~\cite{BMZ,pnas,RTS}.
We will review this work in more detail in \Sec~\ref{Sec_statMech} below, 
 but roughly speaking the predictions were as follows.
In chronological order,
\begin{enumerate}
\item[$\bullet$] the authors of~\cite{BMZ} opined that \BPGD\ fails for $r>(1+\eps_k)2^k\ln(k)/k$.
\item[$\bullet$] More optimistically, it was predicted in~\cite{pnas} that \BPGD\ will find satisfying assignments efficiently
		up to $r\sim2^k\ln2$.
\item[$\bullet$] Finally and most pessimistically, according to~\cite{RTS} \BPGD\ ought to fail for $r>\rho\cdot 2^k/k$ for an absolute
		constant $\rho>0$.
\end{enumerate}
All of these predictions derived from fairly sophisticated statistical mechanics reasoning, and both~\cite{pnas,RTS} quote experimental evidence,
	thereby (unintentionally) highlighting the need for a rigorous analysis.
\Thm~\ref{Thm_main} confirms the scenario put forward in~\cite{RTS}, but
does not sit well with the predictions from~\cite{pnas}.
Furthermore, the present analysis shows that 
 the reasoning from~\cite{BMZ}, where the demise of \BPGD\ was attributed to a certain change in the geometry of the
 set of satisfying assignments, is off the mark.

A potential objection to a negative result like \Thm~\ref{Thm_main} is that it might hinge on a small
detail of the algorithm that could easily be fixed.
However, in the sequel we will see  that
in the regime~(\ref{eqmain})
 our analysis
refutes a key hypothesis upon which {\tt BPdec} depends.
In other words, we show that {\tt BPdec} falls victim to a  conceptual issue, not a technicality.
Furthermore, some of the arguments used to prove \Thm~\ref{Thm_main} may be of independent interest as they can be expected to extend
to applications of BP beyond random $k$-SAT.
For instance, we develop a technique for tracing BP on certain quasi-random problem instances.

Finally, we point out that \Thm~\ref{Thm_main} has no immediate bearing on the potentially more powerful
Survery Propagation algorithm.
We will comment on Survey Propagation in \Sec~\ref{Sec_statMech} below.

\subsection{The {\tt BPdec} algorithm}\label{Sec_BPdec}
Fix a satisfiable $k$-CNF $\Phi$ on the variables $V=\cbc{x_1,\ldots,x_n}$.
We generally represent truth assignments as maps $\sigma:V\ra\cbc{-1,1}$, with $-1$ representing `false' and $1$ representing `true'.
	(It turns out that using $\pm1$ instead of the more common $0,1$ simplifies the description of BP quite a bit.)
Let $\cS(\Phi)$ denote the set of all satisfying assignments of $\Phi$.
The algorithm {\tt BPdec} is an attempt at implementing the following thought experiment.

\pagebreak

\noindent{
\begin{experiment}\label{Exp_dec}\upshape
\emph{Input:} A satisfiable $k$-CNF $\Phi$.\
\emph{Result:} An assignment $\sigma:V\rightarrow\cbc{-1,1}$.
\begin{tabbing}
mmm\=mm\=mm\=mm\=mm\=mm\=mm\=mm\=mm\=\kill
{\algstyle 0.}	\> \parbox[t]{40em}{\algstyle
	Let $\Phi_0=\Phi$.}\\
{\algstyle 1.}	\> \parbox[t]{40em}{\algstyle
	For $t=0,\ldots,n-1$ do}\\
{\algstyle 2.}	\> \> \parbox[t]{38em}{\algstyle
		Compute the fraction
			$$M_{x_{t+1}}(\Phi_{t})=\frac{\abs{\cbc{\sigma\in\cS(\Phi_t):\sigma(x_{t+1})=1}}}{|\cS(\Phi_t)|}$$
		of satisfying assignments of $\Phi_{t}$
				in which the variable $x_{t+1}$ takes the value $1$.}\\
{\algstyle 3.}	\> \> \parbox[t]{38em}{\algstyle
		Assign
				$$\sigma(x_{t+1})=\left\{\begin{array}{cl}
							1&\mbox{ with probability $M_{x_{t+1}}(\Phi_{t})$,}\\
							-1&\mbox{ with probability $1-M_{x_{t+1}}(\Phi_{t})$}.
							\end{array}\right.$$
							}\\
{\algstyle 4.}	\> \> \parbox[t]{38em}{\algstyle
		Obtain the formula $\Phi_{t+1}$ from $\Phi_{t}$ by substituting the value $\sigma(x_{t+1})$ for $x_{t+1}$ and simplifying, i.e.,
			\begin{enumerate}
			\item[$\bullet$] remove all clauses that got satisfied by setting $x_{t+1}$ to $\sigma(x_{t+1})$,
			\item[$\bullet$] omit $x_t$ from all the other clauses.
			\end{enumerate}
	}\\
{\algstyle 5.}	\> \parbox[t]{40em}{\algstyle
	Return the assignment $\sigma$.}
\end{tabbing}
\end{experiment}
}

\noindent
A moment's reflection reveals that the above experiment not only produces a satisfying assignment, but that
its (random) outcome is in fact uniformly distributed over the set $\cS(\Phi)$.
We observe that in the formulas $\Phi_t$ obtained at intermediate steps some clauses can (and typically will) have length less than $k$.

Referring to the successive assignments of variables
and the corresponding shrinking of the formula,
we call the above experiment the \emph{decimation process}.
The obvious obstacle to implementing it is the computation of the marginal probabilities $M_{x_{t+1}}(\Phi_{t})$.
Indeed, this task is $\#P$-hard on worst-case inputs.

Yet, under what conditions could we hope to compute (or approximate) the marginals $M_x(\Phi_{t})$?
Clearly, the marginals are influenced by `local' effects.
For instance, if $x$ occurs in a \emph{unit clause} $a$ of $\Phi_{t}$, i.e., a clause whose other $k-1$ variables have been assigned already
without satisfying $a$, then $x$ \emph{must} be assigned so as to satisfy $a$.
Hence, if $x$ appears in $a$ positively, then $M_x(\Phi_{t})=1$, and otherwise $M_x(\Phi_{t})=0$.
Similarly, if $x$ occurs \emph{only} positively in $\Phi_{t}$, then $M_x(\Phi_{t})\geq1/2$.
Furthermore, these local effects propagate:
 if $x$ appears in a clause $a$ whose other variables $y$ are subject to influences from other clauses $b_y\neq a$,
then the local effects operating on the variables $y$ may impact $x$ via $a$.
In the most extreme case, think of a variable $x$
that occurs in a clause $a$ whose other variables are all constrained by unit clauses 
to take values that fail to satisfy $a$.
Then $a$ effectively turns into a unit clause for $x$.

The key hypothesis underlying {\tt BPdec} is that in random formulas
such local effects \emph{determine} the marginals $M_x(\Phi_{t})$ asymptotically.
To define `local' precisely, we need a metric on the variables/clauses.
This metric is the shortest path distance on the \emph{factor graph} $G=G(\Phi_{t})$ of $\Phi_{t}$,
which is a bipartite graph whose vertices are the variables $V_t=\cbc{x_{t+1},\ldots,x_n}$ and the clauses of $\Phi_{t}$.
Each clause is adjacent to the variables that occur in it.
For an integer $\omega\geq1$ let $N^{\brk{\omega}}(x)$ signify the set of all vertices of $G$ that have distance at
most $2\omega$ from $x$.
Then the induced subgraph $G[N^{\brk\omega}(x)]$ corresponds to the sub-formula
of $\Phi_{t}^{\brk\omega}$ obtained by removing all clauses and variables at distance more than $2\omega$ from $x_t$.
Note that all vertices at distance precisely $2\omega$ are variables.
Hence, any satisfying assignment of
$\Phi$ induces a satisfying assignment of the sub-formula.
Let us denote by
	$$M_{x}^{\brk\omega}(\Phi_{t})=
			M_{x}(\Phi_{t}^{\brk\omega})=\frac{\abs{\cbc{\sigma\in\cS(\Phi_{t}^{\brk\omega}):\sigma(x)=1}}}{\abs{\cS(\Phi_{t}^{\brk\omega})}}$$
the marginal probability that $x_t$ takes the value $1$
in a random satisfying assignment of this sub-formula.

Of course, in the worst case the `local' marginals $M_{x}^{\brk\omega}(\Phi_{t})$ are
just as difficult to compute as the $M_{x}(\Phi_{t})$ themselves.
But  {\tt BPdec} employs an efficient heuristic called \emph{Belief Propagation} (`BP'),
which yields certain values $\mu_{x_t}^{\brk\omega}(\Phi_{t})\in\brk{0,1}$;
we will state this heuristic below.
If  $G[N^{\brk\omega}(x_t)]$ is a tree, then provably
$\mu_{x_t}^{\brk\omega}(\Phi_{t})=M_{x_t}^{\brk\omega}(\Phi_{t})$~\cite{BMZ}.
Moreover, standard arguments show that in a random formula $\PHI$ actually
$G[N^{\brk\omega}(x_t)]$ is a tree \whp\ so long as $\omega=o(\ln n)$.
More generally, in order to obtain an efficient algorithm it would be sufficient
for the BP outcomes $\mu_{x_t}^{\brk\omega}(\Phi_{t})$ to approximate the true overall marginals $M_{x_t}(\Phi_{t})$ well for \emph{some}
(say, polynomially computable, polynomially bounded) function $\omega=\omega(n)\geq1$.
This leads to the following hypothesis underpinning {\tt BPdec} (cf.~\cite{pnas}).

\begin{hypothesis}\label{Hyp_BP}
With probability $1-o(1)$ over the choice of $\PHI$ and
the random decisions in Experiment~\ref{Exp_dec}
the following holds for all $0\leq t< n$.
\begin{enumerate}
\item[i.] For any $\eps>0$ there is $\omega=\omega(\eps,k,r)$ such that
		$|M_{x_{t+1}}(\PHI_{t})-M_{x_{t+1}}^{\brk\omega}(\PHI_{t})|\leq\eps.$
\item[ii.] For any $\eps>0$ 
		there is $\omega=\omega(\eps,k,r)$ such that
		$|M_{x_{t+1}}(\PHI_{t})-\mu_{x_{t+1}}^{\brk\omega}(\PHI_{t})|\leq\eps.$
\end{enumerate}
\end{hypothesis}

Hypothesis~\ref{Hyp_BP} motivates the following algorithm~\cite{Allerton},
which is called \emph{\BPGD} because it combines BP (Step~2)
with a decimation step (Steps 3--4).

\noindent{
\begin{algorithm}\label{Alg_BPdec}\upshape\texttt{BPdec$(\Phi)$}\\\sloppy
\emph{Input:} A $k$-CNF $\Phi$ on $V=\cbc{x_1,\ldots,x_n}$.
\emph{Output:} An assignment $\sigma:V\rightarrow\cbc{-1,1}$.
\vspace{-2mm}
\begin{tabbing}
mmm\=mm\=mm\=mm\=mm\=mm\=mm\=mm\=mm\=\kill
{\algstyle 0.}	\> \parbox[t]{40em}{\algstyle
	Let $\Phi_0=\Phi$.}\\
{\algstyle 1.}	\> \parbox[t]{40em}{\algstyle
	For $t=0,\ldots,n-1$ do}\\
{\algstyle 2.}	\> \> \parbox[t]{38em}{\algstyle
		Use BP to compute $\mu_{x_{t+1}}^{\brk\omega}(\Phi_{t})$.}\\
{\algstyle 3.}	\> \> \parbox[t]{38em}{\algstyle
		Assign 
		$$\sigma(x_{t+1})=\left\{\begin{array}{cl}
							1&\mbox{ with probability $\mu_{x_{t+1}}^{\brk\omega}(\Phi_{t})$,}\\
							-1&\mbox{ with probability $1-\mu_{x_{t+1}}^{\brk\omega}(\Phi_{t})$}.
							\end{array}\right.$$
							}\\
{\algstyle 4.}	\> \> \parbox[t]{38em}{\algstyle
		Obtain the formula $\Phi_{t+1}$ from $\Phi_{t}$ by substituting the value $\sigma(x_{t+1})$ for $x_{t+1}$ and simplifying.}\\
{\algstyle 5.}	\> \parbox[t]{40em}{\algstyle
	Return the assignment $\sigma$.}
\end{tabbing}
\end{algorithm}}

\begin{remark}
The function $\omega=\omega(k,r,n)$ is ``hard-wired'' into the above algorithm, and our analysis does not
depend on any assumptions on $\omega$.
In particular, the statement of \Thm~\ref{Thm_main} is understood to hold for \emph{all} integer-valued functions $\omega=\omega(n)\geq0$.
\end{remark}

Although, strictly speaking, Hypothesis~\ref{Hyp_BP} provides neither a necessary nor a
sufficient condition for {\tt BPdec} to succeed on random $k$-CNFs \whp,
the hypothesis inspired the algorithm
	(we will get back to this in \Sec~\ref{Sec_statMech}).
Combining parts of the present analysis of the dynamics of the BP computation
	(more precisely, \Thm~\ref{Thm_dynamics} below)
with techniques for analyzing the geometry of the space of satisfying assignments,
 we proved the following in~\cite{Angelica}.

\begin{corollary}\label{Cor_main}
Both statements of Hypothesis~\ref{Hyp_BP} are false for $k,r$ satisfying~(\ref{eqmain}).
\end{corollary}

To complete the presentation of the algorithm,
we need to define Belief Propagation for $k$-SAT;
for a detailed derivation we point the reader to~\cite{BMZ,MM,Pearl}.
Ultimately, we need to define the value $\mu_{x_{t+1}}^{\brk\omega}(\Phi_{t})$ in Step~2 of {\tt BPdec}.

Let $N(v)$ denote the neighborhood of a vertex $v$ of the factor graph $G(\Phi_{t})$.
For a variable $x\in V_t$ and a clause $a\in N(x)$ we will denote the ordered pair $(x,a)$ by $x\ra a$.
Similarly, $a\ra x$ stands for the pair $(a,x)$.
Furthermore, we let $\sign(x,a)=1$ if $x$ occurs in $a$ positively, and  $\sign(x,a)=-1$ otherwise.

The {\bf\em message space} $\cM(\Phi_{t})$ is the set of all tuples
	$$(\mu_{x\ra a}(\zeta))_{x\in V_t,\,a\in N(x),\,\zeta\in\{-1,1\}}$$
such that $\mu_{x\ra a}(\pm1)\in\brk{0,1}$
and $\mu_{x\ra a}(-1)+\mu_{x\ra a}(1)=1$ for all $x,a,\zeta$.
For $\mu\in\cM(\Phi_t)$ we define 
	$\mu_{a\ra x}(\sign(x,a))=1$
and 
	\begin{equation}\label{eqBPai}
	\mu_{a\ra x}(-\sign(x,a))=
			1-\hspace{-4mm}\prod_{y\in N(a)\setminus\cbc x}\hspace{-4mm}
				\mu_{y\ra a}\bc{-\sign(y,a)}.
	\end{equation}
Furthermore, we define the {\bf\em Belief Propagation operator} $\BP$ 
as follows:
for any $\mu\in\cM(\Phi_{t})$ we define $\BP(\mu)\in\cM(\Phi_{t})$ by letting 
	\begin{eqnarray}
	(\BP(\mu))_{x\ra a}(\zeta)&=&\frac{\displaystyle\prod_{b\in N(x)\setminus\cbc a}\mu_{b\ra x}(\zeta)}
			{\displaystyle\prod_{b\in N(x)\setminus\cbc a}\mu_{b\ra x}(-1)+
			\prod_{b\in N(x)\setminus\cbc a}\mu_{b\ra x}(1)}
			\label{eqBPoperator}
	\end{eqnarray}
unless the denominator equals zero, in which case $(\BP(\mu))_{x\ra a}(\zeta)=\frac12$.

Finally, the $\mu_x^{\brk\omega}(\Phi_t)$ in steps~2--3 of {\tt BPdec} are defined as follows.
Let $\mu^{\brk0}=\frac12\cdot\vecone\in\cM(\Phi_{t})$ be the vector with all entries equal to $\frac12$.
Moreover, define inductively $\mu^{\brk{\ell+1}}=\BP(\mu^{\brk{\ell}})$ for $0\leq\ell<\omega$.
Then
	\begin{eqnarray}\nonumber
	\mu_x^{\brk\omega}(\Phi_{t})&=&\frac{\displaystyle\prod_{b\in N(x)}\mu_{b\ra x}^{\brk\omega}(1)}
			{\displaystyle\prod_{b\in N(x)}\mu_{b\ra x}^{\brk\omega}(-1)+
			\prod_{b\in N(x)}\mu_{b\ra x}^{\brk{\omega}}(1)}
		\label{eqBPmarginal}
	\end{eqnarray}
for any $x\in V_t$,
unless the denominator is zero, in which case we set $\mu_x^{\brk\omega}(\Phi_{t})=\frac12$.

The intuition here is that the $\mu_{x\ra a}(\zeta)$ are `messages' from a variable $x$ to the clauses $a$ in which
$x$ occurs, indicating how likely $x$ were to take the value $\zeta$ if clause $a$ were removed from the formula.
Based on these, (\ref{eqBPai}) yields messages $\mu_{a\ra x}(\zeta)$ from clauses $a$ to variables $x$, indicating
the probability that $a$ is satisfied if $x$ takes the value $\zeta$ and all other variables $y\in N(a)\setminus\cbc x$ are assigned
independently with probability $\mu_{y\ra a}(\pm1)$.
The BP operator~(\ref{eqBPoperator}) then uses the messages $\mu_{a\ra x}$ in order to `update' the messages
from variables to clauses.
More precisely, for each $x$ and $a\in N(x)$ the new messages $(\BP(\mu))_{x\ra a}(\zeta)$ are computed
under the hypothesis that all other clauses $b\in N(x)\setminus\cbc a$ are satisfied with probabilities $\mu_{b\ra x}(\zeta)$ independently
if $x$ takes the value $\zeta$.
Finally, the difference between~(\ref{eqBPoperator}) and~(\ref{eqBPmarginal}) is that the latter product runs over \emph{all} clauses $b\in N(x)$.
An inductive proof shows that, if for a variable $x$ the subgraph $G\brk{N^{\brk\omega}(x)}$ of the factor graph is acyclic, then
in fact $\mu_x^{\brk\omega}(\Phi_t)=M_x^{\brk\omega}(\Phi_t)$~\cite{BMZ}.

\subsubsection{Variations of the algorithm.}
{\tt BPdec} could be called the ``vanilla'' version of \BPGD.
It is the simplest but arguably the most natural variant.
Nonetheless, several other installments have been suggested and experimented with.
They differ in how the number $\omega$ of iterations is chosen and how exactly the result
of the Belief Propagation calculation is used to decimate.

In the ``vanilla'' variant we used an {\em a priori} number $\omega$ of iterations.
An alternative idea is to iterate the Belief Propagation operator until it reaches a fixed point.
More precisely, to accommodate numerical inaccuracies one could stop after $\omega$ iterations, with $\omega\geq1$ the least integer such that
for some small $\eps>0$ we have
	\begin{equation}\label{eqBPconv}
	\max_{x\ra a}|\mu^{\brk\omega}_{x\ra a}(1)-\mu^{\brk{\omega-1}}_{x\ra a}(1)|<\eps,
	\end{equation}
where the maximum is taken over all edges of the factor graph (e.g., \cite{BMZ,MM}).
Unfortunately, it is not generally assured that the convergence criterion~(\ref{eqBPconv}) will ever be met. 
Hence, one would need to specify how to proceed otherwise.
For instance, one could 
specify an a priori maximum number of iterations.
Our analysis can be adapted easily to accommodate these modifications (details omitted).

More importantly, one could come up with a more sophisticated {\em decimation strategy}, i.e.,
a different way of using the BP result $\mu^{\brk\omega}$ to choose the variable to be assigned next and its value.
In {\tt BPdec} we went for the ``vanilla rule'':
	the variables are assigned in the natural order, and each time the assignment is performed randomly based on the BP estimate
	of the marginal.

But in experiments a more common decimation strategy is the ``most biased variable'' rule:
	at each time choose a variable $x\in V_t$ that maximizes the ``bias'' $|\mu_x^{\brk\omega}(\Phi_t)-\frac12|$,
	and assign it randomly based on the BP estimate.
Experimentally the most biased variable rule allows for slightly better results than the vanilla rule.
For instance, in random $4$-SAT, experiments indicate that the former succeeds up to $m/n=9.24$,
and the latter up to $m/n=9.05$ \cite{Kroc,RTS}.

The statistical mechanics ideas that underpin Belief Propagation guided decimation do not 
endorse a preference for the ``most biased variable'' rule over the ``vanilla'' strategy.
But a heuristic argument in favor of ``most biased variable''
is that it might reduce the effect of numerical errors building up~\cite{Allerton}.
The present analysis does not seem to extend to ``most biased variable'' in a straightforward manner.
Thus, analyzing it remains an interesting open problem.

\subsubsection{Comparison with combinatorial algorithms.}
The difference between the previously studied combinatorial algorithms for random $k$-SAT and Belief Propagation can
be explained nicely in terms of the factor graph.
Indeed, in order to decide upon the value of a variable the previous algorithms only took the clauses and variables
	at distance two~\cite{AchSorkin,PureLiteral,ChaoFranco2,mick,FrSu,HajiSorkin,KKL} or four~\cite{BetterAlg} into consideration.
Based on this information, the variable is assigned following some simple combinatorial rule.

{\tt BPdec} can be viewed as a systematic way of making a ``less shortsighted'' decision.
The algorithm takes into account clauses/variables at distance up to $2\omega$,
	where $\omega$ may be a function that grows with $n$.
Indeed, the idea of determining the marginal $M_{x_{t+1}}^{\brk\omega}(\PHI_t)$
yields a meaningful way of incorporating the data from all these clauses/variables.
In particular, {\tt BPdec} implicitly implements many of the rules that are used in the combinatorial algorithms
	(e.g., the ``Unit Clause'' rule).
In this sense, {\tt BPdec} can be seen as a clever generalization of many of these combinatorial algorithms.
However, this also means that the techniques used in the previous analyses of combinatorial algorithms
are insufficient to tackle  {\tt BPdec}.

\subsection{The statistical physics perspective}\label{Sec_statMech}

\subsubsection{Clustering and correlation decay.}
Closely following the non-rigorous paper~\cite{pnas}, we discuss in this section  the statistical mechanics motivation 
for  {\tt BPdec}.
This will provide the basis for the discussion of the non-rigorous predictions as to the algorithm's performance.

According to the physicists' ``cavity method'', 
the random formula $\PHI$ undergoes several further
phase transitions prior to the satisfiability threshold. 
These phase transitions affect the correlations 
between the truth values that can be assigned to different variables.
Thus, fix a variable $x$ and let $\omega=\omega(n)=o(\ln n)$ be a function that tends to infinity slowly, say $\omega=\lceil\ln\ln n\rceil$.
Furthermore, let $\cB$ be the set of all variables at distance exactly $2\omega$ from $x$ in the factor graph.
How do the values assigned to variables on the ``far away boundary'' $\cB$ affect the truth value of $x$?

The strongest possible decay of correlations occurs when the boundary $\cB$ has no impact on $x$ at all.
To formalize this, let $\tau:\ra\cbc{-1,1}$ be a satisfying assignment of $\PHI$
and let $M^{\brk\omega}_{x}(\PHI,\tau)$ be the fraction of all satisfying assignments of $\PHI$ that set $x$ to true and
that coincide with $\tau$ on $\cB$. 
In symbols,
	$$M^{\brk\omega}_{x}(\PHI,\tau)=\frac{\abs{\cbc{\sigma\in\cS(\PHI):\sigma(x)=1,\,\sigma(y)=\tau(y)\mbox{ for all }y\in\cB}}}
			{\abs{\cbc{\sigma\in\cS(\PHI):\sigma(y)=\tau(y)\mbox{ for all }y\in\cB}}}.$$
Also recall that $M_x(\PHI)$ denotes the marginal probability that $x$ takes the value ``true'' in a random satisfying assignment of $\PHI$
	(without any boundary condition).
The  \emph{Gibbs uniqueness} condition requires that
	\begin{equation}\label{eqGibbsUnique}
	\max_{\tau\in\cS(\PHI)}\abs{M^{\brk\omega}_{x}(\PHI,\tau)-M_x(\PHI)}=o(1).
	\end{equation}
In words, fixing the ``far away'' variables does not make it noticeably more or less like
for $x$ to take the value ``true''. 
Hence, the marginal $M_x(\PHI)$ is governed entirely by the effects of variables at distance less than $2\omega$ from $x$,
	i.e., by the \emph{local} structure of the formula.

Consequently, it seems reasonable to expect that Belief Propagation yields the correct marginals so long as~(\ref{eqGibbsUnique}) holds.
It is known rigorously that \whp~(\ref{eqGibbsUnique}) holds up to $r\sim r_u= 2\ln k/k$
(a function that tends to zero for large $k$), and that Belief Propagation does indeed yield the correct marginals for such densities~\cite{MontanariShah}.
That is,  for $r<r_u$ \whp\
	\begin{equation}\label{eqBPworks}
	|\mu^{\brk\omega}_x(\PHI)-M_x(\PHI)|=o(1)\qquad\mbox{ for any $x\in V$}.
	\end{equation}

To define the second correlation decay property,  let us denote by $\vec\tau$ a uniformly random
element of $\cS(\PHI)$.
Then the \emph{non-reconstruction} condition is that
	\begin{equation}\label{eqnonreconstruction}
	\Erw_{\vec\tau}\abs{M^{\brk\omega}_{x}(\PHI,\vec\tau)-M_x(\PHI)}=o(1).
	\end{equation}
Hence, fixing the far away boundary to a ``typical'' satisfying assignment has no discernible effect on $x$.
Neglecting a $o(1)$-fraction of ``atypical'' cases $\vec\tau$,
one might still expect~(\ref{eqBPworks}) to hold so long as~(\ref{eqnonreconstruction}) is satisfied.
However, this conjecture awaits a rigorous proof.
According to the cavity method, (\ref{eqnonreconstruction}) holds up to $r\sim r_d=2^k\ln(k)/k$.
Moreover, the best rigorously analyzed algorithm (which is based on local search) succeeds in finding a satisfying assignment in polynomial time
right up to $r\sim r_d$ \whp~\cite{BetterAlg}.

To state the third property, 
let us denote the joint distribution of the truth values of the variables $\cB$ under a random
satisfying assignment by $M_{\cB}$.
Thus, $M_{\cB}$ is a probability distribution over $\cbc{-1,1}^{{\cB}}$.
Then the \emph{replica symmetry} condition requires that
the truth values of the variables in $\cB$ are asymptotically independent.
Formally,
	\begin{equation}\label{eqreconstruction}
	\abs{M_\cB(\tau)-\hspace{-3mm}\prod_{x\in \cB:\tau(x)=1}\hspace{-3mm}M_{x}(\PHI)\cdot\hspace{-3mm}
		\prod_{x\in\cB:\tau(x)=-1}\hspace{-3mm}(1-M_{ x}(\PHI))}=o(1)\qquad\mbox{for any $\tau\in \cbc{-1,1}^{{\cB}}$}.
	\end{equation}
It has duly been conjectured in~\cite{pnas} that~(\ref{eqreconstruction}) suffices to obtain~(\ref{eqBPworks}), i.e., to ensure that Belief Propagation
yields the correct marginals on $\PHI$.
The cavity method predicts that~(\ref{eqreconstruction}) holds for $$r\leq r_c=2^k\ln2-3\ln2+O(2^{-k}),$$
	while the conjectured  satisfiability threshold~\cite{Mertens} is
	\begin{equation}\label{eqthreshold}
	\rk=2^k\ln2-\frac{1+\ln2}2+O(2^{-k})\approx r_c+0.19.
	\end{equation}
The best current rigorous lower bound on $\rk$ matches $r_c$~\cite{Kosta}.

The densities $r_d,r_c$ are also conjectured to mark a change in the geometry of the set $\cS(\PHI)$ of satisfying assignments.
Let us turn $\cS(\PHI)$ into a graph by considering $\sigma,\tau\in\cS(\PHI)$ adjacent if their Hamming distance is equal to one.
While for densities $r<r_d$ the graph $\cS(\PHI)$ is conjectured to be (essentially) connected, for $r_d<r<r_c$ it shatters into an exponential
number of tiny connected components \whp\
More precisely, $\cS(\PHI)$ admits a decomposition
	\begin{equation}\label{eqclustering}
	\cS(\PHI)=\bigcup_{i=1}^N\cC_i
	\end{equation}
into ``clusters'' $\cC_i$ such that $|\cC_i|\leq\exp(-\Omega(n))|\cS(\PHI)|$ for all $i$ and such that
any two satisfying assignments in different clusters have Hamming distance $\Omega(n)$.
This decomposition was 
established rigorously in~\cite{AchACO}.
Intuitively, the cluster decomposition explains why~(\ref{eqnonreconstruction}) fails to hold for $r>r_d$:
the conditional marginal $M_{x}^{\brk\omega}(\PHI,\tau)$ 
corresponds to the marginal of $x$ \emph{within the cluster of $\tau$}, in contrast to the marginal $M_x(\PHI)$ over the {\em entire} set of satisfying assignments.

Further,  for $r_c<r<\rk$ the set of satisfying assignments still decomposes into exponentially many well-separated clusters \whp\
But now a bounded number of clusters are conjectured to dominate.
That is, if we order the clusters by size $\abs{\cC_1}\geq\cdots\geq\abs{\cC_N}$, then for a bounded number $\gamma=O(1)$
we have
	\begin{equation}\label{eqcondensation}
	\abs{\cS(\PHI)}\sim\abs{\cC_1\cup\cdots\cup\cC_\gamma}.
	\end{equation}
This structure 
goes by the name of \emph{condensation} in physics.
The values that different variables take within each cluster $\cC_1,\ldots,\cC_\gamma$ are conjectured
to be heavily dependent.
Furthermore, (\ref{eqcondensation}) implies that $M_{\cB}$ is but a convex combination of
a small (bounded) number of such intra-cluster distributions. 
Hence, 
the ``condensed'' geometry~(\ref{eqcondensation}) appears to be irreconcilable with the factorization property~(\ref{eqreconstruction}).

\subsubsection{Belief Propagation.}

Based on this ``static'' picture, three different hypotheses have been put forward as to the likely performance
of Belief Propagation guided decimation.
Most optimistically,  the authors of~\cite{pnas} argue that Belief Propagation guided decimation ought to
find satisfying assignments efficiently for densities right up to $r_c$.
Their prediction derives from the opinion that~(\ref{eqreconstruction}) should be sufficient to obtain~(\ref{eqBPworks}),
and that~(\ref{eqBPworks}) is the key to the success of Belief Propagation Guided Decimation.
Specifically, \cite{pnas} refers to the the ``most biased variable'' variant.
However, the precise decimation strategy is irrelevant to their considerations, which are in effect 
at odds with \Thm~\ref{Thm_main}.

A second prediction is that \BPGD\ should fail to find satisfying assignments for $r>r_d$~\cite{BMZ}. 
This conjecture is based on the hunch that the decomposition of $\cS(\PHI)$ into
``clusters'' and the ensuing demise of~(\ref{eqnonreconstruction}) 
cause~(\ref{eqBPworks}) to fail.
Agreeing with~\cite{pnas}, the authors appear to view (\ref{eqBPworks}) as the key to the performance of {\tt BPdec}.

According to the third prediction~\cite{RTS}, {\tt BPdec} fails for densities $r>\rho_0\cdot 2^k/k$, with $\rho_0>0$ an absolute constant (independent of $k$).
This prediction is based on a non-rigorous analysis of the decimation process, i.e., the idealized thought experiment that {\tt BPdec} strives
to implement (Experiment~\ref{Exp_dec}).
Crucially, the authors of~\cite{RTS} realize that (\ref{eqBPworks}) does {\em not} guarantee the success of {\tt BPdec}.

Instead, their analysis indicates that as the decimation process proceeds to assign variables,
the remaining unassigned variables are bound by clauses that become shorter and shorter.
In effect, the clauses become more and more difficult to satisfy, and thus the remaining set of satisfying assignments shrinks rapidly.
In other words, successive decimation of variables has a similar effect as increasing the density of the formula.
Consequently, after a number $t$ of decimations {\tt BPdec} may wind up with a formula $\PHI_{t}$
that violates~(\ref{eqreconstruction}) and thus~(\ref{eqBPworks}), even though the initial formula $\PHI$ may well have satisfied those conditions.
The contribution~\cite{RTS} supersedes an earlier attempt at studying the effect of decimation~\cite{Allerton}.

\Thm~\ref{Thm_main} is in agreement with the prediction from~\cite{RTS}.
But an important advantage of the present work over the (non-rigorous) contribution~\cite{RTS} is that
here we manage to analyze the {\em actual algorithm} {\tt BPdec}.
By contrast, \cite{RTS} only deals with the decimation process
	(i.e., Experiment~\ref{Exp_dec}, the idealized experiment that assumes knowledge of the precise marginals).
That is,
going significantly beyond the ambition of~\cite{RTS},
here we develop a technique for explicitly analyzing the dynamics of the message passing procedure.

In summary, the predictions in~\cite{BMZ,pnas} as to the performance of Belief Propagation are inaccurate
because they ignore the effect of decimation.
By contrast, as conjectured in~\cite{RTS} and proved here, in actuality  {\tt BPdec} 
gets itself into trouble by assigning
and decimating one variable after the other.
Thus, computing the correct marginals in the original formula $\PHI$ is one thing,
	but continuing to do so as decimation proceeds is quite another.%
\footnote{%
	Let us mention, as a cautionary tale, that both~\cite{pnas,RTS} quote experimental evidence
	to support their claims. This illustrates the difficulty of producing reliable experimental results
	on large random CSPs, and thus the need for rigorous results.}

\subsubsection{Survey Propagation.}
Let us briefly comment on Survey Propagation guided decimation,
the physicists' flagship algorithm~\cite{BMZ,MPZ}.
It is based on the idea of working with a different probability distribution.
Namely, instead of the uniform probability over satisfying assignments, 
Survey Propagation aims for the uniform distribution over the \emph{clusters} $\cC_i$ in the 
decomposition~(\ref{eqclustering}).
These clusters can be encoded as generalized assignments $\tau:V\ra\cbc{-1,0,1}$,
with $\tau(x)=\pm1$ indicating that variable $x$ takes the value $\pm1$ in \emph{all} the assignments
in $\cC_i$, and $\tau(x)=0$ indicating that $x$ can take either value~\cite{MMW,MM}.
Survey Propagation guided decimation combines a message passing algorithm for approximating the marginals of these generalized assignments
with a decimation procedure (see~\cite{BMZ} for details).

According to the cavity method, the Survey Propagation distribution enjoys a factorization property akin to~(\ref{eqreconstruction})
for densities $r$ right up to $\rk$.
In fact, the physicists' computation of the conjectured $\rk$ depends on this assumption~\cite{Mertens,MM}.
Furthermore, the (conjectured) factorization property nurtured hopes that Survey Propagation
may perform well for densities ``close to'' $\rk$~\cite{BMZ,pnas,MPZ}.

Given the above discussion of Belief Propagation, the obvious problem with this forecast is that it ignores the effect of decimation.
More precisely, it might well be that in the undecimated random formula $\PHI$ the Survey Propagation distribution
factorizes for densities right up to $\rk$.
But this might not be the case in a formula $\PHI_t$ where some variables have been decimated.
Since the arguments of~\cite{RTS} do not seem to extend to Survey Propagation,
	there is currently not even a non-rigorous study of the effect of decimation in this case.
Thus, while experiments consistently indicate that Survey Propagation guided decimation outperforms {\tt BPdec}, 
it is unclear how much so.
Guided by the present analysis of Belief Propagation, I venture to pose

\begin{conjecture}\label{Con_SP}
There is an absolute 
constant $\rho_1>0$ such that the
Survey Propagation Guided Decimation algorithm as stated in~\cite{BMZ} fails to find a satisfying assignment of $\PHI$ \whp\
for $r>\rho_1\cdot 2^k/k$.
\end{conjecture}

\noindent
If true, Conjecture~\ref{Con_SP} would imply that \SPGD\ is inferior to conceptually much simpler local-search algorithms (such as~\cite{BetterAlg}),
at least for large clause lengths $k$.

\subsection{Further related work}\label{Sec_related}

In full generality, Belief Propagation is a generic technique for computing the marginals of a probability distribution described by
an ``acyclic graphical model''~\cite{Pearl}.
But special instantiations of Belief Propagation have been (re)discovered several times for several applications.
Examples include statistical inference~\cite{Pearl}, coding theory~\cite{Gallager} and statistical mechanics~\cite{Bethe}, where the method
	is also referred to as ``Bethe-Peierls approximation''.
For a coherent discussion see~\cite{MM} and the references therein.

In spite of BP's practical success (and popularity), rigorous analyses of the algorithm are scarce.
A few exist in the context of LDPC decoding (e.g., \cite{MU,RSU01}).
We also analyzed BP for graph $3$-coloring~\cite{BP3col} on a certain class of expander graphs.
A further related result deals with the conceptually \emph{much} simpler Warning Propagation algorithm on certain random 3-CNFs
(``planted model'')~\cite{WP}.
In the random $k$-XORSAT problem (random linear equations mod $2$),
Belief Propagation reduces to Warning Propagation due to the algebraic nature of the problem and can thus be analyzed easily~\cite{MM}.
Furthermore, there has been some recent progress on analyzing certain variants of BP (such as the ``max-product algorithm'')
for certain optimization problems that are polynomial-time solvable in the worst case (e.g.,~\cite{BSS,GSW}).

A distinctive feature of \BPGD\ in comparison to earlier algorithmic applications of Belief Propagation
is the decimation step
(that the algorithm assigns one variable at a time and reruns Belief Propagation on the reduced formula).
In terms of analyzing the algorithm, decimation poses a substantial challenge.
The present paper furnishes the first analysis of this kind of algorithm on a non-trivial type of instances.
By contrast, previous analyses deal with algorithms that use BP  in a `one shot' fashion, i.e.,
the supposed marginals obtained via BP are used directly to assign all variables \emph{at once}~\cite{BP3col,MU,RSU01}.
Roughly speaking, this approach seems to work best if the problem instances are somewhat over-constrained so that
there is (essentially) a unique solution.
By contrast,
as we saw in \Sec~\ref{Sec_statMech}
for $r<\rk$ the random formula $\PHI$ \whp\ has \emph{exponentially} many satisfying
assignments, whose typical pairwise distance is close to $\frac n2$ \whp\
Furthermore, as we saw in \Sec~\ref{Sec_statMech} it is the decimation step that precipitates the
demise of \BPGD.

In the context of random constraint satisfaction problems,
Belief Propagation works out to be a special (the ``replica symmetric'') case of a larger statistical mechanics framework called the \emph{cavity method}~\cite{MM}.
The cavity method provides a toolbox for deriving highly non-trivial exact conjectures on various phase transitions in random CSPs.
The conjectured value~(\ref{eqthreshold}) of the $\rk$ is an example, but the method is quite general and has been applied
to a host of further CSPs as well.

The study of the BP marginals on the undecimated random formula $\PHI$ is somewhat related to the so-called
\emph{reconstruction problem}.
This problem has been studied on `symmetric' random CSPs, which include problems such as
(hyper)graph coloring~\cite{Prasad}, but not $k$-SAT.
The proofs in~\cite{Prasad} are based on indirect arguments (related to the second moment method),
which do not seem to extend to an analysis of {\tt BPdec}.

\subsection{Preliminaries and notation}\label{Sec_pre}

In this section we collect a few well-known results and introduce a bit of notation.
First of all, we note for later reference a well-known estimate of the expected number of satisfying assignments (see, e.g., \cite{yuval} for a derivation).

\begin{lemma}\label{Lemma_expected}
We have
	$\Erw|\cS(\PHI)|=\Theta(2^n(1-2^{-k})^m)\leq2^n\exp(-rn/2^k)$.
\end{lemma}

Furthermore, we are going to need the following Chernoff bound on the tails of a binomially distributed random variable or, more generally,
a sum of independent Bernoulli trials~\cite[p.~21]{JLR}.

\begin{lemma}\label{Lemma_Chernoff}
Let $X$ be a sum of independent Bernoulli variables with mean $\mu>0$.
Let $$\varphi(x)=(1+x)\ln(1+x)-x.$$
Then for any $t>0$,
	\begin{eqnarray*}
	\pr\brk{X>\mu+t}\leq\exp(-\mu\cdot\varphi(t/\mu)),&&
	\pr\brk{X<\mu-t}\leq\exp(-\mu\cdot\varphi(-t/\mu)).
	\end{eqnarray*}
In particular, for any $t>1$ we have
	$\pr\brk{X>t\mu}\leq\exp\brk{-t\mu\ln(t/\eul)}.$
\end{lemma}

For a real $b\times a$ matrix $\Lambda$ 
let 
	$$\cutnorm{\Lambda}=\max_{\zeta\in\RR^{V_a}\setminus\cbc0}\frac{\norm{\Lambda\zeta}_1}{\norm{\zeta}_\infty}.$$
Thus, $\cutnorm{\Lambda}$ is the  norm of $\Lambda$ viewed as an operator from $\RR^a$ equipped with the $L^\infty$-norm
to $\RR^b$ endowed with the $L^1$-norm.
For a set $A\subset\brk a=\cbc{1,\ldots,a}$ we let $\vecone_A\in\{0,1\}^{a}$ denote the indicator vector of $A$.
The following well-known fact about the norm $\cutnorm\cdot$ of matrices with diagonal entries equal to zero is going to come in handy.
	
\begin{fact}\label{Lemma_cutnorm}
For a real $b\times a$ matrix  $\Lambda$ with zeros on the diagonal we have
	$$\cutnorm{\Lambda}\leq 24\max_{A\subset\brk a,B\subset \brk b:A\cap B=\emptyset}\abs{\scal{\Lambda\vecone_A}{\vecone_B}}.$$
\end{fact}

\noindent
Finally, throughout the paper we let $S_n$ denote the set of permutations of $\brk n$.

\section{The probabilistic framework for analyzing {\tt BPdec}}\label{Sec_framework}

\subsection{Outline}

The single most important technique for analyzing algorithms on the random input $\PHI$ is the
``method of deferred decisions''.
Where it applies, the dynamics of the algorithm can typically be traced tightly via differential equations,
martingales, or Markov chains.
Virtually all of the previous analyses of algorithms for random $k$-SAT are based on this
	approach~\cite{AchSorkin,PureLiteral,ChaoFranco2,mick,BetterAlg,FrSu,HajiSorkin,KKL}.
Unfortunately, the `deferred decisions' technique is limited to very simple, `shortsighted' algorithms
that decide upon the value of a variable $x$ on the basis of the clauses/variables at distance, say, one or two
from $x$ in the factor graph~\cite{AchSorkin}.
By contrast, in order to assign some variable $x_t$, {\tt BPdec} explores clauses at distance up to $2\omega$ from $x_t$,
where (potentially) $\omega=\omega(n)\ra\infty$.
This renders a `deferred decisions' approach hopeless.

Therefore, to prove \Thm~\ref{Thm_main} we need a fundamentally different strategy.
In the present section we set up the probabilistic framework for the analysis.
We will basically reduce the analysis of {\tt BPdec} to the problem of analyzing the BP operator
on the formula that is obtained from $\PHI$ by substituting `true' for the first
$t$ variables $x_1,\ldots,x_t$ and simplifying (\Thm~\ref{Thm_bias} blow). 
In the next section we will show that this decimated formula enjoys a few simple quasirandomness
properties with probability extremely close to one.
Finally, 
we will show that these  properties suffice
to trace  the BP computation.

Applied to a fix, non-random formula $\Phi$ on $V=\cbc{x_1,\ldots,x_n}$, {\tt BPdec} yields an assignment
$\sigma:V\ra\cbc{-1,1}$ (that may or may not be satisfying).
This assignment is random, because {\tt BPdec} itself is randomized.
Hence, for any fixed $\Phi$ running {\tt BPdec}$(\Phi)$ induces a probability distribution $\beta_\Phi$ on $\cbc{-1,1}^V$.
With $\cS(\Phi)$ the set of all satisfying assignments of $\Phi$,
the `success probability' of {\tt BPdec} on $\Phi$ is just
	$$\mathrm{success}(\Phi)=\beta_\Phi(\cS(\Phi)).$$

Thus, to establish \Thm~\ref{Thm_main} we need to show that in the \emph{random} formula,
	$$\mathrm{success}(\PHI)=\beta_{\PHI}(\cS(\PHI))=\exp(-\Omega(n))$$
is exponentially small  \whp\
To this end, we are going to prove that the measure $\beta_{\PHI}$ is `rather close' to the uniform distribution
on $\cbc{-1,1}^V$ \whp, of which $\cS(\PHI)$ constitutes only an exponentially small fraction.

To facilitate the analysis, we are going to work with a slightly modified version of {\tt BPdec}.
While the original {\tt BPdec} assigns the variables in the natural order $x_1,\ldots,x_n$,
the modified version {\tt PermBPdec} chooses a permutation $\pi$ of $\brk n$ uniformly at random
and assigns the variables in the order $x_{\pi(1)},\ldots,x_{\pi(n)}$.
Let $\bar\beta_\Phi$ denote the probability distribution induced on $\cbc{-1,1}^V$ by {\tt PermBPdec}$(\Phi)$.
Because the uniform distribution over $k$-CNFs
is invariant under permutations of the variables, we obtain

\begin{fact}\label{Fact_pi}
If $\bar\beta_{\PHI}(\cS(\PHI))\leq\exp(-\Omega(n))$ \whp,
then $\mathrm{success}(\PHI)=\beta_{\PHI}(\cS(\PHI))\leq\exp(-\Omega(n))$ \whp
\end{fact}

Let $\Phi$ be a $k$-CNF and let $\delta>0$. 
Given a permutation $\pi$ and a partial assignment $\sigma:\cbc{x_{\pi(s)}:s\leq t}\ra\cbc{-1,1}$, we let
$\Phi_{t,\pi,\sigma}$ denote the formula obtained from $\Phi$ by substituting the values $\sigma(x_{\pi(s)})$
for the variables $x_{\pi(s)}$ for $1\leq s\leq t$ and simplifying.
Formally, $\Phi_{t,\pi,\sigma}$ is obtained from $\Phi$ as follows:
\begin{enumerate}
\item[$\bullet$] remove all clauses $a$ of $\Phi$ that contain a variable $x_{\pi(s)}$ with $1\leq s\leq t$ such that
		$\sigma(x_{\pi(s)})=\sign(x_{\pi(s)},a)$.
\item[$\bullet$] for all clauses $a$ that contain a $x_{\pi(s)}$ with $1\leq s\leq t$ such that
		$\sigma(x_{\pi(s)})\neq\sign(x_{\pi(s)},a)$, remove $x_{\pi(s)}$ from $a$.
\item[$\bullet$] remove any empty clauses (resulting from clauses of $\Phi$
			that become unsatisfied if we set 
			$x_{\pi(s)}$ to 
		$\sigma(x_{\pi(s)})$ for $1\leq s\leq t$)
		from the formula.
\end{enumerate}

For a number $\delta>0$ and an index $l> t$ we say that $x_{\pi(l)}$ is {\bf\em $(\delta,t)$-biased} if 
	$$\abs{\mu_{x_{\pi(l)}}^{\brk\omega}(\Phi_{t,\pi,\sigma})-1/2}>\delta.$$
Moreover, the triple $(\Phi,\pi,\sigma)$ is {\bf\em $(\delta,t)$-balanced} if no more than
$\delta(n-t)$ variables are $(\delta,t)$-biased.

Let $\pi$ be the permutation chosen by {\tt PermBPdec}$(\Phi)$, and let $\sigma$ be the partial assignment
constructed in the first $t$ steps.
The variable $x_{\pi(t+1)}$ is uniformly distributed over the set $V\setminus\cbc{x_{\pi(s)}:s\leq t}$ of currently unassigned variables.
Hence, if  $(\Phi,\pi,\sigma)$ is $(\delta,t)$-balanced,
then the probability that $x_{\pi(t+1)}$ is $(\delta,t)$-biased
is bounded by $\delta$.
(This conclusion was the purpose of decimating the variables in a random order.)
Furthermore, given that $x_{\pi(t+1)}$ is not $(\delta,t)$-biased, 
the probability that {\tt PermBPdec}  will assign set it to `true' lies in the interval
$\brk{\frac12-\delta,\frac12+\delta}$.
Consequently,
	\begin{equation}\label{eqbalanced}
	\abs{\frac12-\pr\brk{\sigma(x_{\pi(t+1)})=1|\mbox{$(\Phi,\pi,\sigma)$ is $(\delta,t)$-balanced}}}\leq2\delta.
	\end{equation}
Thus, the smaller $\delta$, the closer $\sigma(x_{\pi(t+1)})$ comes to being uniformly distributed.
Hence, if $(\delta,t)$-balancedness holds for all $t$ with a `small' $\delta$, then $\bar\beta_\Phi$
will be close to the uniform distribution on $\cbc{-1,1}^V$.

To put this observation to work, we define
	\begin{equation}\label{eqdeltat}
	\delta_t=\exp(-c(1-t/n)k)\quad\mbox{ and }\quad \hat t=\bc{1-\frac{\ln(kr/2^k)}{c^2k}}n,
	\end{equation}
where $c>0$ is a small enough absolute constant
	\footnote{Setting $c=10^{-10^{10}}$ will evidently suffice, but no attempt at finding the optimal $c$ has been made.}.
In addition, we let
	\begin{equation}\label{eqDeltat}
	\Delta_t=\sum_{s=1}^t\delta_t.
	\end{equation}
\begin{lemma}\label{Lemma_Deltat}
For any $0\leq t\leq \hat t$ we have
	$$\Delta_t\leq(1+o(1))\frac{n}{ck\exp(ck(1-t/n))}.$$
Furthermore, $\Delta_{\hat t}\sim \frac{n}{ck}\brk{(kr/2^k)^{-c}-\exp(-ck)}$.
\end{lemma}
\begin{proof}
We have
	\begin{eqnarray}\label{eqLemmaDeltat1}
	\Delta_t&=&\sum_{s=1}^t\delta_s=\exp(-ck)\sum_{s=1}^t\exp(csk/n)=\exp(-ck)\brk{\frac{\exp(ck(t+1)/n)-1}{\exp(ck/n)-1}-1}.
	\end{eqnarray}
Since $\exp(ck/n)=1+ck/n+O(n^{-2})$ and ${\hat t}=\Omega(n)$, we obtain from~(\ref{eqLemmaDeltat1})
	\begin{eqnarray*}
	\Delta_{\hat t}&\sim&\frac{n}{ck}\brk{\exp\bc{ck\bc{\frac{{\hat t}}n-1}}-\exp(-ck)}=\frac{n}{ck}\brk{(kr/2^k)^{-c}-\exp(-ck)}.
	\end{eqnarray*}
Furthermore, for $1\leq t\leq {\hat t}$ equation (\ref{eqLemmaDeltat1}) yields the upper bound
	\begin{eqnarray*}
	\Delta_t&\leq&\exp(-ck)\cdot\frac{\exp(ck(t+1)/n)}{\exp(ck/n)-1}=\frac{\exp(ck(t-n)/n)}{1-\exp(-ck/n)}\\
		&\sim&\frac{n}{ck}\exp(-ck(1-t/n)),
	\end{eqnarray*}
as $\exp(-ck/n)=1-ck/n+O(n^{-2})$.
\qed\end{proof}

For $\xi>0$ we say that $\Phi$ is {\bf\em $(t,\xi)$-uniform} if 
 	$$\abs{\cbc{(\pi,\sigma)\in S_n\times\cbc{-1,1}^V:\mbox{$(\Phi,\pi,\sigma)$ is not $(\delta_t,t)$-balanced}}}\leq2^nn!\cdot\exp\brk{-10(\xi n+\Delta_t)}.$$
Proceeding by induction on $t$, we are going to use~(\ref{eqbalanced}) to relate the distribution $\bar\beta_\Phi$ to the
uniform distribution on $\cbc{-1,1}^V$ for $(t,\xi)$-uniform formulas.
More precisely, in \Sec~\ref{Sec_uniform} we are going to prove

\begin{proposition}\label{Prop_uniform}
Suppose that 
$\Phi$ is $(t,\xi)$-uniform for all $0\leq t\leq {\hat t}$.
Then
	\begin{equation}\label{eqProp_uniform}
	\bar\beta_\Phi(\cE)\leq\frac{|\cE|}{2^{\hat t}}\cdot\exp\brk{4\Delta_{\hat t}}+\exp(-\xi n/2)
		\quad\mbox{for any $\cE\subset\cbc{-1,1}^V$}.
	\end{equation}
\end{proposition}
\Prop~\ref{Prop_uniform} reduces the proof of \Thm~\ref{Thm_main}
to showing that $\PHI$ is $(t,\xi)$-uniform with some appropriate probability.

To prove this, we need two simple definitions.
We call a clause $a$ of a formula $\Phi$ {\bf\em redundant} if $\Phi$ has another clause $b$ such that $a,b$ have at least two variables in common.
Furthermore, we call the formula $\Phi$ {\bf\em tame} if
\begin{enumerate}
\item[i.] $\Phi$ has no more than $\ln n$ redundant clauses, and
\item[ii.] no more than $\ln n$ variables occur in more than $\ln n$ clauses of $\Phi$.
\end{enumerate}
The following is a well-known fact.

\begin{lemma}\label{Lemma_tame}
The random formula $\PHI$ is tame \whp
\end{lemma}

Now, the following result provides the key estimate
for proving that $\PHI$ is $(t,\xi)$-uniform with a very high probability.

\begin{theorem}\label{Thm_bias}
There is a constant $\rho_0>0$ such that for any $k,r$ satisfying
	$2^k\rho_0/k\leq r\leq 2^k\ln 2$ 
there is $\xi=\xi(k,r)>0$ so that for $n$ large enough the following holds.
Fix any permutation $\pi$ of $\brk n$ and any assignment $\sigma\in\cbc{-1,1}^V$.
Then for any $0\leq t\leq {\hat t}$ we have
	\begin{equation}\label{eqThmbias}
	\pr\brk{(\PHI,\pi,\sigma)\mbox{ is $(\delta_t,t)$-balanced}|\PHI\mbox{ is tame}}\geq
		1-\exp\brk{-3\xi n-10\Delta_t}.
	\end{equation}
\end{theorem}
We defer the proof of \Thm~\ref{Thm_bias} to \Sec~\ref{Sec_tracing}.

\begin{corollary}\label{Cor_bias}
In the notation of \Thm~\ref{Thm_bias}, 
	$$\pr\brk{\forall t\leq {\hat t}:\PHI\mbox{ is $(t,\xi)$-uniform}|\PHI\mbox{ is tame}}\geq1-\exp(-\xi n).$$
\end{corollary}
\begin{proof}
For $1\leq t\leq{\hat t}$ and a $k$-CNF $\Phi$ we let $X_t(\Phi)$ signify the number of pairs $(\pi,\sigma)\in S_n\times\cbc{-1,1}^V$
such that $(\Phi,\pi,\sigma)$ fails to be $(\delta_t,t)$-balanced.
Then \Thm~\ref{Thm_bias} yields
	$$\Erw\brk{X_t(\PHI)|\PHI\mbox{ is tame}}\leq 2^nn!\cdot\exp(-3\xi n-10\Delta_t).$$
Hence, by Markov's inequality and the union bound
	\begin{equation}\label{eqCorbias1}
	\pr\brk{\exists t\leq {\hat t}:X_t(\PHI)>2^nn!\cdot\exp(-\xi-10\Delta_t)|\PHI\mbox{ is tame}}\leq n\exp(-2\xi n)\leq\exp(-\xi n).
	\end{equation}
Since $\PHI$ is $(t,\xi)$-uniform if $X_t(\Phi)\leq2^nn!\cdot\exp(-\xi n-10\Delta_t)$, 
the assertion follows from (\ref{eqCorbias1}).
\qed\end{proof}

\medskip
\noindent\emph{Proof of \Thm~\ref{Thm_main}.}
Let us keep the notation of \Thm~\ref{Thm_bias}.
By \Lem~\ref{Lemma_tame} we may condition on $\PHI$ being tame.
Let $\cU$ be the event that $\PHI$ is $(t,\xi)$-uniform for all $1\leq t\leq {\hat t}$.
Let $\cS$ be the event that $|\cS(\PHI)|\leq n\cdot\Erw|\cS(\PHI)|$.
By \Cor~\ref{Cor_bias} and Markov's inequality, we have $\PHI\in\cU\cap\cS$ \whp\
If $\PHI\in\cU\cap\cS$, then by \Prop~\ref{Prop_uniform}
	\begin{eqnarray}\nonumber
	\bar\beta_{\PHI}(\cS(\PHI))&\leq&\frac{\abs{\cS(\PHI)}}{2^{\hat t}}\cdot\exp\brk{4\Delta_{\hat t}}+\exp(-\xi n/2)\\
		&\leq&n\cdot\Erw|\cS(\PHI)|\cdot2^{-{\hat t}}\exp\brk{4\Delta_{\hat t}}+\exp(-\xi n/2).\label{Xeqmain1}
	\end{eqnarray}
By \Lem s~\ref{Lemma_expected} and~\ref{Lemma_Deltat}
we have $\Erw|\cS(\PHI)|\leq2^n\exp(-rn/2^k)$
and $\Delta_{\hat t}\leq\frac{n}{ck}(kr/2^k)^{-c}$. 
Plugging these estimates and the definition~(\ref{eqdeltat}) of ${\hat t}$ into~(\ref{Xeqmain1}), 
we find that given $\PHI\in\cU\cap\cS$,
	\begin{eqnarray*}
	\bar\beta_{\PHI}(\cS(\PHI))&\leq&
		n\exp\brk{n\bc{-\frac{r}{2^k}+\frac{\ln(kr/2^k)}{c^2k}+\frac{4}{ck}(kr/2^k)^{-c}}}+\exp(-\xi n/2).
	\end{eqnarray*}
Recalling that $\rho=kr/2^k$, we thus obtain
	\begin{eqnarray}\label{Xeqmain2}
	\bar\beta_{\PHI}(\cS(\PHI))&\leq&
		n\exp\brk{-\frac nk\bc{\rho-\frac{\ln2\ln\rho}{c^2}-\frac{4}{c\rho^{c}}}}+\exp(-\xi n/2).
	\end{eqnarray}
Hence, if $\rho\geq\rho_0$ for a sufficiently large constant $\rho_0>0$, then
(\ref{Xeqmain2}) yields $\bar\beta_{\PHI}(\cS(\PHI))=\exp(-\Omega(n))$.
Finally, \Thm~\ref{Thm_main} follows from Fact~\ref{Fact_pi}.
\qed

\subsection{Proof of \Prop~\ref{Prop_uniform}}\label{Sec_uniform}

We consider an additional variant of {\tt BPdec} that receives the order $\pi$ in which
variables are to be decimated as an input parameter.

\noindent{
\begin{algorithm}\label{Alg_BPdecsigma}\upshape\texttt{BPdec$(\Phi,\pi)$}\\\sloppy
\emph{Input:} A $k$-SAT formula $\Phi$ on $V=\cbc{x_1,\ldots,x_n}$ and a permutation $\pi\in S_n$.\\
\emph{Output:} An assignment $\tau:V\rightarrow\cbc{-1,1}$.
\begin{tabbing}
mmm\=mm\=mm\=mm\=mm\=mm\=mm\=mm\=mm\=\kill
{\algstyle 0.}	\> \parbox[t]{40em}{\algstyle
	Let $\Phi_0=\Phi$.}\\
{\algstyle 1.}	\> \parbox[t]{40em}{\algstyle
	For $t=0,\ldots,n-1$ do}\\
{\algstyle 2.}	\> \> \parbox[t]{38em}{\algstyle
		Compute the BP results $\mu_x^{\brk\omega}(\Phi_{t})$.}\\
{\algstyle 3.}	\> \> \parbox[t]{38em}{\algstyle
		Let 
				$$\sigma(x_{\pi(t+1)})=\left\{\begin{array}{cl}
							1&\mbox{ with probability $\mu_{x_{\pi(t+1)}}^{\brk\omega}(\Phi_{t})$,}\\
							-1&\mbox{ with probability $1-\mu_{x_{\pi(t+1)}}^{\brk\omega}(\Phi_{t})$}.
							\end{array}\right.$$
		}\\
{\algstyle 4.}	\> \> \parbox[t]{38em}{\algstyle
		Obtain $\Phi_{t+1}$ from $\Phi_{t}$ by substituting the value $\sigma(x_{\pi(t+1)})$ for $x_{\pi(t+1)}$ and simplifying.}\\
{\algstyle 5.}	\> \parbox[t]{40em}{\algstyle
	Return the assignment $\sigma$.}
\end{tabbing}
\end{algorithm}}

Fix a $k$-CNF $\Phi$ that is $(t,\xi)$-uniform for all $1\leq t\leq{\hat t}$.
Let $S_n$ be the set of all permutations on $\brk n$.
Let $\lambda_\Phi$ be the probability distribution  on pairs $(\vec\pi,\vec\sigma)\in S_n\times\cbc{-1,1}^V$ induced by 
choosing a permutation $\vec\pi\in S_n$ uniformly at random and letting $\vec\sigma={\tt BPdec}(\Phi,\vec\pi)$.
Then $\bar\beta_\Phi$ is the $\vec\sigma$-marginal of $\lambda_\Phi$, i.e.,
	\begin{equation}\label{eqbetalambda}
	\bar\beta_\Phi(\cE)=\lambda_\Phi(S_n\times\cE)\quad\mbox{for any }\cE\subset\cbc{-1,1}^V.
	\end{equation}

In order to study $\lambda_\Phi$, we consider another distribution $\lambda_\Phi'$ on 
pairs $(\vec\pi,\vec\sigma')\in S_n\times\cbc{0,1}^V$ that is easier
to analyze and that will turn out to be `close' to $\lambda_\Phi$.
To define $\lambda_\Phi'$, let  $\cB_t$ be the set of all pairs $(\pi,\sigma)$ such that
$(\Phi,\pi,\sigma)$ is not $(\delta_t,t)$-balanced.
Moreover, let $\cB=\bigcup_{t=0}^T\cB_t$.
The distribution $\lambda_\Phi'$ is induced by choosing a permutation $\vec\pi$ uniformly at random 
and running the following algorithm on $\Phi,\vec\pi$.

\noindent{
\begin{algorithm}\label{Alg_BPdecsigma'}\upshape\texttt{BPdec$'(\Phi,\pi)$}\\\sloppy
\emph{Input:} A $k$-SAT formula $\Phi$ on $V=\cbc{x_1,\ldots,x_n}$ and a permutation $\pi\in S_n$.\\
\emph{Output:} An assignment $\vec\sigma':V\rightarrow\cbc{-1,1}$.
\begin{tabbing}
mmm\=mm\=mm\=mm\=mm\=mm\=mm\=mm\=mm\=\kill
{\algstyle 0.}	\> \parbox[t]{40em}{\algstyle
	Let $\Phi_0=\Phi$.}\\
{\algstyle 1.}	\> \parbox[t]{40em}{\algstyle
	For $t=0,\ldots,n-1$ do}\\
{\algstyle 2.}	\> \> \parbox[t]{38em}{\algstyle
		Compute the BP results $\mu_x^{\brk\omega}(\Phi_{t})$.}\\
{\algstyle 3.}	\> \> \parbox[t]{38em}{\algstyle
		If $(\Phi,\pi,\vec\sigma')$ is $(\delta_t,t)$-balanced, then
			}\\
\> \> \> \parbox[t]{36em}{\algstyle
			let 
				$$\vec\sigma'(x_{\pi(t+1)})=\left\{\begin{array}{cl}
							1&\mbox{ with probability $\mu_{x_{\pi(t+1)}}^{\brk\omega}(\Phi_{t})$,}\\
							-1&\mbox{ with probability $1-\mu_{x_{\pi(t+1)}}^{\brk\omega}(\Phi_{t})$}.
							\end{array}\right.$$
	}\\
 \> \> \parbox[t]{38em}{\algstyle else}\\
\> \> \> \parbox[t]{38em}{\algstyle 
		let $\vec\sigma'(x_{\pi(t)})=\zeta$ with probability $\frac12$ for $\zeta=\pm1$.}\\
{\algstyle 4.}	\> \> \parbox[t]{38em}{\algstyle
		Obtain $\Phi_{t+1}$ from $\Phi_{t}$ by substituting the value $\vec\sigma'(x_{\pi(t+1)})$ for $x_{\pi(t+1)}$ and simplifying.}\\
{\algstyle 5.}	\> \parbox[t]{40em}{\algstyle
	Output the assignment $\vec\sigma'$.}
\end{tabbing}
\end{algorithm}
}

Roughly speaking, ${\tt BPdec}'$ disregards the BP outcome if it strays too far from the `flat' vector $\frac12\vecone$.
We claim that $\lambda_\Phi$ and $\lambda_\Phi'$ are related as follows.
For $\cF\subset S_n\times\cbc{-1,1}^V$ let
	$$\cF_{\hat t}=\cbc{(\pi,\sigma)\in S_n\times\cbc{-1,1}^V:\exists (\pi^*,\sigma^*)\in\cF:\forall 1\leq t\leq {\hat t}:
			\pi^*(t)=\pi(t),\sigma^*(x_{\pi(t)})=\sigma(x_{\pi(t)})}.$$
Thus, $\cF_{\hat t}$ is the set of all $(\pi,\sigma)$ that coincide with some $(\pi^*,\sigma^*)\in\cF$ ``up to time $\hat t$''.
In particular, $\cF\subset\cF_{\hat t}$.

\begin{lemma}\label{Lemma_eqlambdalambda'}
For any $\cF\subset S_n\times\cbc{-1,1}^V$ we have
	$\lambda_\Phi(\cF)\leq\lambda_\Phi'(\cF_{\hat t})+\lambda_\Phi'(\cB).$
\end{lemma}
\begin{proof}
By construction, for any $(\pi,\sigma)\not\in\cB_t$ and any $\zeta\in\cbc{-1,1}$ we have
	\begin{eqnarray*}
		\lambda_\Phi\brk{\vec\sigma(x_{\pi(t+1)})=\zeta|\vec\pi=\pi\wedge\forall s\leq t:\vec\sigma(x_{\pi(s)})=\sigma(x_{\pi(s)})}\\
		&\hspace{-8cm}=&\;\hspace{-4cm}
			\lambda_\Phi'\brk{\vec\sigma'(x_{\pi(t+1)})=\zeta|\vec\pi=\pi\wedge\forall s\leq t:\vec\sigma'(x_{\pi(s)})=\sigma(x_{\pi(s)})}.
	\end{eqnarray*}
Hence, Bayes' rule yields that for any pair $(\pi,\sigma)\not\in\cB$,
	\begin{equation}\label{eqlambdalambda'1}
	\lambda_\Phi\brk{\forall t\leq {\hat t}:\vec\pi(t)=\pi(t)\wedge\vec\sigma(t)=\sigma(t)}=
		\lambda_\Phi'\brk{\forall t\leq {\hat t}:\vec\pi(t)=\pi(t)\wedge\vec\sigma'(t)=\sigma(t)}.
	\end{equation}
In particular, $\lambda_\Phi(\cB)=\lambda_\Phi'(\cB)$.
Hence, for any event $\cF$ we obtain
	\begin{eqnarray*}
	\lambda_\Phi(\cF)&\leq&
		\lambda_\Phi(\cF_{\hat t})\leq\lambda_\Phi(\cF_{\hat t}\setminus\cB)+\lambda_\Phi(\cB)
		\;\stacksign{(\ref{eqlambdalambda'1})}{=}\;\lambda_\Phi'(\cF_{\hat t}\setminus\cB)+\lambda_\Phi'(\cB)\leq\lambda_\Phi'(\cF_{\hat t})+\lambda_\Phi'(\cB),
	\end{eqnarray*}
as desired.
\qed\end{proof}
	
Let $\lambda''$ be the uniform probability distribution on $S_n\times\cbc{-1,1}^V$,
and let $(\vec\pi,\vec u)$ denote a pair chosen from $\lambda''$.
To relate $\lambda_\Phi'$ and $\lambda''$,
let $A_t(\pi,\sigma)$ be equal to one
	if $(\pi,\sigma)\not\in\cB_t$ and $x_{\pi(t)}$ is $(\delta_t,t)$-biased in $(\Phi,\pi,\sigma)$,
	and set $A_t(\pi,\sigma)=0$ otherwise.
In addition, let $A(\pi,\sigma)=\sum_{t\leq{\hat t}}A_t(\pi,\sigma)$.

\begin{lemma}\label{Lemma_eqlambda'lambda''}
For any pair $(\pi,\sigma)\in S_n\times\cbc{-1,1}^V$ we have
	\begin{eqnarray*}\nonumber
	\lambda_\Phi'\brk{\forall t\leq {\hat t}:\vec\pi(t)=\pi(t)\wedge\vec\sigma'(x_{\pi(t)})=\sigma(x_{\pi(t)})}\\
		&\hspace{-10cm}\leq&\hspace{-5cm}\;
			\lambda''\brk{\forall t\leq {\hat t}:\vec\pi(t)=\pi(t)\wedge\vec u(x_{\pi(t)})=\sigma(x_{\pi(t)})}
				\cdot2^{A(\pi,\sigma)}\prod_{t\leq T}1+2\delta_t.
		\label{eqlambda'lambda''}
	\end{eqnarray*}
\end{lemma}
\begin{proof}
Fix any pair $(\pi,\sigma)\in S_n\times\cbc{-1,1}^V$ and let
$\cL_t$ be the event that $$\vec\pi(t)=\pi(t)\mbox{ and }\vec\sigma'(x_{\pi(t)})=\sigma(x_{\pi(t)}).$$
Then for any $1\leq t\leq {\hat t}$
we can bound the conditional probability
	$\lambda_\Phi'\brk{\cL_t|\vec\pi(t)=\pi(t)\wedge\bigwedge_{s<t}\cL_s}$ as follows.
\begin{description}
\item[Case 1: $(\pi,\sigma)\in\cB_t$.] In this case $(\Phi,\pi,\sigma)$ is not $(\delta_t,t)$-balanced.
	Therefore, step~3 of {\tt BPdec'}
		chooses the value $\vec\sigma'(x_{\pi(t)})$ uniformly.
		Hence, the event $\vec\sigma'(x_{\pi(t)})=\sigma(x_{\pi(t)})$ occurs with probability $\frac12$.
\item[Case 2: $(\pi,\sigma)\not\in\cB_t$ and $A_t(\pi,\sigma)=0$.]
	Since $(\Phi,\pi,\sigma)$ is $(\delta_t,t)$-balanced, step~3 of {\tt BPdec'} uses the BP marginals $\mu_{x_{\pi(t)}}(\zeta)$ in order to
	assign $x_{\pi(t)}$.
	Because $A_t(\pi,\sigma)=0$, the variable $x_{\pi(t)}$ is not $(\delta_t,t)$-biased, whence
		$\mu_{x_{\pi(t)}}(\zeta)\leq\frac12+\delta_t$ for both $\zeta=-1$ and $\zeta=1$.
	Hence, the probability that $\vec\sigma'(x_{\pi(t)})=\sigma(x_{\pi(t)})$ is bounded by $\frac12+\delta_t$.
\item[Case 3: $A_t(\pi,\sigma)=1$.]
	In this case we just use the trivial fact that the probability of the event $\vec\sigma'(x_{\pi(t)})=\sigma(x_{\pi(t)})$
	is bounded by $1\leq2(\frac12+\delta_t)$.
\end{description}
In any case, we obtain the bound
	$\lambda_\Phi'\brk{\cL_t|\vec\pi(t)=\pi(t)\wedge\bigwedge_{s<t}\cL_s}\leq2^{A_t(\pi,\sigma)}(\frac12+\delta_t)$.
Consequently, as $\lambda''$ is the uniform distribution, we get
	\begin{equation}\label{eqMultiplyUp}
	\frac{\lambda_\Phi'\brk{\cL_t|\vec\pi(t)=\pi(t)\wedge\bigwedge_{s<t}\cL_s}}{\lambda''\brk{\cL_t|\vec\pi(t)=\pi(t)\wedge\bigwedge_{s<t}\cL_s}}
		\leq2^{A_t(\pi,\sigma)}(1+2\delta_t).
	\end{equation}
Multiplying (\ref{eqMultiplyUp}) up for $t\leq {\hat t}$ yields the assertion.
\qed\end{proof}

\noindent
To put \Lem~\ref{Lemma_eqlambda'lambda''} 
 to work, we need to estimate $A(\vec\pi,\vec\sigma')$.

\begin{lemma}\label{Lemma_eqA}
We have
	$\lambda_\Phi'\brk{A(\vec\pi,\vec\sigma')>4(\Delta_{\hat t}+\xi n)}\leq\exp(-\xi n).$
\end{lemma}
\begin{proof}
We are going to bound the probability that $A_t(\vec\pi,\vec\sigma')=1$ given the values
$\vec\pi(s)$, $\vec\sigma'(x_{\vec\pi(s)})$ for $1\leq s<t$.
\begin{description}
\item[Case 1: the event $\cB_t$ occurs.]
	Then $A_t=0$ by definition.
\item[Case 2: the event $\cB_t$ does not occur.]
	In this case $(\Phi,\pi,\sigma)$ is $(\delta_t,t)$-balanced, which means that
	no more than $\delta_t(n-t)$ variables are biased.
	Since the permutation $\vec\pi$ is chosen uniformly at random, the probability that $x_{\vec\pi(t)}$ is $(\delta_t,t)$-biased
	is bounded by $\delta_t$.
\end{description}
Thus, in either case the conditional probability of the event $A_t=1$ is bounded by $\delta_t$.
This implies that the random variable $A(\vec\pi,\vec\sigma')=\sum_{t\leq {\hat t}}A_t(\vec\pi,\vec\sigma')$ is stochastically dominated
by a sum of mutually independent Bernoulli variables with means $\delta_1,\ldots,\delta_{\hat t}$.
Therefore, the assertion follows from \Lem~\ref{Lemma_Chernoff} (the Chernoff bound).
\qed\end{proof}

\medskip
\noindent\emph{Proof of \Prop~\ref{Prop_uniform}.}
Combining \Lem s~\ref{Lemma_eqlambda'lambda''} and~\ref{Lemma_eqA}, we see that
	\begin{eqnarray}\nonumber
	\lambda_\Phi'\brk{\cF_{\hat t}}&\leq&\lambda_\Phi'\brk{A_{\hat t}(\vec\pi,\vec\sigma')>4(\Delta_{\hat t}+\xi n)}+\lambda_\Phi'\brk{\cF_{\hat t}
			\wedge A_{\hat t}\leq4(\Delta_{\hat t}+\xi n)}\\
		&\leq&\exp(-\xi n)+
			\lambda''\brk{\cF_{\hat t}}\cdot2^{4(\Delta_{\hat t}+\xi n)}\prod_{t\leq {\hat t}}1+2\delta_t\nonumber\\
		&\leq&\lambda''\brk{\cF_{\hat t}}\cdot\exp(6\Delta_{\hat t}+4\xi n)+\exp(-\xi n)
			\quad\mbox{ for any $\cF\subset S_n\times\cbc{-1,1}^V$}.
		\label{eqlambda'lambda''final}
	\end{eqnarray}
Our assumption that $\Phi$ is $(t,\xi)$-uniform ensures that
	$\lambda''\brk{\cB_t}\leq\exp(-10(\xi n+\Delta_{\hat t}))$ for any $t\leq {\hat t}$.
Together with~(\ref{eqlambda'lambda''final}), this implies that
	\begin{eqnarray*}
	\lambda_\Phi'\brk{\cB_t}&\leq&\lambda''\brk{\cB_t}\exp(6\Delta_{\hat t}+4\xi n)+\exp(-\xi n)\leq2\exp(-\xi n)\quad\mbox{ for any $t\leq {\hat t}$}.
	\end{eqnarray*}
Therefore, by the union bound
	\begin{equation}\label{eqB}
	\lambda_\Phi'\brk\cB\leq2{\hat t}\exp(-\xi n)\leq\exp(-0.9\xi n).
	\end{equation}
Finally, consider any $\cE\subset\cbc{-1,1}^V$.
Let $\cF=S_n\times\cE$.
Then
	\begin{eqnarray*}
	\bar\beta_\Phi(\cE)&=&\lambda_\Phi\brk{\cF}\qquad\qquad\qquad\qquad\qquad\qquad\qquad\qquad\qquad\mbox{[due to~(\ref{eqbetalambda})]}\\
		&\leq&\lambda_\Phi'\brk{\cF_{\hat t}}+\lambda_\Phi'\brk\cB\qquad\qquad\qquad\qquad\qquad\qquad\qquad\mbox{[by \Lem~\ref{Lemma_eqlambdalambda'}]}\\
		&\leq&\lambda_\Phi'\brk{\cF_{\hat t}}+\exp(-0.9\xi n)\qquad\quad\qquad\qquad\qquad\qquad\mbox{ [by~(\ref{eqB})]}\\
		&\leq&\lambda''\brk{\cF_{\hat t}}\exp(6(\Delta_{\hat t}+\xi n))+\exp(-\xi n/2)\ \qquad\qquad\,\mbox{[by~(\ref{eqlambda'lambda''final})]}\\
		&=&\frac{\abs{\cF_{\hat t}}}{n!2^n}\cdot\exp(6(\Delta_{\hat t}+\xi n))+\exp(-\xi n/2)
				\ \ \qquad\qquad\mbox{[as $\lambda''$ is uniform]}\\
		&=&\frac{\abs\cE}{2^{\hat t}}\cdot\exp(6(\Delta_{\hat t}+\xi n))+\exp(-\xi n/2)
			\quad\ \qquad\qquad\mbox{[by the definition of $\cF_{\hat t}$]},
	\end{eqnarray*}
as desired.
\qed

\section{Tracing the Belief Propagation operator}\label{Sec_tracing}

\subsection{Overview}\label{Sec_Overview}

The goal in this section (and the rest of the paper)
is to establish \Thm~\ref{Thm_bias}, which states that for any fixed permutation $\pi$ and any fixed assignment $\sigma$
the triple $(\PHI,\pi,\sigma)$ is $(\delta_t,t)$-balanced with probability very close to one.
The basic symmetry properties of the random formula $\PHI$ allow us to 
assume without loss of generality that $\pi=\id$ is the identity and that $\sigma=\vecone$ is the all-true assignment.
More precisely, we observe the following.

\begin{fact}
Fix any permutation $\pi$ of $\brk n$ and any assignment $\sigma\in\cbc{0,1}^V$.
Then for any $0\leq t\leq \hat t$ we have
	$$
	\pr\brk{(\PHI,\pi,\sigma)\mbox{ is $(\delta_t,t)$-balanced}}=\pr\brk{(\PHI,\id,\vecone)\mbox{ is $(\delta_t,t)$-balanced}}.$$
\end{fact}
\begin{proof}
For a $k$-CNF $\Phi$ let $\Phi^{\pi,\sigma}$ be the formula obtained by replacing
	\begin{enumerate}
	\item[$\bullet$] each occurrence of the literal $x_i$	
		in $\Phi$ by $x_{\pi(i)}$ if $\sigma(x_{\pi(i)})=1$, and by $\neg x_{\pi(i)}$ if $\sigma(x_{\pi(i)})=-1$, and
	\item[$\bullet$] each occurrence of the literal $\neg x_i$ in $\Phi$ 
		 by $\neg x_{\pi(i)}$ if $\sigma(x_{\pi(i)})=1$, and by $x_{\pi(i)}$ if $\sigma(x_{\pi(i)})=-1$.
	\end{enumerate}
Then $(\Phi,\id,\vecone)$ is $(\delta_t,t)$-balanced iff $(\Phi^{\pi,\sigma},\pi,\sigma)$ is.
Furthermore, the map $\Phi\mapsto\Phi^{\pi,\sigma}$ is a bijection.
Consequently, for the uniformly random formula $\PHI$ the resulting formula $\PHI^{\pi,\sigma}$ is uniformly random as well,
for any $\pi,\sigma$.
\qed\end{proof}

Thus, we assume from now on that $\pi=\id$ and $\sigma=\vecone$.
Then the decimated formula $\PHI_{t,\pi,\sigma}$ is simply obtained from $\PHI$  by substituting the
value `true' for $x_1,\ldots,x_t$ and simplifying.
To unclutter the notation, we are going to denote $\PHI_{t,\pi,\sigma}$ by $\PHI^t$ from now on.
Let $G$ be the factor graph of $\PHI^t$.

Our task is to study the BP operator defined in (\ref{eqBPai}) and~(\ref{eqBPoperator}) on $\PHI^t$.
That is, starting from the initial set of messages $\mu^{\brk 0}_{x\ra a}(\pm1)=\frac12$ for all $x\in V_t$, $a\in N(x)$, we define inductively for $\ell\geq0$
	\begin{eqnarray}\label{eqBPaiell}
	\mu_{a\ra x}^{\brk\ell}(\zeta)&=&\left\{\begin{array}{cl}
		1&\mbox{ if }\zeta=\sign(x,a),\\
		\displaystyle1-\hspace{-4mm}\prod_{y\in N(a)\setminus\cbc x}\hspace{-4mm}
				\mu_{y\ra a}^{\brk\ell}\bc{-\sign(y,a)}&\mbox{ if }\zeta=-\sign(x,a)
		\end{array}\right.
	\end{eqnarray}
and
	\begin{eqnarray}	
	\mu_{x\ra a}^{\brk{\ell+1}}(\zeta)&=&\BP(\mu^{\brk\ell})=\frac{\displaystyle\prod_{b\in N(x)\setminus\cbc a}\mu_{b\ra x}^{\brk\ell}(\zeta)}
			{\displaystyle\prod_{b\in N(x)\setminus\cbc a}\mu_{b\ra x}^{\brk\ell}(-1)+
			\prod_{b\in N(x)\setminus\cbc a}\mu_{b\ra x}^{\brk\ell}(1)},
			\label{eqBPiaell}
	\end{eqnarray}
unless the denominator equals zero, in which case $\mu_{x\ra a}^{\brk{\ell+1}}(\zeta)=\frac12$.

\subsubsection{A non-rigorous sketch of a rigorous analysis.}
Before launching into the details of the (long and technical) proof,
we are going to give a brief sketch based on heuristic considerations.
The aim of this is to develop some intuition. 
Roughly speaking, \Thm~\ref{Thm_bias} asserts that with probability very close to one,
	most of the messages $\mu_{x\ra a}^{\brk\ell}(\pm1)$ are close to $1/2$.
Hence, letting $$\Delta_{x\ra a}^{\brk\ell}=\mu_{x\ra a}^{\brk\ell}(1)-\frac12,$$ we aim to show that $|\Delta_{x\ra a}^{\brk\ell}|$ is small for most $x,a$.
The proof of this is by induction on $\ell$.
That is, given the $\Delta_{x\ra a}^{\brk{\ell}}$ we need to prove that the biases $\Delta_{x\ra a}^{\brk{\ell+1}}$ do not ``blow up''.
More precisely, let us denote by
	\begin{equation}\label{eqtheta}
	\theta=1-t/n
	\end{equation}
the fraction of unassigned variables.
Then our induction hypothesis is that  for all but $\delta_t\theta n$ variables we have
	$$\max_{a\in N(a)}|\Delta_{x\ra a}^{\brk{\ell}}|\leq\delta_t=\exp(-c\theta k),$$
and the goal is to show that the same holds true for $\ell+1$.
To establish this, we need to investigate one iteration  of the update rules~(\ref{eqBPaiell})--(\ref{eqBPiaell}).

Rewriting~(\ref{eqBPaiell}) in terms of the biases $\Delta_{y\ra a}^{\brk{\ell}}$, we obtain
	\begin{eqnarray}\nonumber
	\mu_{a\ra x}^{\brk\ell}(-\sign(x,a))&=&1-\hspace{-4mm}\prod_{y\in N(a)\setminus\cbc x}\hspace{-4mm}
				\frac12-\sign(y,a)\Delta_{y\ra a}^{\brk\ell}\\
				&=&1-2^{1-|N(a)|}\hspace{-4mm}\prod_{y\in N(a)\setminus\cbc x}\hspace{-4mm}
				1-2\sign(y,a)\Delta_{y\ra a}^{\brk\ell}.
				\label{eqheuristic1}
	\end{eqnarray}
How many factors do we expect the product in~(\ref{eqheuristic1}) to have?
In the undecimated formula $\PHI$, each clause has length $k$.
But in $\PHI^t$, only a $\theta$ fraction of variables remain unassigned.
Hence, the average length of a clause of $\PHI^t$ should be $\theta k$.
If indeed $|N(a)|\leq10\theta k$, say, and if $|\Delta_{y\ra a}^{\brk\ell}|\leq\delta_t=\exp(-c\theta k)$ for \emph{all} $y\in N(a)\setminus\cbc x$, then 
we can approximate~(\ref{eqheuristic1}) by
	\begin{eqnarray}
	\mu_{a\ra x}^{\brk\ell}(-\sign(x,a))&=&1-2^{1-|N(a)|}\hspace{-4mm}\prod_{y\in N(a)\setminus\cbc x}\hspace{-4mm}
				1-2\sign(y,a)\Delta_{y\ra a}^{\brk\ell}\nonumber\\
				&\approx&1-2^{1-|N(a)|}\exp\bc{-2\hspace{-4mm}\sum_{y\in N(a)\setminus\cbc x}\hspace{-4mm}\sign(y,a)\Delta_{y\ra a}^{\brk\ell}}\nonumber\\
				&\approx&1-2^{1-|N(a)|}\bc{1-2\hspace{-4mm}\sum_{y\in N(a)\setminus\cbc x}\hspace{-4mm}\sign(y,a)\Delta_{y\ra a}^{\brk\ell}}.
										\label{eqheuristic2}
	\end{eqnarray}
Assume, furthermore, that $a$ is ``not too short'' -- say, $|N(a)|\geq0.1\theta k$.
Then $2^{1-|N(a)|}\leq2^{1-0.1\theta k}$ is small, and thus the expression in~(\ref{eqheuristic2}) is close to $1$.
Hence, we can approximate it by
	\begin{eqnarray}						\label{eqheuristic3}
	\mu_{a\ra x}^{\brk\ell}(-\sign(x,a))&\approx&\exp\brk{-2^{1-|N(a)|}\bc{1-2\hspace{-4mm}
			\sum_{y\in N(a)\setminus\cbc x}\hspace{-4mm}\sign(y,a)\Delta_{y\ra a}^{\brk\ell}}}.
	\end{eqnarray}

To proceed, we are going to plug~(\ref{eqheuristic3}) into~(\ref{eqBPiaell}) to estimate $\Delta_{x\ra a}^{\brk{\ell+1}}$.
While it is easy enough to multiply the exponentials from~(\ref{eqheuristic3}) together to approximate the numerator of~(\ref{eqBPiaell}), 
the denominator seems a bit unwieldy.
To sidestep this issue, we simply estimate the \emph{ratio} $\mu_{x\ra a}^{\brk{\ell+1}}(1)/\mu_{x\ra a}^{\brk{\ell+1}}(-1)$
	(assuming that $\mu_{x\ra a}^{\brk{\ell+1}}(-1)>0$).
The denominator cancels.
Since $\mu_{x\ra a}^{\brk{\ell+1}}(1)+\mu_{x\ra a}^{\brk{\ell+1}}(-1)=1$ by construction, we see that
	$$\frac{\mu_{x\ra a}^{\brk{\ell+1}}(1)}{\mu_{x\ra a}^{\brk{\ell+1}}(-1)}
		=\frac{1+2\Delta_{x\ra a}^{\brk{\ell+1}}}{1-2\Delta_{x\ra a}^{\brk{\ell+1}}}.
		$$
Hence, to show that $\Delta_{x\ra a}^{\brk{\ell+1}}$ is close to zero it suffices to prove that 
$\mu_{x\ra a}^{\brk{\ell+1}}(1)/\mu_{x\ra a}^{\brk{\ell+1}}(-1)$ is close to one.
To this end we invoke (\ref{eqheuristic3}), obtaining
	\begin{eqnarray}\nonumber
	\frac{\mu_{x\ra a}^{\brk{\ell+1}}(1)}{\mu_{x\ra a}^{\brk{\ell+1}}(-1)}&=&\prod_{b\in N(x)\setminus\cbc a}\frac{\mu_{b\ra x}^{\brk\ell}(1)}{\mu_{b\ra x}^{\brk\ell}(-1)}\\		&\approx&
		\exp\brk{\sum_{b\in N(x)\setminus\cbc a}\hspace{-4mm}2^{1-|N(b)|}\bc{\sign(x,b)
				-2\hspace{-4mm}\sum_{y\in N(b)\setminus\cbc x}\hspace{-4mm}\sign(x,b)\sign(y,b)\Delta_{y\ra b}^{\brk\ell}}}.
					\label{eqheuristic4}
	\end{eqnarray}
Thus, we need to show that for all but $\delta\theta n$ variables $x$ the exponent is close to zero.

To deal with the $\sum_{b\in N(x)\setminus\cbc a}2^{1-|N(b)|}\sign(x,b)$ bit,
we need to estimate in how many clauses of a given length $x$ is likely to appear.
Letting
		\begin{equation}\label{eqrho}
		\rho=kr/2^k,
		\end{equation}
we find that the expected number of clauses of length $j$ where $x\in V_t$ appears is asymptotically equal to
	\begin{equation}					\label{eqheuristic5}
	\frac{km}n\bink{k-1}{j-1}\theta^{j-1}\bcfr{1-\theta}{2}^{k-j}
			=\rho2^j\cdot\pr\brk{\Bin(k-1,\theta)=j-1}\leq\rho2^j.
	\end{equation}
Indeed, the expected number of clauses of $\PHI$ that $x$ appears in equals $km/n=kr=2^k\rho$.
Furthermore, each of these gives rise to a clause of length $j$ in $\PHI^t$ iff exactly $j-1$ among the other $k-1$ variables in the clause
are from $V_t$, while the $k-j$ remaining variables are in $V\setminus V_t$ and occur with negative signs.
(If one of them had a positive sign, the clause would have been satisfied by setting the corresponding variable to true. It would thus
not be present in $\PHI^t$ anymore.)
Since $x$ appears with a random sign in each of these clauses, the sum
	$$\sum_{b\in N(x)\setminus\cbc a:|N(b)|=j}\sign(x,b)$$
can be viewed as a random walk with an expected length of $\rho2^j$.
Thus, we expect an outcome of $O(\sqrt{2^j\rho})$.
In this case, we find that
	$$\sum_{b\in N(x)\setminus\cbc a:|N(b)|=j}\hspace{-4mm}
		2^{1-|N(b)|}\sign(x,b)=2^{1-j}\cdot O(\sqrt{2^j\rho})=O(\sqrt \rho 2^{-j/2}).$$
Together with the Chernoff bound, (\ref{eqheuristic5}) shows that $x$ is unlikely to occur in clauses of lengths less than $0.1\theta k$ or more than $10\theta k$.
Furthermore, our assumption that $\theta k\geq\ln(\rho)/c^2$ implies that $\sqrt \rho 2^{-j/2}\leq\exp(-0.01\theta k)$ for all $j\geq0.1\theta k$.
Hence, we expect that for all but, say, $\delta_t\theta n/2$ variables $x\in V_t$ 
	\begin{equation}\label{eqheuristic6}
	\max_{a\in N(x)}\abs{\sum_{b\in N(x)\setminus\cbc a}2^{1-|N(b)|}\sign(x,b)}\leq O(\theta k\exp(-0.01\theta k))\leq\delta_t/4.
	\end{equation}

The second contribution
	$$\sum_{b\in N(x)\setminus\cbc a}
				\sum_{y\in N(b)\setminus\cbc x}\hspace{-4mm}2^{2-|N(b)|}\sign(x,b)\sign(y,b)\Delta_{y\ra b}^{\brk\ell}$$
is a \emph{linear} function of the bias vector $\Delta^{\brk\ell}$ from the previous round.
Indeed, this operator can be represented by a matrix
	\begin{eqnarray*}
	\Lambda^*&=&(\Lambda^*_{x\ra a,y\ra b})_{x\ra a,y\ra b}\quad\mbox{ with entries }\\
	\Lambda^*_{x\ra a,y\ra b}&=&\left\{\begin{array}{cl}
			2^{2-|N(b)|}\sign(x,b)\sign(y,b)&\mbox{ if }a\neq b,\,x\neq y,\mbox{ and }\,b\in N(x),\\
			0&\mbox{ otherwise.}\end{array}\right.
	\end{eqnarray*}
with $x\ra a$, $y\ra b$ ranging over all edges of the factor graph of $\PHI^t$.

Since $\Lambda^*$ is based on $\PHI^t$, it is a random matrix.
One could therefore try to use standard arguments to  bound it in some norm (say, $\cutnorm{\Lambda^*}$).
The problem with this approach is that $\Lambda^*$ is very high-dimensional:
	it operates on a space whose dimension is equal to the number of \emph{edges} of the factor graph.
In effect, standard random matrix arguments do not apply.

To resolve this problem, consider a ``projection'' of $\Lambda^*$ onto a space of dimension merely $|V_t|=\theta n$, namely
	\begin{eqnarray*}
	\Lambda:\RR^{V_t}\ra \RR^{V_t},&&\Gamma=(\Gamma_y)_{y\in V_t}\mapsto
		\cbc{\sum_{b\in N(x)}\sum_{y\in N(b)\setminus\cbc x}2^{2-|N(b)|}\sign(x,b)\sign(y,b)\Gamma_y}_{x\in V_{t}}
	\end{eqnarray*}
One can think of $\Lambda$ as a signed and  weighted adjacency matrix of $\PHI^t$.
Standard arguments easily show
that
	$\cutnorm{\Lambda}\leq\delta_t^4\theta n$ is small 
with a very high probability.
In effect, we expect that for all but, say,  $\delta_t\theta n/2$ variables $x\in V_t$ we have
	\begin{equation}\label{eqheuristic7}
	\max_{a\in N(x)}\abs{\sum_{b\in N(x)\setminus\cbc a}
				\sum_{y\in N(b)\setminus\cbc x}\hspace{-4mm}2^{2-|N(b)|}\sign(x,b)\sign(y,b)\Delta_{y\ra b}^{\brk\ell}}\leq\delta_t/4.
	\end{equation}
Combining (\ref{eqheuristic6}) and~(\ref{eqheuristic7}), we thus expect that for all but $\delta_t\theta n$ variables $x$
the expression~(\ref{eqheuristic4}) is sufficiently close to one to conclude that $\max_{a\in N(x)}\abs{\Delta^{\brk{\ell+1}}_{x\ra a}}\leq\delta_t$,
thereby completing the induction.

\subsubsection{Rigorizing the sketch.}
While the above outlines a strategy for tracing the BP operator, we clearly glossed over numerous issues.
The rest of the paper is devoted to rectifying them.
To provide a bit of orientation, let us briefly highlight the most important items, and indicate how they are going to be fixed.

The first issue is that \Thm~\ref{Thm_bias} claims a rather strong bound on the probability  that $\PHI^t$  is $(\delta_t,t)$-balanced.
To obtain this bound, we are going to proceed in two steps:
	in \Sec~\ref{Sec_R} we will exhibit a small number \emph{quasirandom properties} and
	show that these hold in $\PHI^t$ with the required probability.
Then, in \Sec~\ref{Sec_dynamics} we are going to show \emph{deterministically} that any formula that
has these properties is $(\delta_t,t)$-balanced.

A second major issue is the  $\approx$ signs in the above discussion.
Their use depended on the assumption that $|\Delta_{x\ra a}^{\brk\ell}|\leq\delta_t$ for \emph{all} $x\in V_t$, $a\in N(x)$.
However, this assumption is not going to be valid for any $\ell\geq1$.
Indeed, for some $x$, $\max_{a\in N(x)}|\Delta_{x\ra a}^{\brk\ell}|$ is going to be close or even equal to $1/2$:
	think of a variable that appears in a clause of length one (a ``unit clause''), or of a variable of a very high degree that appears only positively
		(a ``pure literal'').
Hence, we will need to cope with a small but non-empty set $T\brk\ell$ of ``exceptional'' variables $x$ with $\max_{a\in N(x)}|\Delta_{x\ra a}^{\brk\ell}|>\delta_t$.

To study the impact of the exceptional set, we decompose~(\ref{eqheuristic1}) as
	\begin{eqnarray}\nonumber
	\mu_{a\ra x}^{\brk\ell}(-\sign(x,a))&=&1-2^{1-|N(a)|}\brk{\prod_{y\in N(a)\setminus(T\brk\ell\cup\cbc x)}
				1-2\sign(y,a)\Delta_{y\ra a}^{\brk\ell}}\\
				&&\qquad\quad\ \qquad\cdot\brk{\prod_{y\in N(a)\cap T\brk\ell\setminus\cbc x}
				1-2\sign(y,a)\Delta_{y\ra a}^{\brk\ell}}.
				\label{eqheuristic8}
	\end{eqnarray}
There are going to be various cases depending on the length of the clause.
If, say, $0.1\theta k\leq|N(a)|\leq10\theta k$ and $|N(a)\cap T\brk\ell\setminus\cbc x|\leq1$ (i.e., the second product contains at most one factor),
then the above heuristic computation essentially goes through.
This case is going to be represented by the set $N_{\leq1}(x,T\brk\ell)$ below.

More generally, if $0.1\theta k\leq|N(a)|\leq10\theta k$, 
say, then the product~(\ref{eqheuristic8}) is quite close to one, regardless of $|N(a)\cap T\brk\ell\setminus\cbc x|$.
Thus, a single ``exposed'' clause $a$, or even a small number, are not going to affect the ratio~(\ref{eqheuristic4}) much.
To exploit this, 
we will establish as part of the quasirandomness property that for \emph{any} possible set $T\brk\ell$
only very few variables $x$ are ``heavily exposed'',
meaning that they appear in many clauses that contain several variables from $T\brk\ell$
(cf.\ {\bf Q2} and {\bf Q3} below).
Furthermore, we will generally show that there are only very few variables that occur in a clause
$a$ such that $|N(a)|\not\in\brk{0.1\theta k,10\theta k}$ (cf.~{\bf Q1}).

A third issue is the dimension reduction in the linear operator, i.e., that we work with $\Lambda$ instead of $\Lambda^*$.
To vindicate this point, we need to show that for most variables $x$ the bias $\Delta_{x\ra a}^{\brk\ell}$ is essentially independent of $a$.
Furthermore, we need to modify the operator $\Lambda$ to ``cut out'' the exceptional set $T\brk\ell$ where the BP operator has a highly non-linear behavior.
This is going to be mirrored in condition {\bf Q4} below.

Let us now turn this sketch into an actual proof.
In \Sec~\ref{Sec_R} we introduce the quasirandomness property and state the deterministic result about BP on quasirandom formulas (\Thm~\ref{Thm_dynamics}).
Then, from \Sec~\ref{Sec_dynamics} onwards, we prove \Thm~\ref{Thm_dynamics}.
Finally, in \Sec~\ref{Apx_reasonble} we establish that the quasirandomness property holds on $\PHI^t$ with the required probability.

\subsection{The quasirandomness property}\label{Sec_R}

In this section we will exhibit a few simple quasirandomness properties that $\PHI^{t}$
is very likely to possess.
From \Sec~\ref{Sec_dynamics} on we will show that these properties suffice to trace the BP operator.

To state the quasirandomness properties, fix a $k$-CNF $\Phi$.
Let $\Phi^t=\Phi_{t,\id,\vecone}$ denote the CNF obtained from $\Phi$ by substituting `true' for $x_1,\ldots,x_t$ and simplifying ($1\leq t\leq n$).
Let $V_{t}=\cbc{x_{t+1},\ldots,x_n}$ be the set of variables of $\Phi^t$.
As before, we will denote the factor graph of $\Phi^t$ by $G=G(\Phi^t)$,
and the neighborhood of a vertex $v$  by $N(v)$.
We continue to let $\theta$ and $\rho$ be defined as in~(\ref{eqtheta}) and~(\ref{eqrho}).

For a variable $x\in V_t$ and a set $T\subset V_t$ let
	\begin{eqnarray}\label{eqNleq1}
	N_{\leq1}(x,T)&=&\cbc{b\in N(x):\abs{N(b)\cap T\setminus\cbc x}\leq1\wedge0.1\theta k\leq|N(b)|\leq10\theta k}.
	\end{eqnarray}
Thus, $N_{\leq1}(x,T)$ is the set of all clauses that contain $x$ (which may or may not be in $T$) and at most
one other variable from $T$.
In addition, there is a condition on the \emph{length} $|N(b)|$ of the clause $b$ in the decimated formula $\Phi_t$.
Recall from \Sec~\ref{Sec_Overview} that having assigned the first $t$ variables, we should `expect' the average clause length to be $\theta k$.

With $c>0$ as in~(\ref{eqdeltat})
we let
	$$k_1=\sqrt c\theta k.$$
Moreover, for a variable $x\in V_t$ and a set $T\subset V_t$ let
	\begin{eqnarray*}
	N_1(x,T)&=&\cbc{b\in N(x):\abs{N(b)\setminus T}\geq k_1,|N(b)\cap T\setminus x|=1},\\
	N_{>1}(x,T)&=&\cbc{b\in N(x):\abs{N(b)\setminus T}\geq k_1,|N(b)\cap T\setminus x|>1}.
	\end{eqnarray*}

\begin{definition}\label{Def_reasonable}
Let $\delta>0$.
We say that $\Phi$ is \emph{\bf\em ($\delta,t)$-quasirandom} if  {\bf Q0}--{\bf Q4}  in Figure~\ref{Fig_reasonable}
are satisfied.
\end{definition}

\begin{figure}[t]\center\normalsize
\begin{tabular}{ll}
{\bf Q0.}&  \parbox[t]{14cm}{\itshape $\Phi$ is tame.}\\
{\bf Q1.}& \parbox[t]{14cm}{\itshape 
	No more than $10^{-5}\delta\theta n$ variables occur in clauses
				of length less than $\theta k/10$ or greater than $10\theta k$ in $\Phi^t$.
	Moreover, there are at most $10^{-4}\delta\theta n$  variables $x\in V_{t}$
		such that
			$$\textstyle(\theta k)^3\delta\cdot \sum_{b\in N(x)}2^{-|N(b)|}>1.$$ }\\
{\bf Q2.}& \parbox[t]{14cm}{\itshape 
	If $T\subset V_{t}$ has size $\abs T\leq\delta\theta n$, then there are no more than
		$10^{-4}\delta\theta n$ variables $x$ such that either
			\begin{eqnarray*}
			\sum_{b\in N_1(x,T)}
					2^{-|N(b)|}>\rho(\theta  k)^5\delta,&\mbox{ or }&
			\sum_{b\in N_{>1}(x,T)}
						2^{|N(b)\cap T\setminus\cbc x|-|N(b)|}>\frac{\delta}{\theta  k},\mbox{ or }\\
				\abs{\sum_{b\in N_{\leq1}(x,T)}\hspace{-2mm}\frac{\sign(x,b)}{2^{|N(b)|}}}>\frac{\delta}{1000}.
			\end{eqnarray*}}\\
{\bf Q3.}& \parbox[t]{14cm}{\itshape For any $0.01\leq z\leq 1$ and any set $T\subset V_{t}$ of size
		$|T|\leq 100\delta\theta n$
			we have
				$$\sum_{b:|N(b)\cap T|\geq z|N(b)|}|N(b)|\leq\frac{1.01}z|T|+10^{-4}\delta\theta n.$$
			}\\
{\bf Q4.}& \parbox[t]{14cm}{\itshape For any set $T\subset V_{t}$ of size $|T|\leq 10\delta\theta n$ the linear operator
	$\Lambda_T:\RR^{V_{t}}\rightarrow\RR^{V_{t}}$, 
			\begin{eqnarray*}
				\Gamma=(\Gamma_y)_{y\in V_t}&\mapsto&
				\cbc{\sum_{b\in N_{\leq1}(x,T)}\sum_{y\in N(b)\setminus\cbc x}2^{-|N(b)|}\sign(x,b)\sign(y,b)\Gamma_y}_{x\in V_{t}}
			\end{eqnarray*}
		has norm $\cutnorm{\Lambda_T}\leq \delta^4\theta n$.}\\
\end{tabular}
\caption{The conditions for Definition~\ref{Def_reasonable}.}\label{Fig_reasonable}
\end{figure}

Condition {\bf Q0} simply bounds the number of redundant clauses and the number of variables of very high degree;
it well-known to hold for random $k$-CNFs \whp\
Apart from a bound on the number of very short/very long clauses,
{\bf Q1} provides a bound on the `weight' of clauses in which variables $x\in V_{t}$ typically occur,
where the weight of a clause $b$ is $2^{-|N(b)|}$.
Moreover, {\bf Q2} provides that there is no small set $T$ for which the total weight of the clauses touching that set is very big.
In addition, {\bf Q2} (essentially) requires that for most variables $x$ the weights of the clauses where $x$ occurs positively/negatively should
approximately cancel. 
Further, {\bf Q3} provides a bound on the lengths of clauses that contain many variables from a small set $T$.
Finally, the most important condition is {\bf Q4}, providing a bound on the cut norm
of a signed, weighted matrix representation of $\Phi^t$.

\begin{proposition}\label{Prop_reasonable}
There exists a constant $\rho_0>0$ such that for any $k,r$ satisfying $\rho_0\cdot2^k/k\leq r\leq2^k\ln2$
there is $\xi=\xi(k,r)>0$ so that for $n$ large and $\delta_t$, $\hat t$ as in~(\ref{eqdeltat})  
for any $1\leq t\leq \hat t$ we have
	$$\textstyle\pr\brk{\PHI\mbox{ is $(\delta_t,t)$-quasirandom}\,|\,\mbox{$\PHI$ is tame}}\geq
			1-\exp\brk{-10\bc{\xi n+\Delta_t}}.$$
\end{proposition}
The proof of \Prop~\ref{Prop_reasonable} is a necessary evil:
it is long, complicated and based on standard arguments.
	We defer it to \Sec~\ref{Apx_reasonble}.
Together with the following theorem, which we will establish in \Sec~\ref{Sec_dynamics},  \Prop~\ref{Prop_reasonable} yields
\Thm~\ref{Thm_bias}.

\begin{theorem}\label{Thm_dynamics}
There is $\rho_0>0$ such that for any $k,r$ satisfying $\rho_0\cdot 2^k/k\leq r\leq2^k\ln2$ and $n$ sufficiently large
the following is true.
Let $\Phi$ be a $k$-CNF with $n$ variables and $m$ clauses
that is $(\delta_t,t)$-quasirandom for some $1\leq t\leq \hat t$.
Then $(\Phi,\id,\vecone)$ is $(\delta_t,t)$-balanced.
\end{theorem}

\noindent
The rest of this section deals with the proof of \Thm~\ref{Thm_dynamics}.

\medskip\noindent
{\bf\em For the rest of \Sec~\ref{Sec_tracing}, we keep the notation from \Sec~\ref{Sec_R} and the assumptions of \Thm~\ref{Thm_dynamics}.
To unclutter the notation, we let $\delta=\delta_t$.}

\subsection{Belief Propagation on quasirandom formulas: proof of \Thm~\ref{Thm_dynamics}
	}\label{Sec_dynamics}

Implementing the strategy outlined in \Sec~\ref{Sec_Overview},
we are going to trace the BP operator when iterated from the initial point
	$$\mu_{x\ra a}^{\brk 0}(-1)=\mu_{x\ra a}^{\brk 0}(1)=\frac12\qquad\mbox{ for all $x\in V_t$, $a\in N(x)$}.$$
Let $\mu^{\brk\ell}=\BP^\ell(\mu\brk 0)\in M(\Phi)$ be the result of the first $\ell$ iterations of BP.
Let
	$$\Delta_{x\ra a}^{\brk{\ell}}=\mu_{x\ra a}^{\brk{\ell}}(1)-\frac12.$$
We say that  $x\in V_t$ is $\ell$-{\bf\emph{biased}} if
	$$\max_{a\in N(x)}|\Delta_{x\ra a}^{\brk{\ell}}|\geq0.1\delta.$$
Clearly, no variable is $0$-biased.
Let $B\brk\ell$ be the set of all $\ell$-biased variables.
To prove \Thm~\ref{Thm_dynamics}, the core task will be to bound $|B\brk\ell|$.

To this end,  we are going to construct a sequence of sets $T\brk\ell$ whose sizes are easier to estimate
and that will turn out to be supersets of the $B\brk\ell$.
Actually we will construct sets of variables $T_1\brk\ell$, $T_2\brk\ell$ and sets of clauses $T_3\brk\ell$ inductively
and let $T\brk\ell=T_1\brk\ell\cup T_2\brk\ell\cup N(T_3\brk\ell)$.

For $\ell=0$ we let
$T_1\brk0=T_3\brk0=\emptyset$.
Moreover, let $T_2\brk0$ be the set of all variables $x$ such that
there is a clause $b\in N(x)$ 
		that is either redundant, or $|N(b)|<0.1\theta k$, or $|N(b)|>10\theta k$, or that satisfy
	$\delta(\theta k)^3\sum_{b\in N(x)}2^{-|N(b)|}>1$.

To define $T\brk{\ell+1}$ inductively for $\ell\geq0$, we need a bit of notation:
for $x\in V$ and $a\in N(x)$ we let
	\begin{eqnarray}\label{eqNxa}
	N_{\leq1}^{\brk{\ell+1}}\bc{x\ra a}&=&\cbc{b\in N_{\leq1}(x,T\brk\ell)\setminus\cbc a:\mu_{b\ra x}^{\brk\ell}(-1)>0}.
	\end{eqnarray}
Furthermore, 
 set
	\begin{eqnarray*}
	P_{\leq1}^{\brk{\ell+1}}\bc{x\ra a}&=&\prod_{b\in N_{\leq1}^{\brk{\ell+1}}\bc{x\ra a}}\frac{\mu_{b\ra x}^{\brk\ell}(1)}{\mu_{b\ra x}^{\brk\ell}(-1)},\nonumber\\
	\end{eqnarray*}
In addition, let
	\begin{eqnarray*}
	N_{>1}^{\brk{\ell+1}}\bc{x\ra a}&=&\cbc{b\in N(x)\setminus(\cbc a\cup N_{\leq1}(x,T\brk\ell)):\mu_{b\ra x}^{\brk\ell}(-1)>0},\nonumber\\
	P_{>1}^{\brk{\ell+1}}\bc{x\ra a}&=&\prod_{b\in N_{>1}^{\brk{\ell+1}}(x\ra a)}\frac{\mu_{b\ra x}^{\brk\ell}(1)}{\mu_{b\ra x}^{\brk\ell}(-1)}.\nonumber
	\end{eqnarray*}

The motivation behind these definitions is the following.
Assume for a moment that $\mu_{b\ra x}^{\brk\ell}(-1)\not=0$ for all $b\in N(x)$.
As we saw in \Sec~\ref{Sec_Overview}, to show that
 $\Delta_{x\ra a}^{\brk{\ell+1}}=\mu_{x\ra a}^{\brk{\ell+1}}(1)-\frac12$ is close to zero it suffices to verify that the ratio
	\begin{equation}\label{eqidea}
	\frac{\mu_{x\ra a}^{\brk{\ell+1}}(1)}{\mu_{x\ra a}^{\brk{\ell+1}}(-1)}
		=\prod_{b\in N(x)\setminus\cbc a}\frac{\mu_{b\ra x}^{\brk\ell}(1)}{\mu_{b\ra x}^{\brk\ell}(-1)}=
			P_{\leq1}^{\brk{\ell+1}}\bc{x\ra a}\cdot P_{>1}^{\brk{\ell+1}}\bc{x\ra a}
	\end{equation}
is close to one, because $\mu_{x\ra a}^{\brk{\ell+1}}(-1)+\mu_{x\ra a}^{\brk{\ell+1}}(1)=1$ by construction.
Moreover, (\ref{eqidea}) is close to one if both factors on the r.h.s.\ are.

Now, we let $T_1\brk{\ell+1}$ contain all variables for which $P_{\le1}^{\brk{\ell+1}}\bc{x\ra a}$ fails to be close enough to one:
	$$T_1\brk{\ell+1}=\cbc{x\in V:\max_{a\in N(x)}\abs{P_{\leq1}^{\brk{\ell+1}}\bc{x\ra a}-1}>0.01\delta}.$$
To also deal with the second product $P_{>1}^{\brk{\ell+1}}\bc{x\ra a}$, we define additional sets $T_2\brk{\ell+1}$, $T_3\brk{\ell+1}$.
To define $T_2\brk{\ell+1}$, let us say that a variable $x$ is $(\ell+1)$-{\bf\emph{harmless}} if it enjoys the following four properties.
\begin{description}
\item[H1.] We have $\delta(\theta k)^3\sum_{b\in N(x)}2^{-|N(b)|}\leq1$,
			and  $0.1\theta k\leq|N(b)|\leq10\theta k$ for all $b\in N(x)$.
\item[H2.] $\sum_{b\in N_1(x,T\brk\ell)}2^{-|N(b)|}\leq\rho(\theta k)^5\delta$
		and
		$\sum_{b\in N_{>1}(x,T\brk\ell))}
			2^{|N(b)\cap T\brk{\ell}\setminus\cbc x|-|N(b)|}\leq\delta/(\theta k)$.
\item[H3.] There is at most one clause $b\in N(x)$ such that $|N(b)\setminus T\brk{\ell}|\leq k_1$.
\item[H4.] $\abs{\sum_{b\in N_{\leq1}(x)}\sign(x,b)\cdot2^{-|N(b)|}}\leq0.01\delta$.
\end{description}
Let $H\brk{\ell+1}$ signify the set of all $(\ell+1)$-harmless variables.
Further, let $T_2\brk{\ell+1}$ be the set of all variables $x$ that have at least one of the following properties.

\medskip
\noindent
\begin{tabular}{ll}
{\bf T2a.}&\parbox[t]{14cm}{There is a clause $b\in N(x)$ 
		that is either redundant, or $|N(b)|<0.1\theta k$, or $|N(b)|>10\theta k$.}\\
{\bf T2b.}&\parbox[t]{14cm}{$\delta(\theta k)^3\sum_{b\in N(x)}2^{-|N(b)|}>1$.}\\
{\bf T2c.}&\parbox[t]{14cm}{Either
				\begin{eqnarray*}
			\sum_{b\in N_1(x,T\brk\ell)}
					2^{-|N(b)|}&>&\rho(\theta k)^5\delta,\quad\mbox{ or }\quad
			\sum_{b\in N_{>1}(x,T\brk\ell)}
						2^{\abs{N(b)\cap T\brk\ell\setminus\cbc x}-|N(b)|}>\delta/(\theta k).
			\end{eqnarray*}}\\
{\bf T2d.}&\parbox[t]{14cm}{$x$ occurs in more than $100$ clauses from $T_3\brk\ell$.}\\
{\bf T2e.}&\parbox[t]{14cm}{$x$ occurs in a clause $b$ that contains 
		fewer than $3|N(b)|/4$ variables from $H\brk{\ell}$.}
\end{tabular}

\medskip
\noindent
Items {\bf Q0} and {\bf Q1} from Definition~\ref{Def_reasonable} ensure
that there are only a very few variables that satisfy {\bf H1}, {\bf T2a}, or {\bf T2b}.
We always include these few into the set $T_2\brk{\ell+1}$ of `exceptional' variables.
Moreover, intuitively {\bf H2} and {\bf T2c}--{\bf T2e} capture variables $x$
that are highly exposed to the `exceptional' set $T\brk\ell$ from the previous round.
Furthermore, we let
	\begin{equation}\label{eqDefT3}
	T_3\brk{\ell+1}=\cbc{a\in \Phi^t:|N(a)|\geq100k_1\wedge|N(a)\setminus T\brk\ell|\leq k_1}\setminus T_3\brk\ell
	\end{equation}
contain all clauses that consist almost entirely of `exceptional' variables from $T\brk\ell$,
but \emph{without} including the clauses from the previous set $T_3\brk\ell$.
Finally,
	$$T\brk{\ell+1}=T_1\brk{\ell+1}\cup T_2\brk{\ell+1}\cup N(T_3\brk{\ell+1}).$$
In \Sec~\ref{Sec_TB} we will verify that $T\brk\ell$ does indeed contain the set $B\brk\ell$ of biased variables.
\begin{proposition}\label{Prop_TB}
We have $B\brk\ell\subset T\brk\ell$ for all $\ell\geq0$.
\end{proposition}
Furthermore, in \Sec~\ref{Sec_Tbound} we will establish the following bound on the size of $T\brk\ell$.

\begin{proposition}\label{Prop_Tbound}
We have $\abs{T\brk\ell}<\delta\theta n$ for all $\ell\geq0$.
\end{proposition}
Finally, in \Sec~\ref{Sec_finish} we will derive \Thm~\ref{Thm_dynamics} from \Prop~\ref{Prop_TB} and \Prop~\ref{Prop_Tbound}.

\subsection{Proof of \Prop~\ref{Prop_TB}}\label{Sec_TB}

The proof will be by induction on $\ell$.
We begin with an elementary estimate of the messages $\mu_{b\ra x}$ from clauses to variables.

\begin{lemma}\label{Lemma_P1T}
Let $x$ be a variable and let $b\in N(x)$ be a clause.
Let $t_b=\abs{N(b)\cap B\brk\ell\setminus\cbc x}$.
Then
	$$\abs{1-\mu_{b\ra x}^{\brk{\ell}}(\zeta)}\leq
		2^{2-|N(b)|+t_b}\exp(\delta|N(b)|)
			\quad\mbox{ for $\zeta=\pm1$}.$$
Furthermore, if $2^{2-|N(b)|+t_b}\exp(\delta|N(b)|)|\leq1/2$, then
	$$\exp\brk{-2^{3-|N(b)|+t_b}\exp(\delta|N(b)|)}\leq\mu_{b\ra x}^{\brk{\ell}}(\zeta)\leq1\quad\mbox{ for $\zeta=\pm1$}.$$
\end{lemma}
\begin{proof}
Since for any $y\in N(b)\setminus\cbc x$ we have
$\mu_{y\ra b}^{\brk\ell}(1)=\frac12+\Delta_{y\ra b}^{\brk\ell}$ and
$\mu_{y\ra b}^{\brk\ell}(-1)+\mu_{y\ra b}^{\brk\ell}(1)=1$, we see that
	$$\mu_{y\ra b}^{\brk\ell}\bc{-\sign\bc{y,b}}=\frac12-\sign\bc{y,b}\Delta_{y\ra b}^{\brk\ell}.$$
Therefore, by the definition~(\ref{eqBPaiell}) of $\mu_{b\ra x}^{\brk\ell}(\pm1)$, we have
	\begin{eqnarray*}
	0&\leq&1-\mu_{b\ra x}^{\brk\ell}(-\sign(x,b))
		=\prod_{y\in N(b)\setminus\cbc x}\frac12-\sign\bc{y,b}\Delta_{y\ra b}^{\brk\ell}\\
	&=&2^{1-|N(b)|}\prod_{y\in N(b)\setminus\cbc x}1-2\sign\bc{y,b}\Delta_{y\ra b}^{\brk\ell}\\
	&\leq&2^{1-|N(b)|}\cdot 2^{t_b}\cdot\prod_{y\in N(b)\setminus(\cbc x\cup B\brk\ell)}1+2|\Delta_{y\ra b}^{\brk\ell}|
			\quad\qquad\mbox{ [as $\Delta_{y\ra b}^{\brk\ell}\in\brk{-1/2,1/2}$ for all $y$]}\\
	&\leq&2^{1-|N(b)|}\cdot 2^{t_b}\cdot\exp\brk{2\sum_{y\in N(b)\setminus(\cbc x\cup B\brk\ell)}|\Delta_{y\ra b}^{\brk\ell}|}\\
	&\leq&2^{1-|N(b)|+t_b}\cdot\exp(|N(b)|\delta)\qquad\qquad\qquad\qquad
		\qquad\mbox{[as $|\Delta_{y\ra b}^{\brk\ell}|\leq0.1\delta$ for all $y\not\in B\brk\ell$]}.
	\end{eqnarray*}
The second assertion follows from the elementary inequality $1-z\geq\exp(-2z)$ for $0\leq z\leq1/2$.
\qed\end{proof}

\begin{corollary}\label{Cor_P1T}
Let $x$ be a variable and let $\cT\subset N(x)$ be a set of clauses.
For each $b\in\cT$ let $t_b=\abs{N(b)\cap B\brk\ell\setminus\cbc x}$.
Assume that $t_b<|N(b)|-2$  and $|N(b)|\leq10\theta k$ for all $b\in\cT$.
Then $\mu_{b\ra x}^{\brk\ell}(\pm1)>0$ for all $b\in\cT$ and
	$$\abs{\ln\prod_{b\in\cT}\frac{\mu_{b\ra x}^{\brk\ell}(1)}{\mu_{b\ra x}^{\brk\ell}(-1)}}\leq
			\sum_{b\in\cT}2^{4-|N(b)|+t_b}
				.$$
\end{corollary}
\begin{proof}
For each $b\in\cT$ there is $y\in N(b)\setminus\cbc x$ such that $y\not\in B\brk\ell$,
	because $t_b<|N(b)|-2$.
Therefore, (\ref{eqBPaiell}) shows that $\mu_{b\ra x}^{\brk\ell}(\pm1)>0$.
Since by definition $\mu_{b\ra x}^{\brk\ell}(\zeta)\leq1$ for $\zeta=\pm1$,
\Lem~\ref{Lemma_P1T} implies that for any $b\in\cT$ and we have
	\begin{eqnarray}\label{eqMultiplyMe1}
		\frac{\mu_{b\ra x}^{\brk\ell}(\zeta)}{\mu_{b\ra x}^{\brk\ell}(-\zeta)}
		&\geq&
		1-2^{2-|N(b)|+t_b}\exp(\delta|N(b)|).
	\end{eqnarray}
Our assumptions $t_b<|N(b)|-2$ and $|N(b)|\leq10\theta k$
ensure that
	$$2^{2-|N(b)|+t_b}\leq1/2\quad\mbox{ and }\quad\exp(\delta|N(b)|)\leq1.1,$$
whence $2^{2-|N(b)|+t_b}\exp(\delta|N(b)|)\leq0.6$.
Due to the elementary inequality $1-z\geq\exp(-2z)$ for $z\in\brk{0,0.6}$, (\ref{eqMultiplyMe1}) thus yields
	\begin{eqnarray}\label{eqMultiplyMe2}
		\frac{\mu_{b\ra x}^{\brk\ell}(\zeta)}{\mu_{b\ra i}^{\brk\ell}(-\zeta)}&\geq&\exp\brk{-2^{3-|N(b)|+t_b}\exp(\delta|N(b)|)}
			\geq\exp\brk{-2^{4-|N(b)|+t_b}}.
	\end{eqnarray}
Multiplying~(\ref{eqMultiplyMe2}) up over $b\in\cT$ and taking logarithms yields
	\begin{eqnarray}\label{eqMultiplyMe3}
		\ln\prod_{b\in\cT}\frac{\mu_{b\ra x}^{\brk\ell}(\zeta)}{\mu_{b\ra x}^{\brk\ell}(-\zeta)}&\geq&-\sum_{b\in\cT}2^{3-|N(b)|+t_b}\exp(\delta|N(b)|).
	\end{eqnarray}
Since~(\ref{eqMultiplyMe3}) holds for both $\zeta=-1$ and $\zeta=1$, the assertion follows.
\qed\end{proof}

\begin{corollary}\label{Cor_harmless}
Suppose that $x\in H\brk\ell$ and that $a\in N(x)$ is a clause such that $|N(a)\setminus T\brk{\ell-1}|\leq k_1$.
Moreover, assume that $B\brk{\ell-1}\subset T\brk{\ell-1}$.
Then $|\Delta_{x\ra a}^{\brk\ell}|\leq0.01$.
\end{corollary}
\begin{proof}
For each $b\in N(x)\setminus\cbc a$ let $t_b=\abs{N(b)\cap B\brk{\ell-1}\setminus\cbc x}$.
Then our assumption that $B\brk{\ell-1}\subset T\brk{\ell-1}$ and condition {\bf H3} ensure that for any $b\in N(x)\setminus\cbc a$,
	$$t_b\leq\abs{N(b)\cap T\brk{\ell-1}}\leq|N(b)|-k_1<|N(b)|-2.$$
Furthermore, by {\bf H1} we have $0.1\theta k\leq|N(b)|\leq10k\theta$ for all $b\in N(x)\setminus\cbc a$.
Therefore, \Cor~\ref{Cor_P1T} applies to the set $\cT=N_{>1}(x,T\brk\ell)\setminus\cbc a.$
Since \Cor~\ref{Cor_P1T} yields $\mu_{b\ra x}^{\brk{\ell-1}}\bc0>0$ for all $b\in\cT$, we have
	$\cT=N_{>1}^{\brk\ell}(x\ra a)$, 
	 and thus
	\begin{eqnarray}\label{eqP>1bound}
	|\ln P_{>1}^{\brk\ell}\bc{x\ra a}|&=&
		\abs{\ln\prod_{b\in\cT}\frac{\mu_{b\ra x}^{\brk\ell}(1)}{\mu_{b\ra x}^{\brk\ell}(-1)}}\leq
			\sum_{b\in\cT}2^{4-|N(b)|+t_b}.
	\end{eqnarray}
Moreover, {\bf H2} ensures that
	$\sum_{b\in\cT}2^{t_b-|N(b)|}\leq\delta$, 
whence~(\ref{eqP>1bound}) entails
	\begin{equation}\label{eqWorkItOut}
	|P_{>1}^{\brk\ell}\bc{x\ra a}-1|\leq 10^{-4}.
	\end{equation}

Furthermore,  by {\bf H1} all clauses $b\in N(x)$ have lengths $0.1\theta k\leq |N(b)|\leq10\theta k$.
Moreover, for all $b\in N(x)\setminus\cbc a$ we have $|N(b)\setminus T\brk{\ell-1}|\geq k_1$ by {\bf H3},
and thus $N_{1}(x,T\brk{\ell-1})\subset N_{\leq1}(x,T\brk{\ell-1})$.
Further, since $|N(a)\cap T\brk{\ell-1}|>1$ by assumption, 
we have 
	$$N_{\leq 1}(x\ra a)\brk\ell=N_{\leq1}(x,T\brk{\ell-1}).
		$$
Hence, letting $\cN=N_{\leq1}(x,T\brk{\ell-1})\setminus N_{1}(x,T\brk{\ell-1})$, we have
	\begin{eqnarray}\label{eqWorkItOut0}
	\abs{P_{\leq1}\brk\ell\bc{x\ra a}-1}&=&\prod_{b\in\cN}
			\frac{\mu_{b\ra x}^{\brk{\ell-1}}(1)}{\mu_{b\ra x}^{\brk{\ell-1}}(-1)}\cdot
			\prod_{b\in N_{1}(x,T\brk{\ell-1})}	\frac{\mu_{b\ra x}^{\brk{\ell-1}}(1)}{\mu_{b\ra x}^{\brk{\ell-1}}(-1)}.
	\end{eqnarray}
With respect to the second product, 
\Cor~\ref{Cor_P1T} yields
	\begin{eqnarray}\label{eqWorkItOut10}
	\abs{\ln\prod_{b\in N_{1}(x,T\brk{\ell-1})}	\frac{\mu_{b\ra x}^{\brk{\ell-1}}(1)}{\mu_{b\ra x}^{\brk{\ell-1}}(-1)}}
		&\leq&\sum_{b\in N_1(x,T\brk{\ell-1})}2^{5-|N(b)|}\\
		&\stacksign{\bf H2}{\leq}&32\rho(\theta k)^5\delta\leq10^{-6}\quad\mbox{ [as $\delta=\exp(-c\theta k)$ with $\theta k\geq\ln(\rho)/c^2$]}.\nonumber
	\end{eqnarray}

Furthermore, for any $b\in\cN$ we have
	\begin{eqnarray}\nonumber
	\mu_{b\ra x}^{\brk{\ell-1}}(-\sign(x))&=&
		1-\prod_{y\in N(b)\setminus\cbc x}\frac12-\sign\bc{y,b}\Delta_{y\ra b}^{\brk{\ell-1}}\\
	&=&1-2^{1-|N(b)|}\prod_{y\in N(b)\setminus\cbc x}1-2\sign\bc{y,b}\Delta_{y\ra b}^{\brk{\ell-1}}.\label{eqWorkItOut1}
	\end{eqnarray}
Since $b\in\cN$, for all $y\in N(b)\setminus\cbc x$ we have $y\not\in B\brk{\ell-1}\subset T\brk{\ell-1}$,
and thus $|\Delta_{y\ra b}^{\brk{\ell-1}}|\leq0.1\delta$.
Moreover, $|N(b)|\leq10k\theta$ by {\bf H1}.
Thus, letting
	$$\alpha_b=1-\prod_{y\in N(b)\setminus\cbc x}1-2\sign\bc{y,b}\Delta_{y\ra b}^{\brk{\ell-1}},$$
we find
	\begin{equation}
	0\leq\alpha_b\leq1-(1-0.2\delta)^{|N(b)|}\leq8\delta k\theta.
	\label{eqWorkItOut2}
	\end{equation}
Since $|N(b)|\geq0.1k\theta$ by {\bf H1}, (\ref{eqWorkItOut1}) thus yields
	\begin{equation}
	\mu_{b\ra x}^{\brk{\ell-1}}(-\sign(x))\geq1-2^{1-|N(b)|}(1-\delta k\theta)\geq0.99.
	\label{eqWorkItOut3}
	\end{equation}
Using the elementary inequality $-z-z^2\leq\ln(1-z)\leq-z$ for $0\leq z\leq0.5$, we obtain from~(\ref{eqWorkItOut1}), (\ref{eqWorkItOut2})
	and~(\ref{eqWorkItOut3})
	\begin{eqnarray*}
	\ln\mu_{b\ra x}^{\brk{\ell-1}}(-\sign(x))&\leq&-2^{1-|N(b)|}(1-\alpha_b)\leq-2^{1-|N(b)|}(1-8k\theta\delta),\\
	\ln\mu_{b\ra x}^{\brk{\ell-1}}(-\sign(x))&\geq&-2^{1-|N(b)|}(1-\alpha_b)-2^{2(1-|N(b)|)}(1-\alpha_b)^2\\
			&\geq&-2^{1-|N(b)|}(1+8k\theta\delta)\qquad\qquad\mbox{[as $|N(b)|\geq0.1k\theta$ by {\bf H1}]}.
	\end{eqnarray*}
Summing these bounds up for $b\in\cN$, we obtain
	\begin{eqnarray}
	\abs{\ln\prod_{b\in\cN}
		\frac{\mu_{b\ra x}^{\brk{\ell-1}}(1)}{\mu_{b\ra x}^{\brk{\ell-1}}(-1)}}
	&\leq&\abs{\sum_{b\in\cN}\sign(x,b)2^{1-|N(b)|}}+8k\delta\sum_{b\in\cN}2^{1-|N(b)|}\nonumber\\
	&\leq&2\abs{\sum_{b\in\cN}\sign(x,b)2^{-|N(b)|}}+2(k\theta)^{-3}\qquad\mbox{[by {\bf H1}]}\nonumber\\
	&\leq&2\abs{\sum_{b\in N_{\leq1}(x,T\brk{\ell-1})}\sign(x,b)2^{-|N(b)|}}+2(k\theta)^{-3}+\sum_{x\in N_1(x,T\brk{\ell-1})}2^{1-|N(b)|}\nonumber\\
	&\leq&0.02\delta+2(k\theta)^{-3}+\rho(\theta k)^5\delta\qquad\qquad\quad\mbox{[by {\bf H2}, {\bf H4}]}\nonumber\\
	&\leq&10^{-6}\qquad\qquad\qquad\qquad\mbox{[because $\delta=\exp(-ck\theta)$ and $k\theta\geq\ln(\rho)/c^2]$}.
		\label{eqWorkItOut20}
	\end{eqnarray}

Plugging~(\ref{eqWorkItOut10}) and~(\ref{eqWorkItOut20}) into~(\ref{eqWorkItOut0}), we see that
	$\abs{P_{\leq1}^{\brk\ell}\bc{x\ra a}-1}\leq10^{-5}$, while $\abs{P_{\leq1}^{\brk\ell}\bc{x\ra a}-1}\leq10^{-4}$ by~(\ref{eqWorkItOut}).
Therefore, (\ref{eqidea}) yields
	$$\abs{1-\frac{1+2\Delta_{x\ra a}^{\brk\ell}}{1-2\Delta_{x\ra a}^{\brk\ell}}}=
		\abs{1-\frac{\mu_{x\ra a}^{\brk\ell}(1)}{\mu_{x\ra a}^{\brk\ell}(-1)}}\leq3\cdot 10^{-4},$$
whence $\abs{\Delta_{x\ra a}^{\brk\ell}}\leq0.01$, as desired.
\qed\end{proof}

\begin{corollary}\label{Cor_clause2var}
Let $b$ be a clause such that $N(b)\not\subset T\brk\ell$.
Let $x\in N(b)$.
Assume that $B\brk{\ell-1}\subset T\brk{\ell-1}$.
Then
	$$\mu_{b\ra x}^{\brk{\ell-1}}(-1)>0\mbox{ and }
	\abs{\frac{\mu_{b\ra x}^{\brk{\ell-1}}(1)}{\mu_{b\ra x}^{\brk{\ell-1}}(-1)}-1}\leq\exp\bc{-k_1/2}.$$
\end{corollary}
\begin{proof}
We consider two cases.
\begin{description}
\item[Case 1: $\abs{N(b)\setminus T\brk{\ell-1}}>k_1$.]
		Since $N(b)\not\subset T\brk{\ell}$, we have $|N(b)|\leq10 k\theta$ (by {\bf T2a}).
		Therefore, \Lem~\ref{Lemma_P1T} yields
	$$\exp(-\exp(-0.6k_1))\leq\exp\brk{-2^{3-k_1}\exp(\delta|N(b)|)}\leq
		\mu_{b\ra x}^{\brk{\ell}}(\zeta)\leq1\quad\mbox{ for $\zeta=\pm1$},$$
	whence  the assertion follows.
\item[Case 2: $\abs{N(b)\setminus T\brk{\ell-1}}\leq k_1$.]
	Since $N(b)\not\subset T\brk\ell$, condition {\bf T2a} ensures that $0.1\theta k\leq|N(b)|\leq10\theta k$.
		The assumption $N(b)\not\subset T\brk\ell$ implies that $b\not\in T_3\brk\ell$.
		But since $|N(b)\setminus T\brk{\ell-1}|\leq k_1$, and as $|N(b)|\geq0.1\theta k\geq100k_1$,
		the only possible reason why $b\not\in T_3\brk\ell$ is that $b\in T_3\brk{\ell-1}$ (cf.\ the definition of $T_3\brk\ell$).
		As $N(b)\not\subset T_2\brk\ell$, {\bf T2e} implies
			\begin{equation}\label{eqB22}
			\abs{N(b)\cap H\brk{\ell-1}}\geq 3|N(b)|/4.
			\end{equation}
		Let $J=N(b)\cap H\brk{\ell-1}$.
		Since $b\in T_3\brk{\ell-1}$, we have $\ell\geq2$ and $|N(b)\setminus T\brk{\ell-2}|\leq k_1$.
		Therefore, \Cor~\ref{Cor_harmless} implies that $|\Delta_{y\ra b}|\leq0.01$ for all $y\in J$.
		Thus, for all $x\in N(b)$ we have
			\begin{eqnarray*}
			\mu_{b\ra x}^{\brk{\ell-1}}(-\sign(x,b))&=&1-\prod_{y\in N(b)\setminus\cbc x}\mu_{y\ra b}^{\brk{\ell-1}}(-\sign(y,b))\\
				&\geq&1-(0.501)^{|J|-1}\;\stacksign{(\ref{eqB22})}{\geq}\;1-(0.501)^{3|N(b)|/4 -1}\geq
					1-(0.501)^{0.07k\theta },
			\end{eqnarray*}
		Consequently,  $\mu_{b\ra x}^{\brk{\ell-1}}(-1)>0$ and 
			$$\abs{\frac{\mu_{b\ra x}^{\brk{\ell-1}}(1)}{\mu_{b\ra x}^{\brk{\ell-1}}(-1)}-1}
				\leq2\cdot(0.501)^{0.07k\theta }\leq\exp\bc{-\theta k/100}\leq\exp(-k_1).$$
\end{description}
Thus, we have established the assertion in either case.
\qed\end{proof}

\medskip\noindent
\emph{Proof of \Prop~\ref{Prop_TB}.}
We proceed by induction on $\ell$.
Since $B\brk0=\emptyset$ the assertion is trivial for $\ell=0$.
Thus, assume that $\ell\geq0$ and that $B\brk \ell\subset T\brk\ell$.
Let $x\in V_t\setminus T\brk{\ell+1}$.
We will prove that $x\not\in B\brk{\ell+1}$.
\Cor~\ref{Cor_clause2var} implies that
	\begin{equation}\label{eqTB1}
	\mu_{a\ra x}^{\brk{\ell}}(-1)>0\mbox{ and }
	\abs{\frac{\mu_{a\ra x}^{\brk{\ell}}(1)}{\mu_{a\ra x}^{\brk{\ell}}(-1)}-1}\leq\exp\bc{-k_1/2}
			\quad\mbox{for all $x\not\in T\brk{\ell+1}$, $a\in N(x)$.}
	\end{equation}
We claim
	\begin{equation}\label{eqTB3}
	|P_{>1}^{\brk{\ell+1}}\bc{x\ra a}-1|\leq\delta/100\quad\mbox{for all  $x\not\in T\brk{\ell+1}$, $a\in N(x)$}.
	\end{equation}
To establish~(\ref{eqTB3}), we consider two cases.
\begin{description}
\item[Case 1: $x\not\in N(T_3\brk\ell)$.]
	Let $\cT=N_{>1}^{\brk{\ell+1}}\bc{x\ra a}$ be the set of all clauses $b$ that contribute to the product $P_{>1}^{\brk{\ell+1}}\bc{x\ra a}$.
	Since $x\not\in N(T_3\brk\ell\cup T_3\brk{\ell+1})$, none of the clauses $b\in\cT$ features more than $|N(b)|-k_1$ variables from $T\brk\ell$
	(just by the definition of $T_3\brk{\ell+1}$).
	Furthermore, because $x\not\in T_2\brk{\ell+1}$, {\bf T2c} is not satisfied and thus we obtain the bound
		\begin{equation}\label{eqTB4}
		\sum_{b\in\cT}2^{\abs{N(b)\cap T\brk\ell\setminus\cbc x}-|N(b)|}
			\leq\sum_{b\in N_{>1}(x,T\brk\ell)}2^{\abs{N(b)\cap B\brk\ell\setminus\cbc x}-|N(b)|}
			\leq\delta/(\theta k)\leq\delta/10^4.
		\end{equation}
	Since $x\not\in T\brk{\ell+1}$, {\bf T2a} ensures that $|N(b)|\leq10\theta k$ for all $b\in\cT$.
	Therefore, (\ref{eqTB3}) follows  from (\ref{eqTB4}) and \Cor~\ref{Cor_P1T}.
\item[Case 2: $x\in N(T_3\brk\ell)$.]
	Let $\cT=N_{>1}^{\brk{\ell+1}}\bc{x\ra a}\setminus T_3\brk{\ell}$ be the set of
	all clauses $b$ that occur in the product  $P_{>1}^{\brk{\ell+1}}\bc{x\ra a}$, apart from the ones in $T_3\brk{\ell}$.
	Since $x\not\in T_2\brk{\ell+1}\cup N(T_3\brk{\ell+1})$, this set $\cT$ also satisfies~(\ref{eqTB4}).
	Thus, \Cor~\ref{Cor_P1T} yields
		\begin{eqnarray}\label{eqTB5}
		\abs{\ln\prod_{b\in\cT}\frac{\mu_{b\ra x}^{\brk{\ell+1}}(1)}{\mu_{b\ra x}^{\brk{\ell+1}}(-1)}}
				&\leq&\delta/10^3.
		\end{eqnarray}
	Let $\cT'=N_{>1}^{\brk{\ell+1}}\bc{x\ra a}\cap T_3\brk{\ell}$.
	As condition {\bf T2d} 
	 ensures that $\abs{\cT'}\leq\abs{N(x)
			\cap T_3\brk\ell}\leq100$,
	(\ref{eqTB1}) implies 
		\begin{eqnarray}\label{eqTB555}
		\abs{\ln\prod_{b\in\cT'}\frac{\mu_{b\ra x}^{\brk{\ell}}(1)}{\mu_{b\ra x}^{\brk{\ell}}(-1)}}&\leq&2\abs{\cT'}\exp(-k_1/2)\leq\delta/1000.
		\end{eqnarray}
	Since $N_{>1}^{\brk{\ell+1}}\bc{x\ra a}=\cT\cup\cT'$, (\ref{eqTB5}) and~(\ref{eqTB555}) yield
		$\abs{1-P_{>1}^{\brk{\ell+1}}\bc{x\ra a}}\leq\delta/100.$
\end{description}
Thus, we have established~(\ref{eqTB3}) in either case.

If $x\not\in T_1\brk{\ell+1}\subset T\brk{\ell+1}$ and $a\in N(x)$, then $\abs{P_{\leq1}^{\brk{\ell+1}}\bc{x\ra a}-1}\leq\delta/100$.
Hence, (\ref{eqTB1}) implies that for all $x\not\in T\brk{\ell+1}$ and all $a\in N(x)$ we have $\mu_{x\ra a}^{\brk{\ell+1}}(-1)>0$.
Thus, 
	\begin{eqnarray*}\label{eqTB6}
	\abs{1-\frac{\mu_{x\ra a}^{\brk{\ell+1}}(1)}{\mu_{x\ra a}^{\brk{\ell+1}}(-1)}}
		=\abs{1-P_{\leq1}^{\brk{\ell+1}}\bc{x\ra a}\cdot P_{>1}^{\brk{\ell+1}}\bc{x\ra a}}\leq\delta/99
			\quad\mbox{[by~(\ref{eqTB3})].}
	\end{eqnarray*}
Consequently, $\abs{\Delta_{x\ra a}^{\brk{\ell+1}}}<0.1\delta$, and thus $x\not\in B\brk{\ell+1}$.
\qed

\subsection{Proof of \Prop~\ref{Prop_Tbound}}\label{Sec_Tbound}

We are going to proceed by induction on $\ell$.
We begin by bounding the sizes of the sets $T_2,T_3$.

\begin{lemma}\label{Lemma_T3bound}
Assume that $\abs{T_1\brk\ell\cup T_2\brk\ell}\leq\delta\theta n/3$ and $\abs{N(T_3\brk\ell)}\leq\delta\theta n/2$.
Then $\abs{N(T_3\brk{\ell+1})}\leq\delta\theta n/2$.
\end{lemma}
\begin{proof}
By construction we have $T_3\brk\ell\cap T_3\brk{\ell+1}=\emptyset$ (cf.~(\ref{eqDefT3})).
Furthermore, also by construction $N(T_3\brk\ell)\subset T\brk\ell$, and
each clause in $T_3\brk{\ell+1}$ has at least a $0.99$-fraction of its variables in $T\brk\ell$.
Thus, $|N(b)\cap T\brk\ell|\geq0.99|N(b)|$ for all $b\in T_3\brk{\ell}\cup T_3\brk{\ell+1}$.
Hence, {\bf Q3} yields
	\begin{eqnarray*}
	\abs{N(T_3\brk{\ell+1})}+\abs{N(T_3\brk{\ell})}&\leq&
		\sum_{b\in T_3\brk{\ell}\cup T_3\brk{\ell+1}}|N(b)|\\
	&\leq&\frac{1.01}{0.99}|T\brk\ell|
	\leq1.03\bc{\abs{T_1\brk\ell}+\abs{T_2\brk\ell}+\abs{N(T_3\brk\ell)}}.
	\end{eqnarray*}
Hence, $\abs{N(T_3\brk{\ell+1})}\leq1.03\bc{\abs{T_1\brk\ell}+\abs{T_2\brk\ell}}+0.03\abs{N(T_3\brk\ell)}\leq\theta\delta n/2.$
\qed\end{proof}

\begin{lemma}\label{Lemma_T2bound}
Assume that $\abs{T_1\brk{\ell}\cup T_2\brk{\ell}}\leq\delta\theta n/3$ and
$\abs{N(T_3\brk{\ell})}\leq\delta\theta n/2$.
Moreover, suppose that $\abs{T\brk{\ell-1}}\leq\delta\theta n$.
Then $\abs{T_2\brk{\ell+1}}\leq\delta\theta n/6$.
\end{lemma}
\begin{proof}
Conditions {\bf Q0} and {\bf Q1} readily imply that the number
of variables  that satisfy either {\bf T2a} or {\bf T2b} is $\leq0.001\theta\delta n$.
Moreover, we apply {\bf Q2} to the set $T\brk{\ell}$ of size
	\begin{equation}\label{eqT2bound1}
	\abs{T\brk{\ell}}\leq\abs{T_1\brk{\ell}\cup T_2\brk{\ell}}+\abs{N(T_3\brk{\ell})}\leq0.9\delta\theta n
	\end{equation}
to conclude that the number of variables satisfying {\bf T2c} is $\leq0.001\theta\delta n$ as well.

To bound the number of variables that satisfy {\bf T2d}, consider the subgraph of the factor graph induced on $T_3\brk{\ell}\cup N(T_3\brk{\ell})$.
For each $x\in N(T_3\brk{\ell})$ let $D_x$ be the number of neighbors of $x$ in $T_3\brk{\ell}$.
Let $\nu$ be the set of all $x\in V_t$ so that $D_x\geq100$.
Then {\bf Q3} yields
	\begin{eqnarray*}
	100\nu&\leq&\sum_{x\in N(T_3\brk\ell)}D_x=\sum_{a\in T_3\brk\ell}\abs{N(a)}\leq1.01|T\brk{\ell}|\leq\theta\delta n
	\quad\mbox{[as $N(b)\subset T\brk\ell$ for all $b\in T_3\brk\ell$]}.
	\end{eqnarray*}
Hence, there are at most $\nu\leq0.01\theta\delta n$ variables that satisfy {\bf T2d}.
In summary, we have shown that
	\begin{equation}\label{eqT2bound5}
	\abs{\cbc{x\in V:x\mbox{ satisfies one of {\bf T2a}--{\bf T2d}}}}\leq0.015\theta\delta n.
	\end{equation}

To deal with {\bf T2e}, 
observe that if a clause $a$ has at least $|N(a)|/4$ variables that are \emph{not} harmless, then one of the following statements is true.
\begin{enumerate}
\item[i.] $a$ contains at least $|N(a)|/20$ variables $x$ that violate either {\bf H1}, {\bf H2}, or {\bf H4}.
\item[ii.] $a$ contains at least $|N(a)|/5$ variables $x$ that violate condition {\bf H3}.
\end{enumerate}
Let $\cC_1$ be the set of clauses $a$ for which i.\ holds, and let $\cC_2$ be the set of clauses satisfying ii.,
so that the number of variables satisfying {\bf T2e} is bounded by $\sum_{a\in\cC_1\cup\cC_2}|N(a)|$.

To bound $\sum_{a\in\cC_1}|N(a)|$, let $Q$ be the set of all variables $x$ that violate either {\bf H1}, {\bf H2}, or {\bf H4} at time $\ell$.
Then conditions {\bf Q1} and {\bf Q2} entail that $\abs Q\leq3\cdot10^{-4}\theta\delta n$
(because we are assuming $\abs{T(\ell-1)}\leq\theta\delta n$).
Therefore, condition {\bf Q3} implies that
	\begin{equation}\label{eqT2bound4}
	\sum_{a\in\cC_1}|N(a)|\leq21\abs{Q}+10^{-4}\delta\theta n\leq0.0064\delta \theta n.
	\end{equation}

To deal with $\cC_2$ let $\cB'$ be the set of all clauses $b$ 
such that $|N(b)|\geq100k_1$ but $|N(b)\setminus T\brk\ell|\leq k_1$.
Since we know from~(\ref{eqT2bound1}) that $\abs{T\brk\ell}\leq\delta\theta n$,
condition {\bf Q3} applied to $T\brk{\ell}$ implies
	\begin{equation}\label{eqT2bound3'}
	\abs{N(\cB')}\leq 
		\sum_{b\in \cB'}|N(b)|\leq1.03\abs{T\brk{\ell}}+10^{-4}\delta\theta n\leq1.0301\delta \theta n.
		\end{equation}
In addition, let $\cB''$ be the set of all clauses of length less than $100k_1$.
Since $100k_1=100\sqrt c\theta k\leq0.1\theta k$ by our choice of $c$,
{\bf Q1} implies that
$|N(\cB'')|\leq10^{-4}\delta\theta n$.
Hence, (\ref{eqT2bound3'}) shows that $\cB=\cB'\cup\cB''$ satisfies
	\begin{equation}\label{eqT2bound3}
	\abs{N(\cB)}\leq 1.0302\delta \theta n.
	\end{equation}
Furthermore, let $\cU$ be the set of all clauses $a$ such that $N(a)\subset N(\cB)$.
Let $U$ be the set of variables $x\in N(\cB)$ that occur in at least two clauses from $\cU$.
Then by  {\bf Q3}
	\begin{eqnarray*}
	\abs U+\abs{N(\cB)}&\leq&\sum_{a\in\cU}|N(a)|\leq1.01|N(\cB)|+10^{-4}\delta\theta n,
	\end{eqnarray*}
whence $|U|\leq0.01|N(\cB)|+10^{-4}\delta\theta n\leq0.02\delta \theta n$ due to~(\ref{eqT2bound3}).
Since $\cB\subset\cU$, the set $U$ contains all variables that occur in at least two clauses from $\cB$,
i.e., all variables that violate condition {\bf H3}.
Therefore, any $a\in\cC_2$ contains at least $|N(a)|/5$ variables from $U$.
Applying {\bf Q3} once more, we obtain
	\begin{eqnarray*}
	\sum_{a\in\cC_2}|N(a)|&\leq&5.05\cdot0.02\delta \theta n+10^{-4}\delta\theta n=0.1201\delta \theta n.
	\end{eqnarray*}
Combining this estimate with the bound~(\ref{eqT2bound4}) on $\cC_1$, 
we conclude that the number of variables satisfying {\bf T2e} is bounded
by $\sum_{a\in\cC_1\cup\cC_2}|N(a)|\leq0.127\delta \theta n.$
Together with~(\ref{eqT2bound5}) this yields the assertion.
\qed\end{proof}

In \Sec~\ref{Sec_T1bound} we will derive the following bound on $\abs{T_1\brk{\ell+1}}$.

\begin{proposition}\label{Prop_T1bound}
If $\abs{T\brk\ell}\leq\delta \theta n$, then
$\abs{T_1\brk{\ell+1}\setminus T_2\brk{\ell+1}}\leq\delta \theta n/6$.
\end{proposition}

\begin{proof}[\Prop~\ref{Prop_Tbound}]
We are going to show that
	\begin{equation}\label{eqTbound1}
	\abs{T_1\brk{\ell}\cup T_2\brk{\ell}}\leq\delta \theta n/3\mbox{ and }\abs{N(T_3\brk{\ell})}\leq\delta \theta n/2
	\end{equation}
for all $\ell\geq0$.
This implies that
	$\abs{T\brk{\ell}}
		\leq\delta \theta n$
for all $\ell\geq0$, as desired.

In order to prove~(\ref{eqTbound1}) we proceed by induction on $\ell$.
The bounds for $\ell=0$ are immediate from {\bf Q0} and {\bf Q1}.
Now assume that~(\ref{eqTbound1}) holds for all $l\leq\ell$.
Then \Lem~\ref{Lemma_T3bound} shows that $\abs{N(T_3\brk{\ell})}\leq\delta\theta n/2$.
Moreover, \Lem~\ref{Lemma_T2bound} applies (with the convention that $T\brk{-1}=T\brk{0}$), giving $\abs{T_2\brk{\ell+1}}\leq\delta\theta n/6$.
Finally, \Prop~\ref{Prop_T1bound} shows  $\abs{T_1\brk{\ell+1}\setminus T_2\brk{\ell+1}}\leq\delta\theta n/6$,
whence $\abs{T_1\brk{\ell+1}\cup T_2\brk{\ell+1}}\leq\delta\theta n/3$.
\qed\end{proof}

\subsection{Proof of \Prop~\ref{Prop_T1bound}}\label{Sec_T1bound}

\emph{Throughout this section we assume that $\abs{T\brk\ell}\leq\delta \theta n$.}

For a variable $x\in V_t$ and $a\in N(x)$ we let
	\begin{eqnarray*}
	\sigma_{x\ra a}^{\brk{\ell+1}}&=&\sum_{b\in N_{\leq1}^{\brk{\ell+1}}\bc{x\ra a}}2^{1-|N(b)|}\sign\bc{x,b},\\
	\xi_{x\ra a}^{\brk{\ell+1}}&=&
			\sum_{b\in N_{\leq1}^{\brk{\ell+1}}\bc{x\ra a}}\sum_{y\in N(b)\setminus\cbc x}2^{1-|N(b)|}\sign\bc{x,b}\sign\bc{y,b}
					\Delta_{y\ra b}^{\brk{\ell}},\mbox{ and }\\
	L_{x\ra a}^{\brk{\ell+1}}&=&\sigma_{x\ra a}^{\brk{\ell+1}}+\xi_{x\ra a}^{\brk{\ell+1}}.
	\end{eqnarray*}
In \Sec~\ref{Sec_P0} we are going to establish the following.

\begin{proposition}\label{Prop_P0}
For any variable $x\not\in T_2\brk{\ell+1}$ and any clause $a\in N(x)$ we have
	$$\abs{L_{x\ra a}^{\brk{\ell+1}}+P_{\leq1}^{\brk{\ell+1}}
			}\leq10^{-3}\delta.$$
\end{proposition}

\begin{lemma}\label{Lemma_sigma}
For all but at most $10^{-4}\delta\theta n$ variables $x\in V\setminus T_2\brk{\ell+1}$
we have
	$$\max_{a\in N(x)}\abs{\sigma_{x\ra a}^{\brk{\ell+1}}}\leq0.003\delta.$$
\end{lemma}
\begin{proof}
Applying {\bf Q2} to $Q=T\brk\ell$, we find that for all but $10^{-4}\delta\theta n$ variables $x\in V_t$ we have
	\begin{eqnarray}\label{eqLemmasigma1}
	\abs{\sum_{b\in N_{\leq1}(x,T\brk\ell)}2^{1-|N(b)|}\sign(x,b)}&\leq&2\cdot 10^{-3}\delta.
	\end{eqnarray}
Assume that $x$ satisfies~(\ref{eqLemmasigma1}) and that $x\not\in T_2\brk{\ell+1}$.
Let $a\in N(x)$.
Since $N_{\leq1}^{\brk{\ell+1}}(x\ra a)=N_{\leq1}(x,T\brk\ell)\setminus\cbc a$, we obtain
	\begin{eqnarray*}
	\abs{\sigma_{x\ra a}^{\brk{\ell+1}}}&\leq&\abs{\sum_{b\in N_{\leq1}(x,T\brk\ell)}2^{1-|N(b)|}\sign(x,b)}+2^{1-|N(a)|}\leq2\cdot10^{-3}\delta+2^{1-|N(a)|}\\
		&\leq&2\cdot10^{-3}\delta+\exp(-0.1\theta k)\leq0.003\delta\qquad\mbox{[as $|N(a)|\geq0.1\theta k$ due to {\bf T2a}]},
	\end{eqnarray*}
as desired.
\qed\end{proof}

\begin{lemma}\label{Lemma_B2}
Let $x$ be a variable and let $b_1,b_2\in N(x)$
be such that $|N(b_i)\cap T\brk\ell|\leq2$ and $|N(b_i)|\geq0.1\theta k$ for $i=1,2$.
Then
	$\abs{\Delta_{x\ra b_1}^{\brk\ell}-\Delta_{x\ra b_2}^{\brk\ell}}\leq\delta^3.$
\end{lemma}
\begin{proof}
By \Prop~\ref{Prop_TB} we have $B\brk{\ell-1}\subset T\brk{\ell-1}$.
Furthermore, our assumptions ensure that $N(b_i)\setminus T\brk\ell\neq\emptyset$.
Hence, 
\Cor~\ref{Cor_clause2var}
yields
	\begin{equation}\label{eqB21}
	\mu_{b_i\ra x}^{\brk{\ell-1}}(-1)>0\mbox{ and }
	\abs{\frac{\mu_{b_i\ra x}^{\brk{\ell-1}}(1)}{\mu_{b_i\ra x}^{\brk{\ell-1}}(-1)}-1}\leq\exp\bc{-k_1/2}\leq\delta^6
	\end{equation}
for $i=1,2$.
There are three cases.
\begin{description}
\item[Case 1: there is $c\in N(b)\setminus\cbc{b_1,b_2}$ such that $\mu_{c\ra x}^{\brk{\ell-1}}(1)=0$.]
	Then (\ref{eqBPiaell}) shows that $$\mu_{x\ra b_1}^{\brk\ell}(1)=\mu_{x\ra b_2}^{\brk\ell}(1)=0.$$
	Thus,
		$\Delta_{x\ra b_1}^{\brk\ell}=\Delta_{x\ra b_2}^{\brk\ell}=-1/2$.
\item[Case 2: there is $c\in N(b)\setminus\cbc{b_1,b_2}$ such that $\mu_{c\ra x}^{\brk{\ell-1}}(-1)=0$.]
	Similarly as in Case~1, (\ref{eqBPiaell}) implies $\mu_{x\ra b_i}^{\brk\ell}(-1)=0$ for $i=1,2$.
	Since $\mu_{x\ra b_i}^{\brk\ell}(-1)+\mu_{x\ra b_i}^{\brk\ell}(1)=1$, we thus obtain 
		$\Delta_{x\ra b_1}^{\brk\ell}=\Delta_{x\ra b_2}^{\brk\ell}=1/2$.
\item[Case 3: for all $c\in N(b)\setminus\cbc{b_1,b_2}$ we have $0<\mu_{c\ra x}^{\brk{\ell-1}}(1)<1$.]
	Then (\ref{eqBPiaell}) yields $0<\mu_{x\ra b_i}^{\brk\ell}(-1)<1$ for $i=1,2$.
	Therefore, we can define
		\begin{equation}\label{eqDefqi}
		q_i=\frac{\mu_{x\ra b_i}^{\brk{\ell}}(1)}{\mu_{x\ra b_i}^{\brk{\ell}}(-1)}
			=\frac{\prod_{b\in N(x)\setminus\cbc{b_i}}\mu_{b\ra x}^{\brk{\ell-1}}(1)}{\prod_{b\in N(x)\setminus\cbc{b_i}}\mu_{b\ra x}^{\brk{\ell-1}}(-1)}>0.
		\end{equation}
	Unravelling~(\ref{eqBPiaell}), we see that
		\begin{eqnarray}\nonumber
		\mu_{x\ra b_i}^{\brk{\ell}}(1)&=&
			\frac{\prod_{b\in N(x)\setminus\cbc{b_i}}\mu_{b\ra x}^{\brk{\ell-1}}(1)}{
				\prod_{b\in N(x)\setminus\cbc{b_i}}\mu_{b\ra x}^{\brk{\ell-1}}(1)
			+\prod_{b\in N(x)\setminus\cbc{b_i}}\mu_{b\ra x}^{\brk{\ell-1}}(-1)}\\
			&=&\frac{\prod_{b\in N(x)\setminus\cbc{b_i}}\mu_{b\ra x}^{\brk{\ell-1}}(1)}
				{(1+1/q_i)\prod_{b\in N(x)\setminus\cbc{b_i}}\mu_{b\ra x}^{\brk{\ell-1}}(1)}
					=\frac{q_i}{q_i+1}.
					\label{eqqi}
		\end{eqnarray}
	Hence,
	\begin{eqnarray}\nonumber
	\abs{\Delta_{x\ra b_1}^{\brk{\ell}}-\Delta_{x\ra b_2}^{\brk{\ell}}}
		&=&\abs{\mu_{x\ra b_1}^{\brk{\ell}}(1)-\mu_{x\ra b_2}^{\brk{\ell}}(1)}\\
		&=&\abs{\frac{q_1-q_2}{(1+q_1)(1+q_2)}}\qquad\mbox{[by (\ref{eqqi})]}\nonumber\\
		&=&\abs{\frac{1-q_2/q_1}{(1+1/q_1)(1/q_1+q_2/q_1)}}\leq\frac{q_1}{q_2}
			\abs{1-\frac{q_2}{q_1}}\quad\mbox{[as $q_1,q_2>0$]}.
	\label{eqTB7}
	\end{eqnarray}
	Furthermore, by the definition~(\ref{eqDefqi}) of $q_1,q_2$, we have
		\begin{eqnarray*}
		\frac{q_2}{q_1}&=&\frac{\mu_{b_1\ra x}^{\brk{\ell-1}}(1)}{\mu_{b_1\ra x}^{\brk{\ell-1}}(-1)}\cdot
			\frac{\mu_{b_2\ra x}^{\brk{\ell-1}}(-1)}{\mu_{b_2\ra x}^{\brk{\ell-1}}(1)}.
		\end{eqnarray*}
	Hence, (\ref{eqB21}) yields $|1-\frac{q_2}{q_1}|\leq\delta^5$, and thus
	the desired bound on $\abs{\Delta_{x\ra b_1}^{\brk{\ell}}-\Delta_{x\ra b_2}^{\brk{\ell}}}$ follows from~(\ref{eqTB7}).
\end{description}
Hence, we have established the desired bound in all cases.
\qed\end{proof}

\begin{lemma}\label{Lemma_xi}
For all but at most 
	$0.1\delta\theta n$
variables $x\not\in T_2\brk{\ell+1}$ we have $$\max_{a\in N(x)}\abs{\xi_{x\ra a}^{\brk{\ell+1}}}\leq0.001\delta.$$
\end{lemma}
\begin{proof}
For a variable $y$ let $\cN(y)$ be the set of all clauses $b\in N(y)$ such that $b\in N_{\leq1}(x,T\brk\ell)$ for some variable $x\in V_t$.
If $\cN(y)=\emptyset$ we define $\Delta_y=0$; otherwise select $a_y\in\cN(y)$
arbitrarily and set $\Delta_y=\Delta_{y\ra a_y}^{\brk\ell}$.
Thus, we obtain a vector $\Delta=(\Delta_y)_{y\in V}$ with norm
	$\norm\Delta_\infty\leq\frac12$.
Let $\Xi=(\xi_x)_{x\in V_t}=\Lambda_{T\brk\ell}\Delta$,
where $\Lambda_{T\brk\ell}$ is the linear operator from condition {\bf Q4} in Definition~\ref{Def_reasonable}.
That is, for any $x\in V$ we have
	$$\xi_{x}=\sum_{b\in N_{\leq1}(x,T\brk\ell)}
		\sum_{y\in N(b)\setminus\cbc x}2^{-|N(b)|}\sign(x,b)\sign(y,b)\Delta_{y}.$$
Because  $\abs{T\brk\ell}\leq\delta \theta n$,
condition {\bf Q4} ensures that $\cutnorm{\Lambda_{T\brk\ell}}\leq \delta^4\theta n$.
Consequently,
	\begin{equation}\label{eqXi1}
	\norm\Xi_1=\norm{\Lambda_{T\brk\ell}\Delta}_1\leq\cutnorm{\Lambda_{T\brk\ell}}\norm\Delta_\infty
		\leq \delta^4\theta n.
	\end{equation}
Since $\norm\Xi_1=\sum_{x\in V}\abs{\xi_x}$, (\ref{eqXi1}) implies that
	\begin{equation}\label{eqXi2}
	\abs{\cbc{x\in V:\abs{\xi_x}>\delta^{2}}}\leq	\delta^{2}\theta n.
	\end{equation}

To infer the lemma from (\ref{eqXi2}), we need to establish a relation between $\xi_x$ and $\xi_{x\ra a}$ for $x\not\in T_2\brk{\ell+1}$
and $a\in N(x)$.
Since $\cN(y)\subset N_{\leq1}(x,T\brk\ell)$ for any $y\in V_t$, we see that $|N(b)\cap T\brk\ell|\leq2$ for all $b\in\cN(y)$.
Furthermore, as $x\not\in T_2\brk{\ell+1}$, we have $|N(b)|\geq0.1\theta k$ for all $b\in\cN(y)$ (by {\bf T2a}).
Consequently, \Lem~\ref{Lemma_B2} applies to $b\in\cN(y)$, whence
	$\abs{\Delta_{y\ra b}^{\brk\ell}-\Delta_{y\ra b'}^{\brk\ell}}\leq\delta^3$
		for all $y\in V_t,\,b,b'\in\cN(y)$. 
Hence, 
	\begin{eqnarray}\label{eqxixa2}
	\abs{\Delta_{y\ra b}^{\brk\ell}-\Delta_y}\leq\delta^3\qquad\mbox{for all $y\in V_t,b\in\cN(y)$}.
	\end{eqnarray}
Consequently, we obtain for  $x\not\in T_2\brk{\ell+1}$
	\begin{eqnarray}\nonumber
	\max_{a\in N(x)}\abs{2\xi_x-\xi_{x\ra a}^{\brk{\ell+1}}}
		\\
		&\hspace{-6cm}=&\hspace{-3cm}\max_{a\in N(x)}\bigg|
			\vecone_{a\in N_{\leq1}(x,T\brk\ell)}\cdot\sum_{y\in N(a)\setminus\cbc x}
				2^{1-|N(b)|}\sign\bc{x,b}\sign\bc{y,b}\Delta_y
			\nonumber\\
		&\hspace{-6cm}&\hspace{-3cm}\qquad\qquad
			+\sum_{b\in N_{\leq1}^{\brk{\ell+1}}\bc{x\ra a}}\sum_{y\in N(b)\setminus\cbc x}2^{1-|N(b)|}\sign\bc{x,b}\sign\bc{y,b}
				(\Delta_y-\Delta_{y\ra b}^{\brk{\ell}})\bigg|\nonumber\\
		&\hspace{-6cm}\leq&\hspace{-3cm}
			\sum_{y\in N(a)\setminus\cbc x}2^{1-|N(a)|}\abs{\Delta_{y}}+\hspace{-4mm}
			\sum_{b\in N_{\leq1}^{\brk{\ell+1}}\bc{x\ra a}}\sum_{y\in N(b)\setminus\cbc x}2^{1-|N(b)|}\abs{\Delta_{y\ra a}^{\brk\ell}-\Delta_{y}}\nonumber\\
		&\hspace{-6cm}\stacksign{(\ref{eqxixa2})}{\leq}&\hspace{-3cm}\;
			|N(a)|2^{-|N(a)|}
				+\delta^3\sum_{b\in N\bc{x}}|N(b)|2^{1-|N(b)|}
			\nonumber\\
		&\hspace{-6cm}\leq&\hspace{-3cm}\;10k\theta2^{-0.1k\theta}+10\delta^3k\theta\sum_{b\in N\bc{x}}2^{1-|N(b)|}
					\qquad\mbox{[as $0.1k\theta\leq|N(b)|\leq10k\theta$  by {\bf T2a}]}\nonumber\\
		&\hspace{-6cm}\leq&\hspace{-3cm}\;
			10^{-4}\delta\qquad\qquad\qquad\qquad\qquad\qquad\qquad\quad\ \ 
			\mbox{[by {\bf T2b}]}.
			\label{eqxiaxi}
	\end{eqnarray}
If $x\not\in T_2\brk{\ell+1}$ is such that
$\abs{\xi_x}\leq\delta^2$, then (\ref{eqxiaxi}) implies that
	$\abs{\xi_{x\ra a}}\leq2\cdot10^{-4}\delta$
for any $a\in N(x)$.
Therefore, the assertion follows from~(\ref{eqXi2}).
\qed\end{proof}

\medskip\noindent\emph{Proof of \Prop~\ref{Prop_T1bound}.}
Let $S$ be the set of all variables $x\not\in T_2\brk{\ell+1}$ such that
$\max_{a\in N(x)}|\sigma_{x\ra a}^{\brk{\ell+1}}|\leq 0.003\delta$
and $\max_{a\in N(i)}|\xi_{x\ra a}^{\brk{\ell+1}}|\leq0.001\delta$.
Then \Prop~\ref{Prop_P0} entails that for any $x\in S$ and $a\in N(x)$
	$$\abs{\ln P_{\leq1}^{\brk{\ell+1}}\bc{x\ra a}}\leq\abs{L_{x\ra a}^{\brk{\ell+1}}}+10^{-3}\delta
		\leq|\sigma_{x\ra a}^{\brk{\ell+1}}|+|\xi_{x\ra a}^{\brk{\ell+1}}|+10^{-3}\delta
			\leq0.005\delta.$$
Hence,
	$\abs{P_{\leq1}^{\brk{\ell+1}}\bc{x\ra a}-1}\leq0.01\delta$  for all $x\in S$, $a\in N(x)$,
and therefore
	$$T_1\brk{\ell+1}\setminus T_2\brk{\ell+1}\subset V_t\setminus(S\cup T_2\brk{\ell+1})).$$
Finally, \Lem s~\ref{Lemma_sigma} and~\ref{Lemma_xi} imply
	$\abs{T_1\brk{\ell+1}\setminus T_2\brk{\ell+1}}\leq\abs{V_t\setminus(S\cup T_2\brk{\ell+1})}\leq\delta\theta n/6.$
\qed

\subsection{Proof of \Prop~\ref{Prop_P0}}\label{Sec_P0}

We begin by approximating
$\ln(\mu_{b\ra x}^{\brk{\ell}}(1)/\mu_{b\ra x}^{\brk{\ell}}(-1))$ by a linear function.
In this section we let
 $O_{\rho}\bc\cdot$ denote an asymptotic bound that holds in the limit of large $\rho$.
 That is, $f(\rho)=O(g(\rho))$ if there exist $C>0$, $\rho_*>0$ such that
 $|f(\rho)|\leq C|g(\rho)|$ for $\rho>\rho_*$.

\begin{lemma}\label{Lemma_lin1}
Let $x\in V_t$, $a\in N(x)$, and $b\in N_{\leq1}^{\brk{\ell+1}}\bc{x\ra a}$.
Then $\mu_{b\ra x}^{\brk{\ell}}(-1)>0$ and
	\begin{eqnarray}\nonumber
	\ln\bcfr{\mu_{b\ra x}^{\brk{\ell}}(1)}{\mu_{b\ra x}^{\brk{\ell}}(-1)}&=&
		2^{1-|N(b)|}\brk{\sign\bc{x,b}+2\sum_{y\in N(b)\setminus\cbc x}\sign\bc{x,b}\sign\bc{y,b}\Delta_{y\ra b}^{\brk\ell}}\\
		&&\qquad+
			\frac{(\theta k\delta+|N(b)\cap T\brk\ell\setminus\cbc x|)}{2^{|N(b)|}}\cdot O_{\rho}(k\theta \delta).\label{eqlin1}
	\end{eqnarray}
\end{lemma}
\begin{proof}
The definition of the set $N_{\le1}^{\brk{\ell+1}}\bc{x\ra a}$ ensures that for all
$b\in N_{\le1}^{\brk{\ell+1}}\bc{x\ra a}$ we have $|N(b)\cap T\brk\ell|\leq2$, while $|N(b)|\geq0.1\theta k$.
Therefore, \Lem~\ref{Lemma_P1T} shows that $|\frac12-\mu_{b\ra x}^{\brk{\ell}}(-1)|\leq\delta^2$
	(recall from \Prop~\ref{Prop_TB} that $B\brk\ell\subset T\brk\ell$).
Furthermore, $b$ is not redundant, and thus not a tautology,
because otherwise $N(b)\subset T_2\brk\ell$ due to {\bf T2a}.

Let $s=\sign\bc{x,b}$.
Then $\mu_{b\ra x}^{\brk{\ell}}(s)=1$ and thus the definition~(\ref{eqBPai}) of the messages 
$\mu_{b\ra x}^{\brk{\ell}}(\pm1)$ yields
	\begin{eqnarray}
	\frac{\mu_{b\ra x}^{\brk{\ell}}(-s)}{\mu_{b\ra x}^{\brk{\ell}}(s)}&=&
		\mu_{b\ra x}^{\brk{\ell}}(-s)=
			1-\prod_{y\in N(b)\setminus\cbc x}\mu_{y\ra b}^{\brk{\ell}}(-\sign\bc{y,b})\nonumber\\
			&=&1-\prod_{y\in N(b)\setminus\cbc x}\frac12-\sign\bc{y,b}\Delta_{y\ra b}^{\brk{\ell}}\nonumber\\
			&=&1-2^{1-|N(b)|}\prod_{y\in N(b)\setminus\cbc x}1-2\sign\bc{y,b}\Delta_{y\ra b}^{\brk\ell}.
				\label{eqlin11}
	\end{eqnarray}
Let $\Gamma=N(b)\setminus(T\brk\ell\cup\cbc x)$.
As \Prop~\ref{Prop_TB} shows that $T\brk\ell\supset B\brk\ell$ contains all biased variables, we have
$\abs{\Delta_{y\ra b}^{\brk{\ell}}}\leq\delta$ for all $y\in \Gamma$.
Therefore, we can use the approximation $\abs{\ln(1-z)+z}\leq z^2$ 
for $|z|\leq\frac12$
to obtain
	\begin{eqnarray}\nonumber
	\abs{\brk{\ln\prod_{y\in \Gamma}1-2\sign\bc{y,b}\Delta_{y\ra b}^{\brk\ell}}+
		\sum_{y\in \Gamma}2\sign\bc{y,b}\Delta_{y\ra b}^{\brk\ell}}\\
		&\hspace{-10cm}=&\hspace{-5cm}\;\nonumber
			\abs{\sum_{y\in \Gamma}\ln\bc{1-2\sign\bc{y,b}\Delta_{y\ra b}^{\brk\ell}}+
		2\sign\bc{y,b}\Delta_{y\ra b}^{\brk\ell}}\\
		&\hspace{-10cm}\leq&\hspace{-5cm}\;
			4\sum_{y\in \Gamma}{\Delta_{y\ra b}^{\brk\ell}}^2
				\leq40\theta k\delta^2;
			\label{eqlin12pre}
	\end{eqnarray}
in the last step, we used that $|N(b)|\leq10\theta k$ for all $b\in N_{\leq1}^{\brk{\ell+1}}\bc{x\ra a}$.
Furthermore, $|\Gamma|\leq|N(b)|\leq10\theta k$ and $|\Delta_{y\ra b}^{\brk\ell}|\leq\delta$ for all $y\in\Gamma$.
Hence,
	\begin{equation}\label{eqNotnagel}
	\abs{\sum_{y\in \Gamma}2\sign\bc{y,b}\Delta_{y\ra b}^{\brk\ell}}\leq20\delta k\theta.
	\end{equation}
Therefore, taking exponentials in~(\ref{eqlin12pre}), we obtain
	\begin{eqnarray}\nonumber
	\prod_{y\in \Gamma}1-2\sign\bc{y,b}\Delta_{y\ra b}^{\brk\ell}&=&
		\exp\brk{O(\theta k\delta)^2-\sum_{y\in \Gamma}2\sign\bc{y,b}\Delta_{y\ra b}^{\brk\ell}}\\
		&=&
		1-\sum_{y\in \Gamma}2\sign\bc{y,b}\Delta_{y\ra b}^{\brk\ell}+O_\rho(\theta k\delta)^2.
			\label{eqlin12}
	\end{eqnarray}
Furthermore, the definition of $N_{\leq1}^{\brk{\ell+1}}\bc{x\ra a}$ ensures that
	$$|N(b)\setminus(\Gamma\cup\cbc x)=|N(b)\cap T\brk\ell\setminus\cbc x|\leq1.$$
If there is $y_0\in N(b)\cap T\brk\ell\setminus\cbc x$, then (\ref{eqNotnagel}) and~(\ref{eqlin12}) yield
	\begin{eqnarray*}
	\prod_{y\in N(b)\setminus\cbc x}1-2\sign\bc{y,b}\Delta_{y\ra b}^{\brk\ell}&=&
		(1-2\sign\bc{y_0,b}\Delta_{y_0\ra b}^{\brk\ell})\cdot\prod_{y\in \Gamma}1-2\sign\bc{y,b}\Delta_{y\ra b}^{\brk\ell}\\
		&=&
			1-2\brk{\sum_{y\in N(b)\setminus\cbc x}\sign\bc{y,b}\Delta_{y\ra b}^{\brk\ell}}+
				O_\rho(\theta k\delta).
	\end{eqnarray*}
Hence, in any case we have
	\begin{eqnarray*}
	\prod_{y\in N(b)\setminus\cbc x}1-2\sign\bc{y,b}\Delta_{y\ra b}^{\brk\ell}&=&
		1-2\brk{\sum_{y\in N(b)\setminus\cbc x}\sign\bc{y,b}\Delta_{y\ra b}^{\brk\ell}}\\
			&&\qquad\qquad\qquad\qquad
				+(\theta k\delta+|N(b)\cap T\brk\ell\setminus\cbc x|)\cdot O_\rho(\theta k\delta).
	\end{eqnarray*}
Combining this with~(\ref{eqlin11}) and using the approximation $\abs{\ln(1-z)+z}\leq z^2$ for $|z|\leq1/2$, we see that
	\begin{eqnarray*}
	\ln\bcfr{\mu_{b\ra x}^{\brk{\ell}}(-s)}{\mu_{b\ra x}^{\brk{\ell}}(s)}
		&=&-2^{1-|N(b)|}\brk{1-2\sum_{y\in N(b)\setminus\cbc x}\sign\bc{y,b}\Delta_{y\ra b}^{\brk\ell}}\\
		&&\qquad	+2^{1-|N(b)|}(\theta k\delta+|N(b)\cap T\brk\ell\setminus\cbc x|)\cdot O_\rho(\theta k\delta),\nonumber
	\end{eqnarray*}
whence the assertion follows.
\qed\end{proof}

\begin{proof}[\Prop~\ref{Prop_P0}]
Suppose $x\not\in T_2\brk{\ell+1}$.
By the definition of $P_{\leq1}^{\brk{\ell+1}}\bc{x\ra a}$ we have
	\begin{eqnarray*}
	\ln P_{\leq1}^{\brk{\ell+1}}\bc{x\ra a}&=&\sum_{b\in N_{\leq1}^{\brk{\ell+1}}\bc{x\ra a}}
			\ln\bcfr{\mu_{b\ra x}^{\brk{\ell}}(1)}{\mu_{b\ra x}^{\brk{\ell}}(-1)}.
	\end{eqnarray*}
Hence, \Lem~\ref{Lemma_lin1} yields
	\begin{eqnarray}		\label{eqProp_P01}
	\ln P_{\leq1}^{\brk{\ell+1}}\bc{x\ra a}
		&=&L_{x\ra a}^{\brk{\ell+1}}+
				\hspace{-5mm}\sum_{b\in N_{\leq1}^{\brk{\ell+1}}\bc{x\ra a}}	\hspace{-5mm}
					2^{-|N(b)|}(\theta k\delta+|N(b)\cap T\brk\ell\setminus\cbc x|)O_\rho(\theta k\delta).
	\end{eqnarray}
To complete the proof, we need to estimate the second summand.
Condition {\bf T2b} implies
	\begin{eqnarray}\nonumber
	O_\rho(\delta\theta k)^2\sum_{b\in N_{\leq1}^{\brk{\ell+1}}\bc{x\ra a}}2^{-|N(b)|}&\leq&
		O_\rho(\delta\theta k)^2	\sum_{b\in N(x)}2^{-|N(b)|}\\
		&\leq&\frac{O_\rho(\delta\theta k)^2}{\delta(\theta k)^5}\leq\frac{O_\rho(\delta)}{(\theta k)^3}\leq10^{-4}\delta.
		\label{eqProp_P02}
	\end{eqnarray}
Furthermore,  {\bf T2c} yields
	\begin{eqnarray}\nonumber
	O_\rho(\theta k\delta)\sum_{b\in N_{\leq1}^{\brk{\ell+1}}\bc{x\ra a}}2^{-|N(b)|}|N(b)\cap T\brk\ell\setminus\cbc x|&\leq&
		O_\rho(\theta k\delta)\sum_{b\in N_{1}\bc{x,T\brk\ell}}2^{-|N(b)|}\\
		&\leq&O_\rho(\theta k\delta)\cdot\rho(\theta k)^5\delta\leq10^{-4}\delta.
			\label{eqProp_P03}
	\end{eqnarray}
Finally, the assertion follows by plugging~(\ref{eqProp_P02}) and~(\ref{eqProp_P03}) into~(\ref{eqProp_P01}).
\qed\end{proof}

\subsection{Completing the proof of \Thm~\ref{Thm_dynamics}}\label{Sec_finish}

We are going to show that $\abs{\mu_x(\Phi_{t},\omega)-\frac12}\leq\delta=\delta_t$ for all $x\in V_t\setminus T\brk{\omega+1}$.
This will imply \Thm~\ref{Thm_dynamics}, because $\abs{T\brk{\omega+1}}\leq\delta_t(n-t)$ by \Prop~\ref{Prop_Tbound}.

Thus, let $x\in V_t\setminus T\brk{\omega+1}$.
\Cor~\ref{Cor_clause2var} shows that $\mu_{b\ra x}^{\brk\omega}(\zeta)>0$ for $\zeta=\pm1$.
Hence,
	\begin{eqnarray*}
	P(\zeta)&=&\prod_{b\in N(x)}\mu_{b\ra x}^{\brk\omega}(\zeta)>0\qquad\mbox{for $\zeta=\pm1$.}
	\end{eqnarray*}
Recall from~(\ref{eqBPmarginal}) that
	$$\mu_x(\Phi_{t},\omega)=\frac{P(1)}{P(-1)+P(1)}.$$
If $N(x)=\emptyset$, then trivially $P(-1)=P(1)=1$ and thus $\mu_x(\Phi_{t-1},\omega)=\frac12$.
Thus, assume that $N(x)\not=\emptyset$ and pick an arbitrary $a\in N(x)$.
Then
	\begin{eqnarray*}
	P(\zeta)&=&\mu_{a\ra x}^{\brk\omega}(\zeta)\cdot\mu_{x\ra a}^{\brk{\omega+1}}(\zeta)\quad\mbox{for }\zeta=\pm1.
	\end{eqnarray*}
Since $x\not\in T\brk{\omega+1}\supset B\brk{\omega+1}$ (by \Prop~\ref{Prop_TB}), we have
	$$\abs{\mu_{x\ra a}^{\brk{\omega+1}}(\zeta)-\frac12}=\abs{\Delta_{x\ra a}^{\brk{\omega+1}}}\leq0.1\delta\qquad\mbox{for }\zeta=\pm1.$$
Therefore, 
	\begin{eqnarray*}
	\ln\frac{\mu_{x\ra a}^{\brk{\omega+1}}(-1)}{\mu_{x\ra a}^{\brk{\omega+1}}(1)}&\leq&
		\ln\frac{1+0.2\delta}{1-0.2\delta}\leq0.5\delta,\quad\mbox{and analogously}\\
	\ln\frac{\mu_{x\ra a}^{\brk{\omega+1}}(-1)}{\mu_{x\ra a}^{\brk{\omega+1}}(1)}&\geq&-0.5\delta.
	\end{eqnarray*}
Furthermore, since $x\not\in T\brk{\omega+1}$ \Cor~\ref{Cor_clause2var} yields
	$$\abs{\ln\frac{\mu_{a\ra x}^{\brk\omega}(-1)}{\mu_{a\ra x}^{\brk\omega}(1)}}\leq2\exp(-k_1/2)\leq\delta^2.$$
Hence,
	\begin{eqnarray*}
	\abs{\ln\frac{P(-1)}{P(1)}}&\leq&
		\abs{\ln\frac{\mu_{x\ra a}^{\brk{\omega+1}}(-1)}{\mu_{x\ra a}^{\brk{\omega+1}}(1)}}+
			\abs{\ln\frac{\mu_{a\ra x}^{\brk{\omega}}(-1)}{\mu_{a\ra x}^{\brk{\omega}}(1)}}
		\leq0.5\delta+\delta^2\leq0.51\delta.
	\end{eqnarray*}
Therefore, letting $z=\ln\frac{P(-1)}{P(1)}$, we obtain
	\begin{eqnarray*}
	\abs{\frac12-\mu_x^{\brk\omega}(\Phi_{t-1})}&=&\abs{\frac12-\frac{P(1)}{P(-1)+P(1)}}
		=\abs{\frac12-\frac{1}{1+\exp(z)}}\leq\abs{\frac{1-\exp(z)}{2}}<\delta,
	\end{eqnarray*}
as desired.

\section{Proof of \Prop~\ref{Prop_reasonable}}\label{Apx_reasonble}

Recall from~(\ref{eqdeltat}) that $\delta_s=\exp(-c(1-s/n) k)$ and that
$\hat t=(1-\frac{\ln\rho}{c^2k})n$.
Suppose that $1\leq t\leq \hat t$.
Then $\theta=1-t/n$ satisfies $\theta k\geq\ln(\rho)/c^2$.
We assume throughout that $\rho=kr/2^k\geq\rho_0$ for some large enough number $\rho_0$;
	in particular, we assume that $\rho_0\geq\exp(1/c)$.
Set
	$$\delta=\delta_t=\exp(-ck\theta)$$
for brevity.
Then \Lem~\ref{Lemma_Deltat} yields
	\begin{eqnarray}\label{eqplatt}
	\delta\theta n&>&10^{15}\Delta_t.
	\end{eqnarray}

To prove \Prop~\ref{Prop_reasonable}, we will study two slightly different models of random $k$-CNFs.
In the first ``binomial'' model $\PHIbin$, we obtain a $k$-CNF by including each of the $(2n)^k$ possible clauses over $V=\cbc{x_1,\ldots,x_n}$
with probability $p=m/(2n)^k$ independently, where each clause is an ordered $k$-tuple of not necessarily distinct literals.
Thus, $\PHIbin$ is a random set of clauses, and $\Erw|\PHIbin|=m$.

In the second model, we choose a \emph{sequence} $\PHIseq$ of $m$ independent $k$-clauses $\PHIseq(1),\ldots,\PHIseq(m)$,
	each of which consists of $k$ independently chosen literals.
Thus, the probability of each individual sequence is $(2n)^{-km}$.
The sequence $\PHIseq'$ corresponds to the $k$-CNF $\cbc{\PHIseq(1),\ldots,\PHIseq(m)}$ with \emph{at most} $m$ clauses.
The following well-known fact relates $\PHI$ to $\PHIbin$, $\PHIseq$.

\begin{fact}\label{Fact_models}
For any event $\cE$ we have 
	\begin{eqnarray*}
	\pr\brk{\PHI\in\cE}&\leq&O(\sqrt m)\cdot\pr\brk{\PHIbin\in\cE},\\
	\pr\brk{\PHI\in\cE}&\leq&O(1)\cdot\pr\brk{\PHIseq\in\cE}.
	\end{eqnarray*}
\end{fact}

Due to Fact~\ref{Fact_models} and~(\ref{eqplatt}), it suffices to prove that the statements {\bf Q1}--{\bf Q4} hold for either of $\PHI$, $\PHIbin$, $\PHIseq$
with probability at least $1-\exp(-10^{-13}\delta\theta n)$.

\subsubsection{Establishing  Q1.}
We are going to deal with the number of variables that appear in ``short'' clauses first.

\begin{lemma}\label{Lemma_shortClauses}
With probability at least 
$1-\exp(-10^{-6}\delta\theta n)$
in $\PHI^t$ there are no more than
$\theta n\cdot 10^{-5}\frac{\delta}{\theta k}$ clauses of length less than $0.1\theta k$.
\end{lemma}
\begin{proof}
We are going to work with $\PHIbin$.
Let $L_j$ be the number of clauses of length $j$ in $\PHIbin^t$.
Then for any $j\in\brk k$ we have 
	$$\lambda_j=\Erw\brk{L_j}=m\cdot 2^{j-k}\bink{k}j\theta^j(1-\theta)^{k-j}
		=\frac{2^j\rho\theta n}j\bink{k-1}{j-1}\theta^{j-1}(1-\theta)^{k-j}.$$
Indeed, a clause has length $j$ in $\PHIbin^t$ iff it contains $j$ variables from the set $V_t$ of size $\theta n$ and $k-j$ variables from $V\setminus V_t$
	\emph{and} none of the $k-j$ variables from $V\setminus V_t$ occurs positively.
	The total number of possible clauses with these properties is $2^{j}\bink{k}j(\theta n)^j((1-\theta)n)^{k-j}$,
	and each of them is present in $\PHIbin^t$ with probability $p=m/(2n)^k$ independently.

Let's start by bounding the total number 
$L_*=\sum_{j<\theta k/10}L_j$ of ``short'' clauses.
Its expectation 
 is bounded by
	\begin{eqnarray*}
	\Erw\brk{L_*}&=&\sum_{j<\theta k/10}\lambda_j
		\leq2^{0.1\theta k}\rho\theta n\cdot\pr\brk{\Bin(k-1,\theta)<\theta k/10}\\
		&\leq&2^{0.1\theta k}\rho\theta n\cdot\exp(-\theta k/3)\qquad\mbox{\ \ [by \Lem~\ref{Lemma_Chernoff} (Chernoff)]}\\
		&\leq&\theta \exp(-\theta k/4)n\qquad\qquad\qquad\mbox{[as $\theta k\geq\ln(\rho)/c^2$]}.
	\end{eqnarray*}
Furthermore, $L_*$ is binomially distributed, because clauses appear independently  in $\PHIbin$.
Hence, again by
\Lem~\ref{Lemma_Chernoff} we have
	\begin{eqnarray}\nonumber
	\pr\brk{L_*>\theta n\cdot 10^{-5}\delta/(\theta k)}&\leq&
	\exp\brk{-\frac{10^{-5}\delta}{\theta k}\cdot\ln\bcfr{10^{-5}\delta/(\theta k)}{\exp(1-\theta k/4)}\cdot\theta n}\\
		&\leq&\exp\bc{-\frac{\delta}{5\cdot10^{5}\theta k}\cdot\theta k\cdot\theta n}\leq\exp\bc{-10^{-6}\delta\theta n}
			.
		\label{eqshortClauses}
	\end{eqnarray}
Hence, the assertion follows from~(\ref{eqshortClauses}) and Fact~\ref{Fact_models}.
\qed\end{proof}

\begin{corollary}\label{Cor_shortClauses}
With probability at least $1-\exp(-10^{-6}\delta\theta n)$
in $\PHI^t$  no more than $10^{-6}\delta\theta n$ variables appear in clauses
of length less than $0.1\theta k$.
\end{corollary}
\begin{proof}
This is immediate from \Lem~\ref{Lemma_shortClauses}.
\qed\end{proof}

\noindent
As a next step, we are going to bound the number of variables that appear in clauses of length $\geq10\theta k$.

\begin{lemma}\label{Lemma_longClauses}
With probability at least $1-\exp(-10^{-11}\delta\theta n)$ 
we have
$$\sum_{b\in\PHI^t:\abs{N(b)}>10k\theta}\abs{N(b)}\leq10^{-6}\delta\theta n.$$
\end{lemma}
\begin{proof}
For a given $\mu>0$ let $\cL_\mu$ be the
event that $\PHIseq^t$ has $\mu$ clauses so that the sum of the lengths of these clauses is at least $\lambda=10\theta k\mu$.
Then
	$$\pr\brk{\cL_\mu}\leq\bink m\mu\bink{k\mu}\lambda\theta^{\lambda}\bc{\frac12+\theta}^{k\mu-\lambda}.
		$$
Indeed, there are $\bink m\mu$ ways to choose $\mu$ places for these $\mu$ clauses in $\PHIseq$.
Once these have been specified, there are $k\mu$ literals that constitute the $\mu$ clauses, and we choose $\lambda$ whose
underlying variables are supposed to be in $V_t$;
the probability that this is indeed the case for all of these $\lambda$ literals is $\theta^\lambda$.
Moreover, in order for the each of the clauses to remain in $\PHIseq^t$,
the remaining $k\mu-\lambda$ literals must either be negative
or have underlying variables from $V_t$, leading to the $(\theta+1/2)^{k\mu-\lambda}$ factor.
Thus,
	\begin{eqnarray*}
	\pr\brk{\cL_\mu}
	&\leq&
		\bink{m}\mu\brk{\bc{\frac12+\theta}\bcfr{\eul}{5}^{10\theta}}^{k\mu}\qquad\qquad\qquad\qquad\mbox{[as $\lambda=10k\theta \mu$]}\\
	&\leq&\bcfr{\eul n\rho}{k\mu}^\mu\brk{(1+2\theta)\bcfr{\eul}{5}^{10\theta}}^{k\mu}
			\qquad\quad\qquad\qquad\mbox{[as $m=n\cdot 2^k\rho/k$]}\\
	&\leq&\brk{\frac{\eul n\rho}{k\mu}\bcfr{\eul}{4}^{10\theta k}}^\mu\\
	&=&\brk{\bcfr{10\eul\rho\theta n}{\lambda}^{1/(10\theta k)}\bcfr{\eul}{4}}^{\lambda}
		\qquad\qquad\qquad\qquad\mbox{[as $\lambda=10k\theta \mu$]}.
	\end{eqnarray*}
Hence, if $\lambda\geq10^{-6}\delta\theta n$ we get
	\begin{eqnarray*}\nonumber
	\pr\brk{\cL_\mu}
		&\leq&
		\brk{\bcfr{10^7\eul\rho}\delta^{1/10\theta k}\bcfr{\eul}{4}}^\lambda
		\leq\bcfr{\eul}{3}^\lambda\quad\qquad\mbox{[as $k\theta\geq\ln(\rho)/c^2$ and $\delta=\exp(-ck\theta)$]}\\
		&\leq&\exp(-10^{-10}\delta\theta n)
			.\label{eqlongClauses1}
	\end{eqnarray*}
Thus, we see that in $\PHIseq^t$ with probability at least $1-\exp(-10^{-10}\delta\theta n)$ we have
	\begin{equation}\label{eqLongClauseLengths}
	\sum_{b:\abs{N(b)}>10k\theta}\abs{N(b)}\leq10^{-6}\delta\theta n.
	\end{equation}
Hence, Fact~\ref{Fact_models} implies that (\ref{eqLongClauseLengths}) holds in $\PHI^t$ with probability at least
$1-\exp(-10^{-11}\delta\theta n)$.
\qed\end{proof}

\begin{corollary}\label{Cor_longClauses}
With probability at least $1-\exp(-10^{-11}\delta\theta n)$ 
no more than $10^{-6}\delta\theta n$ variables appear in clauses
of length greater than $10\theta k$.
\end{corollary}
\begin{proof}
The number of such variables 
is bounded by
	$\sum_{b:|N(b)|>10\theta k}|N(b)|.$
Therefore, the assertion follows from \Lem~\ref{Lemma_longClauses}.
\qed\end{proof}

We now come to the second part of {\bf Q1.}
We start with the following simple observation.

\begin{lemma}\label{Lemma_muj}
Let $x\in V_t$.
The expected number of clauses of length $j$ in $\PHIbin^t$ 
where $x$ is the underlying variable of the $l$th literal is
	\begin{equation}\label{eqmuj}
	\mu_j=\rho\cdot\frac{2^j}j\bink{k-1}{j-1}\theta^{j-1}(1-\theta)^{k-j+1}\leq 2^j\rho/j.
	\end{equation}
\end{lemma}
\begin{proof}
There are $2^j\bink{k}{j}(\theta n)^{j-1}((1-\theta)n)^{k-j+1}$ possible clauses that have exactly $j$
	literals whose underlying variable is in $V_t$ such that the underlying variable of the
	$j$th such literal is $x$.
	Each such clause is present in $\PHIbin$ with probability $p=m/(2n)^k=\frac{\rho}k n^{1-k}$ independently.
\qed\end{proof}

\begin{lemma}
With probability at least $1-\exp(-10^{-12}\delta\theta n)$ no more than $10^{-4}\delta\theta n$ variables $x\in V_t$ are such that
$\delta(\theta k)^3\sum_{b\in N(x)}2^{-|N(b)|}>1$.
\end{lemma}
\begin{proof}
For $x\in V_t$ let $X_j(x)$ be the 
number of clauses of length $j$ in $\PHIbin^t$ that contain a $x$, and let $X_{jl}(x)$ be the number
of such clauses where $x$ is the underlying variable of the $l$th literal of that clause $(1\leq l\leq j$).
Then $\Erw\brk{X_{jl}(x)}=\mu_j$, with $\mu_j$ as in~(\ref{eqmuj}).
Since $1/\delta=\exp(ck\theta)$ and $\theta\geq\ln(\rho)/c^2$, 
we see that $2^j\delta^{-1}(\theta k)^{-5}/j>100\mu_j$.
Hence, \Lem~\ref{Lemma_Chernoff} (the Chernoff bound) yields
	\begin{eqnarray*}
	\pr\brk{X_{jl}(x)>10(\mu_j+2^j\delta^{-1}(\theta k)^{-5}/j)}&\leq&\zeta,\quad\mbox{with }\zeta=\exp(-10/(\delta(\theta k)^5)).
	\end{eqnarray*}
Let $V_{jl}$ be the set of all variables $x\in V_t$ such that $X_{jl}(x)>10(\mu_j+2^j\delta^{-1}(\theta k)^{-5}/j)$.
Since the random variables $(X_{jl}(x))_{x\in V_t}$ are mutually independent, \Lem~\ref{Lemma_Chernoff} (the Chernoff bound) yields
	$$\pr\brk{|V_{jl}|>\frac{\delta}{(\theta k)^9}\cdot \theta n}\leq
		\exp\brk{-\frac{\delta\theta n}{(\theta k)^9}\cdot\ln\bcfr{\delta}{\eul(\theta k)^9\zeta}}$$
Since
	$\zeta^{-1}=\exp(10/(\delta(\theta k)^5))=\exp\brk{10\exp(ck\theta)/(\theta k)^5}$
and $k\theta\geq\ln(\rho)/c^2\gg1$,
we have 
	$$\ln\bcfr{\delta}{\eul(\theta k)^9\zeta}\geq-\ln(\zeta)/2,$$
whence
	\begin{eqnarray}
	\pr\brk{|V_{jl}|>\frac{\delta}{(\theta k)^9}\cdot \theta n}
		&\leq&\exp\brk{\frac{\delta\theta n}{2(\theta k)^9}\cdot\ln\zeta}
			\leq\exp\brk{-\frac{\theta n}{(\theta k)^{15}}}\leq\exp(-\delta\theta n).
		\label{eqVijbound}
	\end{eqnarray}
Furthermore, if $x\not\in V_{jl}$ for all $1\leq j\leq10\theta k$ and all $1\leq l\leq j$, then
	\begin{eqnarray*}
	\sum_{b\in N(x):|N(b)|\leq10\theta k}2^{-|N(b)|}
		&\leq&10\sum_{j\leq10\theta k}2^{-j}(j\mu_j+2^j\delta^{-1}(\theta k)^{-5})\\
		&\leq&100\delta^{-1}(\theta k)^{-4}+10\sum_{j\leq10\theta k}j2^{-j}\mu_j\\
		&\leq&100\delta^{-1}(\theta k)^{-4}+10\rho
			<\delta^{-1}(\theta k)^{-3}
				,
	\end{eqnarray*}
where we used that $\theta k\geq\ln(\rho)/c^2$, so that $1/\delta\geq(\theta k)^5\rho$.
Hence, the assertion follows from~(\ref{eqVijbound}), Fact~\ref{Fact_models}
and the bound on the number of variables
in clauses of length $>10\theta k$ provided by \Lem~\ref{Lemma_longClauses}.
\qed\end{proof}

\subsubsection{Establishing Q2.}
Let $T\subset V_t$ be a set of size $|T|\leq\delta\theta n$.
For a variable $x$ we let $\cQ(x,i,j,l,T)$ be the number of clauses $b$ of $\PHIbin^t$
such that the $i$th literal is either  $x$ or $\neg x$,
$|N(b)|=j$, and $|N(b)\cap Q\setminus\cbc x|=l$.

\begin{lemma}\label{Lemma_nastyQ}
Suppose that $l\geq 1$, $j-l>k_1$  and $0.1\theta k\leq j\leq10\theta k$.
Let
	$$\gamma_{j,l}=\left\{\begin{array}{cl}
			10 j2^j\delta\rho&\mbox{ if }l=1,\\
			102^{j-l}\delta^{1.9}&\mbox{ if }l>1.
			\end{array}\right.$$
Then for any $i,x,T$ we have
	$\pr\brk{\cQ(x,i,j,l,T)>\gamma_{j,l}}\leq\exp(-\exp(c^{2/3}\theta k)).$
\end{lemma}
\begin{proof}
The random variable $\cQ(x,i,j,l,T)$ has a binomial distribution,
because clauses appear independently in $\PHIbin$.
With $\mu_j$ from~(\ref{eqmuj}) we have for $l>1$
	$$\Erw\brk{\cQ(x,i,j,l,T)}\leq\bink{j}l\delta^l\mu_j
		\leq\rho\bink jl\delta^l2^j
		\leq2^{j-l}\delta^{1.9};$$
in the last step we used that $\delta^{0.05}\leq1/\rho$, which follows
	from our assumption that $\theta k\geq\ln(\rho)/c^2$, and that $2^l\bink jl\leq (2j)^l\leq(20k\theta)^l\leq\delta^{0.05l}$.
Hence, by \Lem~\ref{Lemma_Chernoff} (the Chernoff bound) in the case $j-l>k_1=\sqrt c\theta k$, $l>1$
we get
	$$\pr\brk{\cQ(x,i,j,l,T)>10\cdot2^{j-l}\delta^{1.9}}\leq\exp(-2^{j-l}\delta^{1.9})
			\leq\exp(-2^{k_1}\delta^{1.9})\leq\exp(-\exp(c^{2/3}\theta k)),$$
as $\delta=\exp(-ck\theta)$.

By a similar token, in the case $l=1$ we have
	$\Erw\brk{\cQ(x,i,j,l,T)}\leq j\delta^l\mu_j
		\leq\rho\delta 2^j$.
Hence, once more by the Chernoff bound
	$$\pr\brk{\cQ(x,i,j,l,T)>10\cdot2^{j}\delta\rho}\leq\exp(-2^{j}\delta\rho)
			\leq\exp(-2^{k_1}\delta)\leq\exp(-\exp(c^{2/3}\theta k)),$$
as claimed.
\qed\end{proof}

\noindent
Let $\cZ(i,j,l,T)$ be the number of variables $x\in V_t$ for which $\cQ(x,i,j,l,T)>\gamma_{j,l}$.

\begin{lemma}\label{Lemma_nastyZ}
Suppose that $l\geq 1$, $j-l>k_1$ and $0.1\theta k\leq j\leq10\theta k$.
Then for any $i,T$ we have
	$$\pr\brk{\cZ(i,j,l,T)>\delta\theta n/(\theta k)^4}\leq\exp\brk{-\frac{\delta\theta n}{2(\theta k)^4}\cdot\exp(c^{2/3}\theta k)}.$$
\end{lemma}
\begin{proof}
Whether a variable $x\in V_t$ contributes to $\cZ(i,j,l,T)$ depends only on those clauses of $\PHIbin^t$ whose
$i$th literal reads either $x$ or $\neg x$.
Since these sets of clauses are disjoint for distinct variables and as clauses appear independently in $\PHIbin$,
 $\cZ(i,j,l,T)$ is a binomial random variable. 
By \Lem~\ref{Lemma_nastyQ},
	$$\Erw\brk{\cZ(i,j,l,T)}\leq\theta n\exp(-\exp(c^{2/3}\theta k)).$$
Hence, \Lem~\ref{Lemma_Chernoff} (the Chernoff bound) yields
	\begin{eqnarray}\nonumber
	\pr\brk{\cZ(i,j,l,T)>\delta\theta n/(\theta k)^4}
		&\leq&\exp\brk{-\frac{\delta\theta n}{(\theta k)^4}\ln\bcfr{\delta}{(\theta k)^4\exp(1-\exp(c^{2/3}\theta k))}}\\
		&\leq&\exp\brk{-\frac{\delta\theta n}{2(\theta k)^4}\cdot\exp(c^{2/3}\theta k)},
\nonumber		\label{eqQ21}
	\end{eqnarray}
as desired.
\qed\end{proof}	

\begin{corollary}\label{Cor_nastyZ}
With probability $1-\exp(-\delta\theta n)$ the random formula $\PHIbin^t$ has the following property.
\begin{equation}\label{eqnastyZ}
\parbox[c]{12cm}{For all $i,j,l,T$ such that $l\geq 1$, 
$j-l>k_1$,
 $0.1\theta k\leq j\leq10\theta k$ and $|T|\leq\delta\theta n$ we have $\cZ(i,j,l,T)\leq\delta\theta n/(\theta k)^4.$}
 \end{equation}
\end{corollary}
\begin{proof}
We apply the union bound.
There are at most $n\bink n{\delta n}$ ways to choose the set $T$, and no more than $n$ ways to choose $i,j,l$.
Hence, by \Lem~\ref{Lemma_nastyZ} the probability that there exist $i,j,l,T$ such that 
$\cZ(i,j,l,T)>\theta n\exp(-\exp(c^{2/3}\theta k))$ is bounded by
	\begin{eqnarray*}
	n^2\bink{n}{\delta\theta n}\exp\brk{-\frac{\delta\theta n}{(\theta k)^4}\cdot\exp(c^{2/3}\theta k)}&\leq&
		\exp\brk{o(n)+\delta\theta n(1-\ln(\delta\theta))-\frac{\delta\theta n}{(\theta k)^4}\cdot\exp(c^{2/3}\theta k)}\\
		&\leq&\exp\brk{\delta\theta n\brk{o(1)-2\ln\delta-\exp(c^{3/4}\theta k)}}
		\leq\exp\brk{-\delta\theta n},
	\end{eqnarray*}
as claimed.
\qed\end{proof}

\begin{corollary}
With probability $1-\exp(-10^{-12}\delta\theta n)$ the random formula $\PHI^t$ has the following property.
\begin{quote}
If $T\subset V_t$ has size $|T|\leq\delta\theta n$, then for all but $10^{-4}\delta\theta n$ variables $x$ we have
	$$\sum_{b\in N_{>1}(x,T)}2^{|N(b)\cap T\setminus\cbc x|-|N(b)|}<\frac{\delta}{\theta  k}
		\mbox{ and }
		\sum_{b\in N_1(x,T)}
					2^{-|N(b)|}<\rho(\theta  k)^5\delta.$$
\end{quote}
\end{corollary}
\begin{proof}
Given $T\subset V_t$ of size $|T|\leq\delta\theta n$,
let $\cV_T$ be the set of all variables $x$ with the following two properties.
\begin{enumerate}
\item[i.] For all $b\in N(x)$ we have $0.1\theta k\leq |N(b)|\leq10\theta k$.
\item[ii.] For all $1\leq i\leq j$, $1\leq l\leq j-k_1$, and $0.1\theta k\leq j\leq10\theta k$ we have $\cQ(x,i,j,l,T)\leq\gamma_{j,l}$.
\end{enumerate}
Then for all $x\in\cV_T$ we have
	\begin{eqnarray*}
	\sum_{b\in N_{>1}(x,T)}2^{|N(b)\cap T\setminus\cbc x|-|N(b)|}&=&
		\sum_{0.1k\theta\leq j\leq10 k\theta}\sum_{i=1}^j\sum_{l=2}^{j-k_1}\cQ(x,i,j,l,T)2^{l-j}\qquad\mbox{ [due to i.]}\\
		&\leq&10k\theta\sum_{0.1k\theta\leq j\leq10 k\theta}\sum_{l=2}^{j-k_1}\gamma_{j,l}2^{l-j}\qquad\qquad\qquad\mbox{[due to ii.]}\\
		&\leq&1000(k\theta)^2\delta^{1.9}<\delta/(k\theta)\qquad\qquad\qquad\mbox{[as $\delta=\exp(-ck\theta)$]}.
	\end{eqnarray*}
Similarly,
	\begin{eqnarray*}
	\sum_{b\in N_1(x,T)}2^{-|N(b)|}&\leq&
		\sum_{0.1k\theta\leq j\leq10 k\theta}\sum_{i=1}^j\cQ(x,i,j,1,T)2^{-j}\qquad\qquad\qquad\qquad\mbox{ [due to i.]}\\
		&\leq&10k\theta\sum_{0.1k\theta\leq j\leq10 k\theta}2^{-j}\gamma_{j,1}\qquad\qquad\qquad\qquad\qquad\qquad\mbox{[due to ii.]}\\
		&\leq&1000(k\theta)^3\delta\rho<\rho(k\theta)^5\delta\qquad\qquad\qquad\qquad\qquad\mbox{[as $k\theta\geq\ln(\rho)/c^2\gg1$].}\\
	\end{eqnarray*}

Thus, to complete the proof we need to show that with sufficiently high probability $\cV_T$ is sufficiently big for all $T$.
By \Lem s~\ref{Lemma_shortClauses} and~\ref{Lemma_longClauses}
with probability $1-2\exp(-10^{-11}\delta\theta n)$ the number of variables $x$ that fail to satisfy i.\ is less than $2\cdot10^{-6}\delta\theta n$.
Furthermore, by \Cor~\ref{Cor_nastyZ} and Fact~\ref{Fact_models}, with probability $\geq1-\exp(-\delta\theta n/2)$
the random formula $\PHI^t$ satisfies~(\ref{eqnastyZ}).
In this case, for all $T$ the number of variables that fail to satisfy ii.\ is bounded by $\delta\theta n/(k\theta)^4<10^{-5}\delta\theta n$.
Thus, with probability $\geq1-\exp(-10^{-12}\delta\theta n)$ we have $\abs{\cV_T}>\theta n(1-10^{-4}\delta)$ for all $T$, as desired.
\qed\end{proof}
	
For a set $T\subset V_t$ and numbers $i\leq j$ we let
$\cN_+(x,i,j,T)$ be the number of clauses $b\in N(x)$
in $\PHIbin^t$ such that $|N(b)|=j$,
the $i$th literal of $b$ is $x$ and $|N(b)\cap T\setminus x|\leq1$.
Similarly, we let $\cN_-(x,i,j,T)$ be the number of 
$b\in N(x)$ such that $|N(b)|=j$,
the $i$th literal of $b$ is $\neg x$ and $|N(b)\cap T\setminus x|\leq1$.
Let $\cB(i,j,T)$ be the set of variables $x\in V_t$ such that
	$$\abs{\cN_+(x,j,l)-\cN_-(x,j,l)}>2^j\delta(\theta k)^{-3}.$$

\begin{lemma}\label{Lemma_cancel}
Let $T\subset V_t$ be a set of size $|T|\leq\delta\theta n$.
Let $i,j$ be such that $i\leq j$ and $0.1\theta k\leq j\leq10\theta k$.
Then in $\PHIbin^t$ we have
	$\pr\brk{\cB(i,j,T)>\delta\theta n/(\theta k)^3}
		\leq\exp\brk{-\delta\theta n\exp(\theta k/22)}.$
\end{lemma}
\begin{proof}
Let $x\in V_t$.
In the random formula $\PHIbin^t$
we have
	$$\Erw\brk{\cN_+(x,i,j,T)+\cN_-(x,i,j,T)}\leq\mu_j\leq2^j\rho\qquad\mbox{(with $\mu_j$ as in (\ref{eqmuj}))}.$$
Furthermore, $\cN_+(x,i,j,T)$, $\cN_-(x,i,j,T)$ are binomially distributed with identical means,
because in $\PHIbin$ each literal is positive/negative with probability $\frac12$.
Hence,  for $j\geq0.1\theta k$ \Lem~\ref{Lemma_Chernoff} (the Chernoff bound) yields
	\begin{eqnarray}\nonumber
	\pr\brk{\abs{\cN_+(x,j,l)-\cN_-(x,j,l)}>2^j\delta(\theta k)^{-3}}&\leq&
		\exp\brk{-\frac{(2^j\delta(\theta k)^{-3})^2}{3(2^j\delta(\theta k)^{-3}+2^j\rho)}}\\
		&\hspace{-6cm}\leq&\hspace{-3cm}\;
			\exp\brk{-\frac{2^j\delta^2}{4(\theta k)^6\rho}}\nonumber\\
		&\hspace{-6cm}\leq&\hspace{-3cm}\;
			\exp(-\exp(\theta k/20))\quad\mbox{[as $\delta=\exp(-ck\theta)$, $j\geq0.1\theta k$].}
			\label{eqcancel}
	\end{eqnarray}
For different variables $x\in V_t$ the random variables $\cN_+(x,i,j)-\cN_-(x,i,j)$ are independent
	(because we fix the position $i$ where $x$ occurs).
Hence, $\cB(i,j,T)$ is a binomial random variable, and~(\ref{eqcancel}) yields
	$$\Erw\brk{\cB(i,j,T)}\leq\theta n\exp(-\exp(\theta k/20)).$$
Consequently, \Lem~\ref{Lemma_Chernoff} (the Chernoff bound) gives
	\begin{eqnarray*}
	\pr\brk{\cB(i,j,T)>\theta\delta n/(\theta k)^3}
		&\leq&\exp\brk{-\frac{\delta\theta n}{(\theta k)^3}\ln\bcfr{\delta\theta n/(\theta k)^3}{\exp(1-\exp(\theta k/20))\theta n}}\\
		&\leq&\exp\brk{-\frac{\delta\theta n}{(\theta k)^3}\cdot\exp(\theta k/21)}\leq
			\exp\brk{-\theta\delta n\exp(\theta k/22)}
	\end{eqnarray*}
provided that $\rho\geq\rho_0$ is sufficiently large.
\qed\end{proof}

\begin{corollary}\label{Cor_cancel}
With probability $\geq1-\exp(-\delta\theta n)$ the random formula $\PHIbin^t$ has the following property.
\begin{equation}\label{eqCorcancel}
\parbox[c]{12cm}{For all $T\subset V_t$ of size $|T|\leq\delta\theta n$ and all $i,j$ such that $i\leq j$, $0.1\theta k\leq j\leq10\theta k$ we have
		$\cB(i,j,T)\leq\delta\theta n/(\theta k)^3$.}
\end{equation}
\end{corollary}
\begin{proof}
Let $i,j$ be such that $i\leq j$, $0.1\theta k\leq j\leq 10\theta k$.
By \Lem~\ref{Lemma_cancel} and the union bound, the probability that there is a set $T$ such that
	$\cB(i,j,T)>\delta\theta n/(\theta k)^3$ is bounded by
	\begin{eqnarray*}
	n\bink{\theta n}{\delta\theta n}\exp\brk{-\delta\theta n\exp(\theta k/22)}
		&\leq&\exp\brk{o(n)+\delta\theta n(1-\ln(\theta\delta)-\exp(\theta k/22))}\\
		&\le&\exp\brk{-2\delta\theta n}\qquad\mbox{[as $\delta=\exp(-ck\theta)$]}.
	\end{eqnarray*}
Since there are no more than $(10k\theta)^2$ ways to choose $i,j$, the assertion follows.
\qed\end{proof}

\begin{corollary}\label{Cor_cancel2}
With probability $\geq1-\exp(-10^{-12}\delta\theta n)$ the random formula $\PHI^t$ has the following property.
\begin{equation}\label{eqCorcancel}
\parbox[c]{12cm}{If $T\subset V_t$ has size $|T|\leq\delta\theta n$, then there are no more than $10^{-5}\delta\theta n$
	variables $x\in V_t$ such that
	$$\abs{\sum_{b\in N_{\leq1}(x,T)}\hspace{-2mm}\frac{\sign(x,b)}{2^{|N(b)|}}}>\frac{\delta}{1000}.$$}
\end{equation}
\end{corollary}
\begin{proof}
Given $T\subset V_t$, let $\cV_T$ be the set of all $x\in V_t$ with the following two properties.
\begin{enumerate}
\item[i.] For all $b\in N(x)$ we have $0.1\theta k\leq |N(b)|\leq10\theta k$.
\item[ii.] For all $1\leq i\leq j$, $0.1\theta k\leq j\leq10\theta k$ we have $\cB(i,j,T)\leq\delta\theta n/(\theta k)^3$.
\end{enumerate}
Then for all $x\in\cV_T$ we have
	\begin{eqnarray*}
	\abs{\sum_{b\in N_{\leq1}(x)}\frac{\sign(x,b)}{2^{|N(b)|}}}&=&\abs{\sum_{0.1\theta k\leq j\leq10\theta k}\sum_{i=1}^j2^{-j}(\cN_+(x,i,j,T)-\cN_-(x,i,j,T))}\\
		&\leq&\sum_{0.1\theta k\leq j\leq10\theta k}\sum_{i=1}^j2^{-j}\abs{\cN_+(x,i,j,T)-\cN_-(x,i,j,T)}\\
		&\leq&\sum_{0.1\theta k\leq j\leq10\theta k}2^{-j}\cdot2^j\delta(\theta k)^{-3}\leq100\delta/(\theta k)<10^{-3}\delta.
	\end{eqnarray*}
Furthermore, by \Lem s~\ref{Lemma_shortClauses} and~\ref{Lemma_longClauses}
with probability $\geq1-2\exp(-10^{-11}\delta\theta n)$ the number of variables $x$ that fail to satisfy i.\ is less than $2\cdot10^{-6}\delta\theta n$.
In addition, by \Cor~\ref{Cor_cancel} and Fact~\ref{Fact_models} 
with probability $\geq1-\exp(-\delta\theta n/2)$
the number of variables $x$ that satisfy ii.\ in $\PHI^t$ is bounded by $10^{-5}\delta\theta n$.
Thus, with probability $\geq1-\exp(-10^{-12}\delta\theta n)$ we have $\cV_T\geq10^{-4}\delta\theta n$ for all $T$,
as claimed.
\qed\end{proof}

\subsubsection{Establishing Q3.}
We carry the proof out in the model $\PHIseq$.
Let $0.01\leq z\leq1$ and let $T$ be a set of size $|T|=q\theta n$ with $0.01\delta\leq q\leq100\delta$.

\begin{lemma}\label{Lemma_EQSZ}
Let $S,Z>0$ be integers and let $\cE_z(T,S,Z)$ be the event that $\PHIseq^t$ contains a set $\cZ$ of $Z$ clauses with the following properties.
\begin{enumerate}
\item[i.] $S=\sum_{b\in\cZ}|N(b)|>1.009|T|/z$,
\item[ii.] For all $b\in\cZ$ we have $0.1\theta k\leq|N(b)|\leq10\theta k$.
\item[iii.] All $b\in\cZ$ satisfy $|N(b)\cap T|\geq z|N(b)|$.
\end{enumerate}
Then 
	$\pr\brk{\cE_z(T,S,Z)}\leq q^{0.99999zS}.$ 
\end{lemma}
\begin{proof}
We claim that in $\PHIseq^t$, 
	\begin{eqnarray*}
	\pr\brk{\cE_z(T,S,Z)}&\leq&\bink{m}Z\bink{kZ}S\bink{S}{zS}2^{S-kZ}\theta^S(1-\theta)^{kZ-S}q^{zS}.
	\end{eqnarray*}
Indeed, $\PHIseq^t$ is based on the random sequence $\PHIseq$ of $m$ independent clauses.
Out of these $m$ clauses we choose a subset $\cZ$ of size $Z$, inducing a $\bink mZ$ factor.
Then, out of the $kZ$ literal occurrences of the clauses in $\cZ$ we choose $S$ (leading to the $\bink{kZ}S$ factor)
whose underlying variables lie in $V_t$, which occurs with probability $\theta=|V_t|/n$ independently for each literal (inducing a $\theta^S$ factor).
Furthermore, all $kZ-S$ literals whose variables are in $V\setminus V_t$ must be negative, because otherwise
the corresponding clauses would have been eliminated from $\PHIseq^t$, and not in $V_t$;
	this explains the $2^{S-kZ}(1-\theta)^{kZ-S}$ factor. 
Finally, out of the $S$ literal occurrences in $V_t$ a total of at least $zS$ has an underlying variable from $T$ (a factor of $\bink{S}{zS}$),
which occurs with probability $q=|T|/(\theta n)$ independently (hence the $q^{zS}$ factor).

Hence, we obtain
	\begin{eqnarray}\nonumber
	\pr\brk{\cE_z(T,S,Z)}&\leq&\bink mZ2^{-kZ}\brk{2^{1/z}\cdot\frac{\eul}z\cdot q}^{zS}\cdot\bink{kZ}S\theta^S(1-\theta)^{kZ-S}\\
		&\leq&\bink mZ2^{-kZ}\brk{2^{1/z}\cdot\frac{\eul}z\cdot q}^{zS}\leq\bink mZ2^{-kZ}(Cq)^{zS}
			\label{eqmZ0}
	\end{eqnarray}
for a certain absolute constant $C>0$, because $z\geq0.01$.
Since all clause lengths are required to be between $0.1\theta k$ and $10\theta k$, we obtain $0.1S/(\theta k)\leq Z\leq10S/(\theta k)$.
Therefore,
	\begin{eqnarray}\nonumber
	\bink mZ2^{-kZ}&\leq&\bcfr{\eul m}{2^kZ}^Z\leq\bcfr{\eul \rho n}{kZ}^Z\qquad\mbox{ [as $m=2^k\rho n/k$]}\\
		&\leq&\bcfr{10\eul\rho\theta n}{S}^{Z}\nonumber\\
		&\leq&\bcfr{10\eul\rho}{1.009q}^{Z}\qquad\qquad\mbox{[as $S\geq1.009q\theta n/z\geq1.009q\theta n$ by i.]}.
	\label{eqmZ1}
	\end{eqnarray}
Since $q\leq100\delta=100\exp(-c\theta k)$ and $\theta k\geq\ln(\rho)/c^2$, we
have $1/q\geq100\rho$ for $\rho\geq\rho_0$ sufficiently large.
Hence, (\ref{eqmZ1}) yields
	\begin{equation}\label{eqmZ2}
	\bink mZ2^{-kZ}\leq q^{-2Z}\leq q^{-20S/(\theta k)}.
	\end{equation}
Plugging~(\ref{eqmZ2}) into~(\ref{eqmZ0}), we obtain for $\theta k\geq\rho_0$ large enough
and $S\geq1.009|T|/z$
	\begin{eqnarray}\label{eqmZ3}\nonumber
	\pr\brk{\cE_z(T,S,Z)}&\leq&q^{-20S/(\theta k)}\cdot(Cq)^{zS}\leq q^{0.99999zS},
	\end{eqnarray}
as claimed.
\qed\end{proof}

\begin{corollary}\label{Cor_EQSZ}
Let $\cE$ be the event that there exist a number $z\in\brk{0.01,1}$, a set $T\subset V_t$ of size $
	|T|\leq100\theta\delta n$
and $S\geq\frac{1.01}z|T|+10^{-6}\delta\theta n$, $Z>0$ such that $\cE_z(T,S,Z)$ occurs.
Then $\cE$ occurs in $\PHI^t$ with probability $\leq\exp(-10^{-7}\delta\theta n)$.
\end{corollary}
\begin{proof}
Let $z\in\brk{0.01,1}$, let $0<q\leq100\delta$, and let $S,Z>0$ be integers such that $S\geq\frac{1.01}zq\theta n+10^{-6}\delta\theta n$.
Let $\cE_z(q,S,Z)$ denote the event that there is a set $T\subset V_t$ of size $|T|=q\theta n$ such that $\cE_z(T,S,Z)$ occurs.
Then by \Lem~\ref{Lemma_EQSZ} and the union bound, in $\PHIseq^t$ we have
	\begin{eqnarray}\nonumber
	\pr\brk{\cE(q,S,Z)}&\leq&\bink{\theta n}{q\theta n}q^{0.99999zS}
		\leq\exp\brk{q\theta n(1-\ln q+1.008\ln q))+0.9\cdot10^{-6}\delta\theta n\ln q}\\
		&\leq&\exp\bc{-0.9\cdot10^{-6}\delta\theta n}\qquad\qquad\qquad\qquad\qquad\mbox{[as $q\leq100\delta<1/\eul$]}
	\label{eqmZ4}
	\end{eqnarray}
Since there are only $O(n^4)$ possible choices of $S$, $Z$, $z$ and $q$, (\ref{eqmZ4}) and Fact~\ref{Fact_models} imply the assertion.
\qed\end{proof}

\begin{corollary}\label{Cor_Q3}
With probability at least $1-\exp(-10^{-12}\delta\theta n)$, $\PHI^t$ has the following property.
\begin{quote}
Let $0.01\leq z\leq1$ and let $T\subset V_t$ have size $0.01\delta\theta n\leq|T|\leq100\delta\theta n$.
Then $$\sum_{b:|N(b)\cap T|\geq z|N(b)|}|N(b)|\leq\frac{1.01}z|T|+2\cdot10^{-5}\delta\theta n.$$
\end{quote}
\end{corollary}
\begin{proof}
\Lem s~\ref{Lemma_shortClauses} and~\ref{Lemma_longClauses} and \Cor~\ref{Cor_EQSZ} imply that with probability 
at least $1-3\exp(-10^{-11}\delta\theta n)$, $\PHI^t$ has the following properties.
\begin{enumerate}
\item[i.] $\cE$ does not occur.
\item[ii.] $\sum_{b:|N(b)|\not\in[0.1\theta k,10\theta k]}|N(b)|\leq10^{-5}\delta\theta n$.
\end{enumerate}
Assume that i.\ and ii.\ hold and let $T\subset V_t$ be a set of size $|T|\leq100\delta\theta n$.
Let $0.01\leq z\leq1$.
Let $\cN_T$ be the set of all clauses $b$ of $\PHI^t$ such that  $|N(b)\cap T|\geq z|N(b)|$ and $0.1\theta k\leq|N(b)|\leq10\theta k$.
Then i.\ implies that
	$$\sum_{b\in\cN_T}|N(b)|\leq\frac{1.009}z|T|+10^{-6}\delta\theta n.$$
Furthermore, ii.\ yields
	\begin{eqnarray*}
	\sum_{b:|N(b)\cap T|\geq z|N(b)|}|N(b)|&\leq&\sum_{b:|N(b)|\not\in[0.1\theta k,10\theta k]}|N(b)|+\sum_{b\in\cN_T}|N(b)|\\
		&\leq&1.009|T|/z+2\cdot10^{-5}\delta\theta n,
	\end{eqnarray*}
as desired.
\qed\end{proof}

\subsubsection{Establishing Q4.}
We are going to work with the probability distribution $\PHIseq$ (sequence of $m$ independent clauses).
Let $\cM$ be the set of all indices $l\in\brk m$ such that the $l$th clause $\PHIseq(l)$
does not contain any of the variables $x_1,\ldots,x_{t}$ positively.
In this case, $\PHIseq(l)$
is still present in the decimated formula $\PHIseq^t$ (with all occurrences of $\neg x_1,\ldots,\neg x_t$ eliminated, of course).
For each $l\in\cM$ let $\cL(l)$ be the number of literals in $\PHIseq(l)$ whose underlying variable is in $V_t$.
We may assume without loss of generality that for any $l\in\cM$ the $\cL(l)$ `leftmost'
literals $\PHIseq(l,i)$, $1\leq i\leq \cL(l)$, are the ones with an underlying variable from $V_t$.

Let $T\subset V_t$.
Analyzing the operator $\Lambda_T$ directly is a little awkward.
Therefore, we will decompose $\Lambda_T$ into a sum of several operators that are easier to investigate.
For any $0.1\theta k\leq L\leq10\theta k$, $1\leq i<j\leq L$, $l\in\cM$,
and any distinct $x,y\in V_t$ we define
	$$m_{xy}(i,j,l,L)=\left\{\begin{array}{cl}
		1&\mbox{ if }\cL(l)=L\wedge[(\PHIseq\bc{l,i}=x\wedge \PHIseq\bc{l,j}=y)\\
			&\qquad\qquad\qquad\qquad\qquad\vee(\PHIseq\bc{l,i}=\neg x\wedge\PHIseq\bc{l,j}=\neg y)],\\
		-1&\mbox{ if }\cL(l)=L\wedge[(\PHIseq\bc{l,i}=x\wedge\PHIseq\bc{l,j}=\neg y)\\
			&\qquad\qquad\qquad\qquad\qquad\vee(\PHIseq\bc{l,i}=\neg x\wedge\PHIseq\bc{l,j}=y)],\\
		0&\mbox{ otherwise},
		\end{array}\right.$$
while we let $m_{xx}(i,j,l,L)=0$.
Moreover, for $x,y\in V_t$ we let
	$$m_{xy}(i,j,L)=\sum_{l\in\cM}m_{xy}(i,j,l,L).$$
For a variable $x\in V_t$ we let $\cN(x,T)$ be the set of all $l\in\cM$ such
that $0.1\theta k\leq\cL(l)\leq 10\theta k$ and the clause $\PHIseq(l)$ contains
at most one literal whose underlying variable is in $T\setminus x$.
Moreover, for $l\in\cM$ let $\cN(x,l)$ be the set of all variables $y\in V_t\setminus\cbc x$ that occur in clause $\PHIseq(l)$ (either positively or negatively).
We are going to  analyze the  operators
	$$\Lambda^{ijL}_T:\RR^{V_t}\rightarrow\RR^{V_t},\quad
		\Gamma=(\Gamma_y)_{y\in V_t}\mapsto
		\cbc{\sum_{l\in\cN(x,T)}\sum_{y\in\cN(x,l)}2^{-L}m_{xy}\bc{i,j,L}\Gamma_y}_{x\in V_t}.$$

\begin{lemma}\label{Cor_eqBern5}
For any 
$0.1\theta k\leq L\leq10\theta k$, $1\leq i<j\leq L$ and for any set $T\subset V_t$
we have 
	$$\pr\brk{\cutnorm{\Lambda^{ijL}_T}\leq\delta^5\theta n}\geq1-\exp(-\theta n).$$
\end{lemma}
\begin{proof}
The proof is based on Fact~\ref{Lemma_cutnorm}.
Fix two sets $A,B\subset V_t$.
For each $l\in\cM$ and any $x,y\in V_t$ the two $0/1$ random variables 
	$$\sum_{(x,y)\in A\times B}\max\cbc{m_{xy}(i,j,l,L),0},\quad\sum_{(x,y)\in A\times B}\max\cbc{-m_{xy}(i,j,l,L),0}$$
are identically distributed,
because the clause $\PHIseq(l)$ is chosen uniformly at random.
In effect, 
the two random variables
	\begin{eqnarray*}
	\mu(A,B)&=&\sum_{l\in\cM}\sum_{(x,y)\in A\times B}\vecone_{l\in\cN(x,T)}\max\cbc{m_{xy}(i,j,l),0},\\
	\nu(A,B)&=&\sum_{l\in\cM}\sum_{(x,y)\in A\times B}\vecone_{l\in\cN(x,T)}\max\cbc{-m_{xy}(i,j,l),0}
	\end{eqnarray*}
are identically distributed.
Furthermore,
both $\mu(A,B)$ and $\nu(A,B)$ are sums of independent Bernoulli variables,
because the clauses $(\PHIseq(l))_{l\in\brk m}$ are mutually independent.

We are need to estimate the mean $\Erw(\mu(A,B))=\Erw(\nu(A,B))$.
As each of the clauses $\PHIseq(l)$ is chosen uniformly,
for each $l\in\brk m$ we have
	$$\pr\brk{l\in\cM\wedge\cL(l)=L}=\bink{k}L\theta^L(1-\theta)^{k-L}2^{L-k}.$$
Therefore, 
	\begin{eqnarray}\nonumber
	\Erw\brk{\mu(A,B)+\nu(A,B)}&\leq&
		m\bink{k}L\theta^L(1-\theta)^{k-L}2^{L-k}\\
	&=&\frac{2^L\rho\theta n}L\bink{k-1}{L-1}\theta^{L-1}(1-\theta)^{k-L}\qquad\mbox{[as $m=2^k\rho/k$]}\nonumber\\
	&\leq&\frac{2^L\rho\theta n}L.
		\label{eqBern1}
	\end{eqnarray}
Hence, \Lem~\ref{Lemma_Chernoff} (the Chernoff bound) yields
	\begin{eqnarray}\nonumber
	\pr\brk{\abs{\mu(A,B)-\Erw(\mu(A,B))}>10\sqrt{2^L\rho/L}\cdot\theta n}&=&
		\pr\brk{\abs{\nu(A,B)-\Erw(\nu(A,B))}>10\sqrt{2^L\rho/L}\cdot\theta n}\\
		&\leq&16^{-\theta n}.\nonumber
	\end{eqnarray}
Hence, by the union bound 
	\begin{eqnarray}\nonumber
	\pr\brk{\exists A,B\subset V_t:\max\cbc{\abs{\mu(A,B)-\Erw(\mu(A,B))},\abs{\nu(A,B)-\Erw(\nu(A,B))}}>10\sqrt{2^L\rho/L}\cdot\theta n}\\
		&\hspace{-10cm}\leq&\hspace{-5cm}\;
		2\cdot4^{\theta n}\cdot 16^{-\theta n}\leq\exp(-\theta n).\label{eqBern2}\nonumber
	\end{eqnarray}
Thus, with probability $\geq1-\exp(-\theta n)$ we have
	\begin{eqnarray*}
	\scal{\Lambda_T^{ijL}\vecone_B}{\vecone_A}&=&2^{-L}(\mu(A,B)-\nu(A,B))\\
		&\leq&2^{-L}\bc{\abs{\mu(A,B)-\Erw\brk{\mu(A,B)}}+\abs{\nu(A,B)-\Erw\brk{\nu(A,B)}}}\\
		&\leq&\theta n\cdot20\sqrt{\frac{\rho}{L2^L}}\leq0.01\delta^5\theta n\qquad\mbox{[as $L\geq0.1k\theta$, $\theta k\geq\ln(\rho)/c^2$,
			 and $\delta=\exp(-ck\theta)$]}.
	\end{eqnarray*}
Finally, the assertion follows from Fact~\ref{Lemma_cutnorm}.
\qed\end{proof}

\begin{corollary}\label{Cor_matrixBound}
With probability at least $1-\exp(-0.1\theta n)$ the random formula $\PHIseq^t$ has the following property.
\begin{quote}
Let $T\subset V_t$ and let
	$$\Lambda_T'=\sum_{0.1\theta k\leq L\leq10\theta k}\sum_{j=1}^L\sum_{1\leq i<j}\Lambda_T^{ijL}.$$
Then $\cutnorm{\Lambda_T'}\leq\delta^{4.9}\theta n$.
\end{quote}
\end{corollary}
\begin{proof}
By \Lem~\ref{Cor_eqBern5} and the union bound, we have
	\begin{eqnarray*}
	\pr\brk{\exists T,i,j,L:\cutnorm{\Lambda^{ijL}_T}>\delta^5\theta n}&\leq&(10\theta k)^32^{\theta n}\cdot\exp(-\theta n)
		\leq\exp(-0.1\theta n).
	\end{eqnarray*}
Furthermore, if $\cutnorm{\Lambda^{ijL}_T}\leq\delta^5\theta n$ for all $i,j,L$, then by the triangle inequality
	$$\cutnorm{\Lambda_T'}\leq(10\theta k)^3\delta^5\theta n\leq\delta^{4.9}\theta n\qquad\mbox{[as $\delta=\exp(-ck\theta)]$},$$
as claimed.
\qed\end{proof}

To complete the proof of {\bf Q4},
we observe that 
for $(x,y)\in V_t\times V_t$ the $(x,y)$ entries of the matrices
$\Lambda_T$ and $\Lambda_T'$ 
 differ only if either $x$ or $y$ occurs in a redundant clause.
Consequently, {\bf Q0} ensures that
	$\cutnorm{\Lambda_T'-\Lambda_T}=o(n).$
Therefore, Fact~\ref{Fact_models} and \Cor~\ref{Cor_matrixBound} imply
$\PHI^t$ satisfies {\bf Q4} with probability at least $1-\exp(-11\Delta_t)$.


\begin{thebibliography}{99}


\bibitem{AchHandbook}
D.~Achlioptas:
Random satisfiability.
in: A.~Biere, M.~Heule, H.~van Maaren, T.~Walsh (eds.):
Handbook of Satisfiability, IOS Press (2009), 245--270.


\bibitem{AchBeameMolloy}
D.~Achlioptas, P.~Beame, M.~Molloy:
Exponential bounds for DPLL below the satisfiability threshold.
Proc.~15th SODA (2004) 139--140.

\bibitem{AchACO}
D.~Achlioptas, A.~Coja-Oghlan: 
Algorithmic barriers from phase transitions.
Proc.~49th FOCS (2008) 793--802.

\bibitem{nae}
D.~Achlioptas, C.~Moore:
Random $k$-SAT: two moments suffice to cross a sharp threshold.
SIAM Journal on Computing {\bf 36} (2006) 740--762.

\bibitem{ANP}
D.~Achlioptas, A.~Naor, Y.~Peres:
Rigorous location of phase transitions in hard optimization problems. 
Nature \textbf{435} (2005) 759--764.

\bibitem{yuval}
D.~Achlioptas, Y.~Peres:
The threshold for random $k$-SAT is $2^k \ln 2 - O(k)$.
Journal of the AMS \textbf{17} (2004) 947--973.

\bibitem{AchSorkin}
D.~Achlioptas, G.~Sorkin: Optimal myopic algorithms for random 3-SAT. 
Proc.\ 41st FOCS (2000) 590--600.


\bibitem{Lenka}
J.~Ardelius, L.~Zdeborova:
Exhaustive enumeration unveils clustering and freezing in random 3-SAT.
Phys.\ Rev.\ E {\bf78} (2008) 040101(R).

\bibitem{BSS}
M.~Bayati, D.~Shah, M.~Sharma:
Max-product for maximum weight matching: convergence, correctness, and LP duality.
IEEE Transaction on Information Theory {\bf 54} (2008) 1241--1251.

\bibitem{Bethe}
H.~Bethe: Statistical theory of superlattices.
	Proc.\ R.\ Soc.\ London A {\bf150} (1935) 552--558.


\bibitem{BMZ}
A.~Braunstein, M.~M\'ezard, R.~Zecchina:
Survey propagation: an algorithm for satisfiability.
Random Structures and Algorithms {\bf 27} (2005)  201--226.

\bibitem{PureLiteral}
A.~Broder, A.~Frieze, E.~Upfal:
On the satisfiability and maximum satisfiability of random 3-{CNF} formulas.
Proc.\ 4th SODA (1993) 322--330.

\bibitem{ChaoFranco2} M.-T.~Chao, J.~Franco:
Probabilistic analysis of a generalization of the unit-clause literal selection heuristic for the $k$-satisfiability problem.
Inform.\ Sci.\ \textbf{51} (1990) 289--314.

\bibitem{Cheeseman}
P.~Cheeseman, B.~Kanefsky, W.~Taylor:
Where the really hard problems are.
Proc.\ IJCAI (1991) 331--337.

\bibitem{mick} V.~Chv\'{a}tal, B.~Reed:
Mick gets some (the odds are on his side). 
Proc.~33th FOCS (1992) 620--627.

\bibitem{BetterAlg}
A.~Coja-Oghlan:
A better algorithm for random $k$-SAT.
SIAM J.\ Computing {\bf 39} (2010) 2823--2864.

\bibitem{BP3col}
A.~Coja-Oghlan, E.~Mossel, D.~Vilenchik:
 A spectral approach to analyzing Belief Propagation for 3-coloring.
 Combinatorics, Probability and Computing {\bf 18} (2009) 881--912.

\bibitem{Angelica}
A.~Coja-Oghlan, A.~Pachon-Pinzon:
The decimation process in random $k$-SAT.
SIAM J.\ Discrete Mathematics {\bf 26} (2012) 1471--1509.

\bibitem{Kosta}
A.~Coja-Oghlan, K.~Panagiotou:
Going after the $k$-SAT threshold.
Proc.\ 45th STOC (2013), to appear.


\bibitem{WP}
U.~Feige, E.~Mossel, D.~Vilenchik:
Complete convergence of message passing algorithms for some satisfiability problems.
Proc.\ 10th RANDOM (2006) 339--350.

\bibitem{FrSu}
A.~Frieze, S.~Suen:
Analysis of two simple heuristics on a random instance of $k$-SAT.
Journal of Algorithms \textbf{20} (1996) 312--355.

\bibitem{Gallager}
R.~G.~Gallager: Low-density parity check codes.
	MIT Press (1963).

\bibitem{GSW}
D.~Gamarnik, D.~Shah,Y.~Wei:
Belief Propagation for min-cost network flow: convergence \&\ correctness.
Operations Research {\bf 60} (2012) 410--428.


\bibitem{GomesSelman}
C.~P.\ Gomes, Bart Selman:
Satisfied with physics.
Science {\bf 297} (2002) 784-785.

\bibitem{HajiSorkin}
M.~Hajiaghayi, G.~Sorkin:
The satisfiability threshold of random 3-SAT is at least $3.52$.
IBM Research Report RC22942 (2003).

\bibitem{JLR}
S.~Janson, T.~{\L}uczak, A.~Ruci\'nski: Random Graphs, Wiley  2000.

\bibitem{KKL}
A.\ Kaporis, L.\ Kirousis, E.\ Lalas:
The probabilistic analysis of a greedy satisfiability algorithm.
Random Structures and Algorithms {\bf28} (2006) 444--480.

\bibitem{KirkpatrickSelman}
S.~Kirkpatrick, B.~Selman:
Critical behavior in the satisfiability of random Boolean expressions. Science {\bf264} (1994) 1297--1301.

\bibitem{KKKS}
L.~Kirousis, E.~Kranakis, D.~Krizanc, Y.~Stamatiou:
Approximating the unsatisfiability threshold of random formulas.
Random Structures Algorithms {\bf 12} (1998) 253--269.

\bibitem{Kroc}
L.~Kroc, A.~Sabharwal, B.~Selman:
Message-passing and local heuristics as decimation strategies for satisfiability.
Proc\ 24th SAC (2009) 1408--1414.

\bibitem{pnas}
F.~Krzakala, A.~Montanari, F.~Ricci-Tersenghi, G.~Semerjian, L.~Zdeborova:
Gibbs states and the set of solutions of random constraint satisfaction problems.
Proc.~National Academy of Sciences {\bf104} (2007) 10318--10323.

\bibitem{MitchellSelmanLevesque}
D.~Mitchell, B.~Selman, H.~Levesque:
Hard and easy distribution of SAT problems.
Proc.\ 10th AAAI (1992) 459--465. 



\bibitem{MMW}
E.~Maneva, E.~Mossel, M.~Wainwright: A new look at survey propagation and its generalizations. J.\ ACM {\bf 54} (2007)

\bibitem{Mertens}
S.~Mertens, M.~M\'ezard, R.~Zecchina:
Threshold values of random $K$-SAT from the cavity method.
Random Struct.\ Alg.\ {\bf28} (2006) 340--373.


\bibitem{MPZ}
M.~M\'ezard, G.~Parisi, R.~Zecchina:
Analytic and algorithmic solution of random satisfiability problems.
Science {\bf 297} (2002) 812--815.

\bibitem{MM}
A.~Montanari, M.~M\'ezard:
Information, physics and computation.
Oxford university press, 2009.

\bibitem{MU}
A.~Montanari, R.~Urbanke:
Modern coding theory: the statistical mechanics and computer science point of view.
Lecture notes, 2007.

\bibitem{Prasad}
A.~Montanari, R.~Restrepo, P.~Tetali:
Reconstruction and clustering in random constraint satisfaction problems.
Preprint (2009).

\bibitem{Allerton}
A.~Montanari, F.~Ricci-Tersenghi, G.~Semerjian:
Solving constraint satisfaction problems through Belief Propa\-gation-guided decimation.
Proc.\ 45th Allerton (2007).

\bibitem{MontanariShah}
A.~Montanari, D.~Shah:
Counting good truth assignments of random k-SAT formulae.
Proc.\ 18th SODA (2007) 1255--1264.

\bibitem{Pearl}
J.~Pearl:
Probabilistic reasoning in intelligent systems: networks of  plausible inference.
Morgan Kaufmann Publishers Inc., San Francisco, CA, USA, 1988.

\bibitem{RTS}
F.~Ricci-Tersenghi,  G.~Semerjian:
On the cavity method for decimated random constraint satisfaction problems
and the analysis of belief propagation guided decimation algorithms.
J.\ Stat.\ Mech.\ (2009) P09001.

\bibitem{RSU01}
T.~Richardson, A.~Shokrollahi, R.~Urbanke:
Design of capacity-approaching irregular low-density parity check codes.
IEEE Trans.\ Info.\ Theory {\bf47} (2001) 619--637.


\end{thebibliography}
\end{document}